\documentclass{cocv}
\usepackage{amsmath,amssymb,amsthm,mathtools}
\usepackage{booktabs}
\usepackage{aliascnt}
\usepackage{array}
\usepackage{graphicx}
\usepackage{placeins}
\usepackage{algorithm}
\usepackage{algpseudocode}
\usepackage{enumitem}
\usepackage{microtype}
\usepackage{hyperref}
\usepackage[nameinlink,capitalize,noabbrev]{cleveref}

\hypersetup{
	colorlinks=false,
	pdfborder={0 0 0},
	pdftitle={Gaussian-Restricted Barycenters for KL-Unbalanced Optimal Transport: Variational Theory and Fixed-Point Convergence},
	pdfauthor={Jiaping Yang and Yunxin Zhang},
	pdfkeywords={unbalanced optimal transport, Gaussian barycenter, calculus of variations, Bures-Wasserstein geometry, Fisher-Rao geometry}
}

\newtheorem{theorem}{Theorem}[section]
\newaliascnt{proposition}{theorem}
\newtheorem{proposition}[proposition]{Proposition}
\aliascntresetthe{proposition}
\newaliascnt{lemma}{theorem}
\newtheorem{lemma}[lemma]{Lemma}
\aliascntresetthe{lemma}
\newaliascnt{corollary}{theorem}
\newtheorem{corollary}[corollary]{Corollary}
\aliascntresetthe{corollary}
\algrenewcommand\algorithmicrequire{\textbf{Input:}}
\algrenewcommand\algorithmicensure{\textbf{Output:}}
\theoremstyle{remark}
\newaliascnt{remark}{theorem}
\newtheorem{remark}[remark]{Remark}
\aliascntresetthe{remark}

\newcommand{\R}{\mathbb{R}}
\newcommand{\Sym}{\mathbb{S}}
\newcommand{\Spp}{\mathbb{S}_{++}}
\newcommand{\N}{\mathcal{N}}
\newcommand{\KL}{\operatorname{KL}}
\newcommand{\tr}{\operatorname{tr}}
\newcommand{\BW}{\mathrm{BW}}
\newcommand{\FR}{\mathrm{FR}}
\newcommand{\GKU}{\mathsf{U}}
\newcommand{\Ashape}{\mathsf{A}}
\newcommand{\Fbar}{\mathsf{F}}
\newcommand{\Zfun}{\mathsf{Z}}
\newcommand{\rhoCher}{\rho_{\theta}}
\newcommand{\rBF}{\mathfrak r}
\newcommand{\norm}[1]{\left\lVert #1\right\rVert}

\newcommand{\AsymSlowRate}{0.947650}
\newcommand{\AsymFastRate}{0.518341}
\newcommand{\LargeGridDistA}{0.4236}
\newcommand{\LargeGridDistC}{0.0568}

\begin{document}
\title{Gaussian-Restricted Barycenters for KL-Unbalanced Optimal Transport: Variational Theory and Fixed-Point Convergence}
\author{Jiaping Yang$^{*,}$}
\address{School of Mathematical Sciences, Fudan University, Shanghai 200433, China.\\
${}^{*}$ Corresponding author: \href{mailto:jpyang22@m.fudan.edu.cn}{\nolinkurl{jpyang22@m.fudan.edu.cn}}}
\author{Yunxin Zhang}\sameaddress{1}
\runningtitle{Gaussian-Restricted KL-UOT Barycenters}
\runningauthors{J. Yang and Y. Zhang}
\begin{abstract} We study Gaussian-restricted barycenters for quadratic two-sided Kullback--Leibler unbalanced optimal transport with independent marginal penalties and no coupling entropy.  Exact profiling of the barycenter mass reduces the problem to a smooth Gaussian shape functional with endogenous Gibbs weights.  We establish global attainment, derive the stationary moment equations, and construct a reverse-KL majorization--minimization (MM) iteration whose full sequence converges from every nondegenerate Gaussian initialization to a stationary fixed point.  The diagonal second variation induces a parallel-sum tensor coupling the Bures--Wasserstein and Fisher--Rao metrics; its finite-mass extension admits a radial cone representation.  Under common penalty scaling, global minimizers converge to a Gaussian Wasserstein barycenter with effective weights; for sufficiently large penalties, the minimizer is unique and admits a first-order analytic expansion.  In the small-penalty regime, the distance of every global minimizer to the compact maximizer set of a weighted Chernoff affinity functional vanishes with respect to the mean--covariance parameter distance.  Numerical experiments illustrate MM descent, local contraction, and the two penalty limits.
 \end{abstract}
\maketitle
\markboth{J. YANG AND Y. ZHANG}{GAUSSIAN-RESTRICTED KL-UOT BARYCENTERS}
\section{Introduction}

Wasserstein barycenters provide a variational notion of averaging based on mass displacement~\cite{AguehCarlier2011}.  In the balanced setting, the inputs have the same total mass and their differences are reconciled entirely through transport.  For Gaussian measures, the problem retains a non-Euclidean covariance geometry governed by the Bures--Wasserstein metric~\cite{Gelbrich1990,Takatsu2011,BhatiaJainLim2019}; representative fixed-point and convergence results are given in~\cite{AlvarezEsteban2016,ChewiEtAl2020,AltschulerEtAl2021}.  Recent extensions treat marginal-constrained Gaussian and mixture barycenters and balanced mixture-level geometries~\cite{DaleryDussonEhrlacher2026,DussonEhrlacherNouaime2026}.

When the input measures have unequal total masses, prior normalization suppresses part of the information carried by the data, whereas enforcing a common mass can confound displacement with creation or loss.  Unbalanced optimal transport (UOT) addresses this limitation by replacing hard marginal constraints with divergence penalties.  Foundational formulations include~\cite{LieroMielkeSavare2018,ChizatDynamic2018,ChizatScaling2018,ChizatFoCM2018}; barycentric developments in Hellinger--Kantorovich, multimarginal, and Kantorovich--Rubinstein settings include~\cite{FrieseckeMatthesSchmitzer2021,ChungPhung2021,BonafiniMinevichSchmitzer2023,Buze2025,BeierEtAl2023,HeinemannKlattMunk2023}.  Discrete entropic barycenter algorithms and related unbalanced computational methods are treated in~\cite{CuturiDoucet2014,SejourneEtAl2022,SejourneEtAl2023}, while Gaussian--Hellinger barycenters have also been used in ensemble forecasting~\cite{DucSawada2024}.

We consider a static quadratic UOT model with separate KL penalties on the source and target marginals.  The two penalties allow different strengths of marginal relaxation on the input and barycenter sides, and their asymmetry makes the orientation of the barycenter problem relevant.  We impose no entropy penalty on the coupling, so that the effects studied below arise from marginal relaxation rather than entropic smoothing of the transport plan.

The inputs and the candidate barycenter are restricted to finite nondegenerate Gaussian measures.  This moment-level model retains total mass, location, covariance, and marginal relaxation while keeping the barycenter interpretable through a finite set of parameters.  Finite dimensionality, however, does not make the variational analysis automatic: the positive-definite covariance cone is open and noncompact, minimizing sequences may approach its singular boundary, and both the adjusted endpoint marginals and their effective weights depend on the candidate barycenter.  Global attainment and convergence therefore require explicit control of escape and covariance collapse.

The penalty scale also changes the qualitative averaging mechanism.  At large penalties, marginal deviations are expensive and the problem approaches Wasserstein displacement averaging; at small penalties, the adjusted marginals coalesce and the limiting criterion favors Chernoff affinity overlap.  The diagonal second variation provides a local geometric counterpart to this transition by connecting transport and Fisher-type responses.

\subsection{Related work and positioning}

Janati et al.~\cite{JanatiEtAl2020} study entropically regularized transport between Gaussian measures and show that the unbalanced problem retains a scaled-Gaussian structure; their associated barycenter equations therefore describe an entropically smoothed coupling problem.  In contrast, the present model puts no entropy penalty on the coupling.  Marginal relaxation first adjusts the endpoint Gaussians, after which they are joined by ordinary quadratic Wasserstein transport, with the optimal coupling supported on an affine graph.  The equal-penalty endpoint without coupling entropy also appears in~\cite{JanatiThesis2021}.

Our earlier work~\cite{YangZhang2026} gives the exact asymmetric two-sided endpoint for a pair of finite Gaussians and supplies the pairwise interface used below.  Nakashima, Ganguly, and Kashima~\cite{NakashimaEtAl2026} independently establish a finite-dimensional Gaussian reduction, global pairwise solvability, and density control.  At the barycenter level, Nguyen et al.~\cite{NguyenEtAl2026} consider centered Gaussian probabilities with fixed mass in a semi-unbalanced formulation.  The present formulation moves from these pairwise or semi-unbalanced settings to a many-input problem in which the barycenter mass is optimized, means are retained, both marginals carry independent penalties, and the normalized input weights depend on the unknown barycenter.  Together, these features create a coupled common-shape optimization rather than a collection of independent pairwise problems.

The many-input Gaussian restriction changes both the variational and algorithmic analysis.  Hellinger--Kantorovich and multimarginal theories~\cite{FrieseckeMatthesSchmitzer2021,ChungPhung2021,BonafiniMinevichSchmitzer2023,BeierEtAl2023} establish existence or variational formulations in ambient measure spaces, whereas the present problem requires compactness directly on the open mean--covariance parameter domain: control of mass and moments alone does not exclude a covariance eigenvalue tending to zero.  The balanced entropic barycenter method of~\cite{CuturiDoucet2014} and the unbalanced computational schemes developed or reviewed in~\cite{SejourneEtAl2022,SejourneEtAl2023} use finite-dimensional parameterizations adapted to their respective settings.  Here exact mass profiling instead produces endogenous Gibbs weights and stationary mean--covariance equations, while MM convergence requires an orbitwise spectral floor.  Balanced Gaussian fixed-point analyses~\cite{AlvarezEsteban2016,ChewiEtAl2020,AltschulerEtAl2021} provide an important point of comparison; the dependence of both the adjusted endpoint marginals and their normalized weights on the current iterate calls for separate descent and noncollapse estimates.

Geometrically, the Gaussian-preserving Hellinger--Kantorovich--Boltzmann flow of Liero et al.~\cite{LieroMielkeTseZhu2026} yields an additive Bures--Wasserstein--Fisher--Rao structure.  The tensor obtained below has a different origin: it is the Schur-complement, or parallel-sum, response after optimizing both KL-relaxed endpoint marginals.  This distinction explains why the two penalty parameters enter directionally and why exchanging them changes the MM linearization even though the diagonal metric is symmetric in the two slots.

\subsection{Problem and contributions}

Let \(\alpha_i=a_i\mu_i\), where \(\mu_i=\N(m_i,\Sigma_i)\) with \(\Sigma_i\in\Spp^d\), and let \(\lambda_i>0\) satisfy \(\sum_{i=1}^n\lambda_i=1\).  For fixed penalties \(\tau_0,\tau_1>0\), we consider
\[
\Fbar(\beta)=\sum_{i=1}^n\lambda_i\,\GKU_{\tau_0,\tau_1}(\alpha_i,\beta),
\qquad
\beta=b\N(m,\Sigma)\in\mathcal G_d^+.
\]
Because the pairwise cost is generally asymmetric when \(\tau_0\neq\tau_1\), a minimizer is called a \emph{right Gaussian KL-UOT barycenter}, with the barycenter in the second argument.  Exchanging \(\tau_0\) and \(\tau_1\) gives the corresponding left formulation; we work with the right formulation throughout.

Two reductions organize the analysis: exact profiling eliminates the candidate mass, while the adjusted endpoint marginals generate endogenous Gibbs weights and a reverse-KL Gaussian centroid relation.  The main contributions are as follows.
\begin{enumerate}[label=(\roman*)]
	\item \emph{Variational structure and geometry.}  We profile the mass exactly, exclude escape and covariance degeneracy to prove global attainment, and derive the stationary moment equations.  The diagonal second variation yields a BW--FR parallel-sum tensor with a radial cone representation, while the Gibbs weights quantify attenuation of remote inputs.
	\item \emph{MM convergence.}  The stationary equations define a reverse-KL majorization--minimization iteration.  An orbitwise spectral floor supports the Kurdyka--\L{}ojasiewicz analysis: from every nondegenerate initialization, the full trajectory has finite parameter length and converges to a stationary fixed point; near a nondegenerate local minimizer, the iteration converges linearly with an intrinsic contraction factor.
	\item \emph{Penalty limits.}  Penalty-uniform localization shows that global barycenters approach a Gaussian Wasserstein barycenter at large penalty and a weighted Chernoff-affinity maximizer set at small penalty.  In the large-penalty regime, we also prove eventual uniqueness and a first-order analytic expansion.
\end{enumerate}

The remainder of the paper is organized as follows. Section~2 establishes the pairwise Gaussian interface, and Section~3 profiles the mass, proves attainment, and derives the stationary equations.  Section~4 develops the local geometry, while Section~5 constructs the MM iteration and proves its global and local convergence properties.  Section~6 combines endpoint expansions with penalty-uniform localization to obtain the two global penalty limits.  Section~7 illustrates the descent, contraction, robustness, and asymptotic mechanisms numerically.

\section{Gaussian reduction for KL-UOT}

\subsection{Finite Gaussian measures and KL divergence}

Let \(\mathcal{M}_{+}(E)\) denote the cone of finite nonnegative Borel measures on a measurable space \(E\).  For finite measures \(\rho,\eta\), define
\begin{equation}
	\KL(\rho\mid\eta)
	=
	\begin{cases}
		\displaystyle
		\int \left(r\log r-r+1\right)\,d\eta,
		& \rho\ll\eta,\quad r=\dfrac{d\rho}{d\eta},\\[0.7em]
		+\infty,&\text{otherwise}.
	\end{cases}
\end{equation}
For probability measures this is the usual relative entropy.  If \(p,\mu\) are probabilities and \(M,a>0\), then
\begin{equation}
	\KL(Mp\mid a\mu)
	=
	M\KL(p\mid\mu)+M\log\frac{M}{a}-M+a.
	\label{eq:scaledKL}
\end{equation}

We write \(\Sym^d=\{H\in\R^{d\times d}:H^\top=H\}\) for the real symmetric matrices, with \(\Spp^d\) and \(\mathbb{S}_{+}^d\) denoting their positive-definite and positive-semidefinite cones.  Set
\[
\mathcal G_d^1
:=
\left\{\N(m,\Sigma):m\in\R^d,\ \Sigma\in\Spp^d\right\}
\]
for the nondegenerate Gaussian probability manifold and
\[
\mathcal G_d^+
:=
\left\{a\nu:a>0,\ \nu\in\mathcal G_d^1\right\}
\]
for the open cone of finite Gaussian measures.

We equip the mean--covariance parameter space with the Euclidean mean--Frobenius structure
\begin{equation}
	\left\langle (h,H),(h',H')\right\rangle_{\rm par}
	:=h^\top h'+\tr(HH'),
	\qquad
	\norm{(h,H)}_{\rm par}^2=\norm{h}^2+\norm{H}_F^2,
	\label{eq:paraminner}
\end{equation}
and the associated parameter distance
\begin{equation}
	d_{\rm par}\!\left(\N(m,\Sigma),\N(m',\Sigma')\right)
	:=\left(\norm{m-m'}^2+\norm{\Sigma-\Sigma'}_F^2\right)^{1/2}.
\end{equation}
Whenever a smooth scalar functional is expressed in the global mean--covariance coordinates, \(\nabla_{\rm par}\) denotes its Euclidean gradient with respect to \eqref{eq:paraminner}.

For \(\alpha,\beta\in\mathcal{M}_{+}(\R^d)\), define the quadratic-cost two-sided KL-UOT functional
\begin{equation}
	\GKU_{\tau_0,\tau_1}(\alpha,\beta)
	=\inf_{\gamma\in\mathcal{M}_{+}(\R^d\times\R^d)}\Bigl\{\int\norm{x-y}^2\,d\gamma+\tau_0\KL(\gamma_0\mid\alpha)+\tau_1\KL(\gamma_1\mid\beta)\Bigr\}.
	\label{eq:UOT}
\end{equation}
There is no entropy penalty on the coupling \(\gamma\).  For \(\tau_0\neq\tau_1\), \eqref{eq:UOT} need not be symmetric.  If \(\gamma^\top\) denotes the transposed coupling, symmetry of the quadratic ground cost implies \(\GKU_{\tau_0,\tau_1}(\alpha,\beta)=\GKU_{\tau_1,\tau_0}(\beta,\alpha)\).  Hence left and right barycenter formulations are interchanged by \(\tau_0\leftrightarrow\tau_1\).

Set
\begin{equation}
	s=\tau_0+\tau_1,
	\qquad
	\theta=\frac{\tau_0}{s}\in(0,1),
	\qquad
	1-\theta=\frac{\tau_1}{s}.
\end{equation}

Unless stated otherwise, throughout the barycenter analysis we assume
\[
\tau_0,\tau_1>0,\qquad n\ge1,\qquad
\lambda_i>0,\quad a_i>0,\quad m_i\in\R^d,\quad \Sigma_i\in\Spp^d
\quad(i=1,\ldots,n),\qquad \sum_{i=1}^n\lambda_i=1.
\]
These standing assumptions are not repeated in every theorem.

\subsection{Normalized Gaussian endpoint problem}

Let \(\mathcal P(\R^d)\) denote the Borel probability measures on \(\R^d\), and let \(\mathcal P_2(\R^d)\) be the subset with finite second moment.  For \(p,q\in\mathcal P(\R^d)\), we use the extended-value convention
\(
W_2^2(p,q):=\inf_{\pi\in\Pi(p,q)}\int_{\R^d\times\R^d}\lvert x-y\rvert^2\,d\pi(x,y)\in[0,\infty].
\)
For \(\mu,\nu\in\mathcal G_d^1\), define
\begin{equation}
	\Ashape_{\tau_0,\tau_1}(\mu,\nu)
	=
	\inf_{p,q\in\mathcal P(\R^d)}
	\left\{
	W_2^2(p,q)
	+\tau_0\KL(p\mid\mu)
	+\tau_1\KL(q\mid\nu)
	\right\}.
	\label{eq:Adef}
\end{equation}
Every pair with finite objective in \eqref{eq:Adef} automatically belongs to \(\mathcal P_2(\R^d)^2\).  Indeed, choose \(c>0\) such that \(\int e^{c\lvert x\rvert^2}\,d\mu(x)<\infty\).  Applying the entropy inequality to \(c(\lvert x\rvert^2\wedge R)\) gives
\[
	c\int(\lvert x\rvert^2\wedge R)\,dp
	\leq \KL(p\mid\mu)+\log\int e^{c\lvert x\rvert^2}\,d\mu,
\]
and monotone convergence yields \(p\in\mathcal P_2(\R^d)\); the same argument with \((q,\nu)\) gives \(q\in\mathcal P_2(\R^d)\).  We may therefore use the following Gaussian reduction.  If a probability measure \(p\in\mathcal P_2(\R^d)\) has finite relative entropy with respect to a nondegenerate Gaussian, let \(G[p]\) denote its moment-matched Gaussian.  For every nondegenerate Gaussian reference \(\mu\), the information-projection identity is
\begin{equation}
	\KL(p\mid\mu)
	=
	\KL(p\mid G[p])
	+
	\KL(G[p]\mid\mu).
\end{equation}
Indeed, \(\log(dG[p]/d\mu)\) is quadratic, hence has the same expectation under \(p\) and \(G[p]\) because these measures have the same first two moments.  Consequently \(\KL(G[p]\mid\mu)\le\KL(p\mid\mu)\), with equality precisely when \(p=G[p]\) whenever the entropies are finite.  Together with Gelbrich's moment bound, these identities establish
\[
	W_2^2(G[p],G[q])\le W_2^2(p,q),\qquad
	\KL(G[p]\mid\mu_0)\le\KL(p\mid\mu_0),\qquad
	\KL(G[q]\mid\mu_1)\le\KL(q\mid\mu_1).
\]
Thus the infimum in \eqref{eq:Adef} is unchanged when \((p,q)\) is restricted to Gaussian probabilities.

The finite-dimensional Gaussian endpoint admits the following closed form.  Appendix~\ref{app:pairwise} verifies the pairwise characterization used here; a fuller pairwise treatment, including a KL-dual certificate, appears in~\cite{YangZhang2026}.  The characterization also includes real-analytic dependence on the endpoint parameters.

Let
\[
\mu_0=\N(m_0,\Sigma_0),
\qquad
\mu_1=\N(m_1,\Sigma_1).
\]
Put
\begin{equation}
	r_j=\frac{2}{\tau_j},\qquad
	K_j=\Sigma_j^{-1},\qquad
	C_j=K_j+r_jI\quad (j=0,1),\qquad
	\delta=r_1-r_0.
\end{equation}
All square roots below are principal square roots.  Define
\begin{equation}
	S_*
	=
	\frac12\left[
	\delta I+
	\left(
	\delta^2I+4C_1^{1/2}C_0C_1^{1/2}
	\right)^{1/2}
	\right],
	\label{eq:Sstar}
\end{equation}
\begin{equation}
	L_*=C_1^{-1/2}S_*C_1^{-1/2},
	\qquad
	P_*=[K_0+r_0(I-L_*)]^{-1},
	\qquad
	Q_*=L_*P_*L_*.
	\label{eq:LPQ}
\end{equation}
For the means, define
\begin{equation}
	h_*=(I+r_0\Sigma_0+r_1\Sigma_1)^{-1}(m_0-m_1),
	\label{eq:hstar}
\end{equation}
\begin{equation}
	u_*=m_0-r_0\Sigma_0h_*,
	\qquad
	v_*=m_1+r_1\Sigma_1h_*.
	\label{eq:uvstar}
\end{equation}
Define
\begin{equation}
	p_*=\N(u_*,P_*),
	\qquad
	q_*=\N(v_*,Q_*),
	\label{eq:pqstar}
\end{equation}
and \(q_*=(T_*)_\#p_*\) under the affine map
\begin{equation}
	T_*(x)=v_*+L_*(x-u_*).
	\label{eq:Tstar}
\end{equation}
The covariance map satisfies the Riccati equation
\begin{equation}
	L_*C_1L_*-\delta L_*=C_0.
	\label{eq:Riccati}
\end{equation}

We abbreviate
\begin{equation}
	\Ashape(\mu_0,\mu_1)
	=
	W_2^2(p_*,q_*)
	+\tau_0\KL(p_*\mid\mu_0)
	+\tau_1\KL(q_*\mid\mu_1).
\end{equation}
Whenever the penalty parameters are fixed, we suppress them from the notation.

\begin{theorem}\label{thm:pairwiseinterface}
	Let \(\tau_0,\tau_1>0\) and \(\mu_0,\mu_1\in\mathcal G_d^1\).  The normalized problem \eqref{eq:Adef} has a unique minimizer \((p_*,q_*)\), given by \eqref{eq:Sstar}--\eqref{eq:pqstar}.  The matrix \(L_*\) is the unique positive-definite solution of the Riccati equation \eqref{eq:Riccati}, and the unique optimal Wasserstein coupling of \((p_*,q_*)\) is supported on the affine graph \eqref{eq:Tstar}.  Moreover,
	\[
	(\mu_0,\mu_1)\longmapsto
	L_*,P_*,Q_*,u_*,v_*,\Ashape(\mu_0,\mu_1)
	\]
	is real analytic on \(\mathcal G_d^1\times\mathcal G_d^1\).
\end{theorem}

\begin{proof}
	The Gaussian reduction is established above.  The remaining finite-dimensional claims are verified in Appendix~\ref{app:pairwise}.
\end{proof}

The finite-mass reduction used in the barycenter problem is given by the following scaling identity.

\begin{proposition}\label{prop:endpoint}
	Let \(\alpha=a\mu_0\) and \(\beta=b\mu_1\) with \(a,b>0\) and \(\mu_0,\mu_1\in\mathcal G_d^1\).  Then
	\begin{equation}
		\GKU_{\tau_0,\tau_1}(a\mu_0,b\mu_1)
		=
		\tau_0a+\tau_1b-sM_*,
		\label{eq:endpointvalue}
	\end{equation}
	where
	\begin{equation}
		M_*
		=
		a^\theta b^{1-\theta}
		\exp\left[-\frac{\Ashape(\mu_0,\mu_1)}{s}\right].
		\label{eq:Mstar}
	\end{equation}
	The unique optimal plan is \(M_*\pi_*\), where \(\pi_*=(\operatorname{Id},T_*)_\#p_*\).
\end{proposition}

\begin{proof}
	The zero plan has value \(\tau_0a+\tau_1b\).  It is not optimal: any normalized competitor with finite objective can be rescaled to a sufficiently small positive transported mass to produce a strictly smaller value.  Hence every optimizer has positive mass.  Write \(\gamma=M\pi\), with \(M>0\), and let \(\mathcal E_{\mu_0,\mu_1}(p,q)\) denote the bracketed normalized objective in \eqref{eq:Adef}.  After minimizing over \(\pi\) and applying the scaled-KL identity, the objective becomes
	\[
	M\mathcal E_{\mu_0,\mu_1}(p,q)
	+\tau_0\left(M\log\frac{M}{a}-M+a\right)
	+\tau_1\left(M\log\frac{M}{b}-M+b\right).
	\]
	For fixed \((p,q)\), strict convexity in \(M\) gives the unique minimizer
	\[
	M(p,q)=a^\theta b^{1-\theta}
	\exp\!\left[-\frac{\mathcal E_{\mu_0,\mu_1}(p,q)}{s}\right],
	\]
	and substitution reduces the objective to \(\tau_0a+\tau_1b-sM(p,q)\), which is strictly increasing in \(\mathcal E_{\mu_0,\mu_1}(p,q)\).  Minimization over \((p,q)\) therefore selects the pairwise optimizer \((p_*,q_*)\), which establishes \eqref{eq:endpointvalue}--\eqref{eq:Mstar} and the plan formula.
\end{proof}

\begin{lemma}\label{lem:meancomponent}
	For \(\mu_0=\N(m_0,\Sigma_0)\) and \(\mu_1=\N(m_1,\Sigma_1)\), the normalized endpoint functional decomposes as
	\[
	\Ashape(\mu_0,\mu_1)
	=
	\Ashape_{\rm mean}(m_0,m_1;\Sigma_0,\Sigma_1)
	+
	\Ashape_{\rm cov}(\Sigma_0,\Sigma_1),
	\]
	where
	\begin{equation}
		\Ashape_{\rm mean}
		=
		(m_0-m_1)^\top
		\left(I+r_0\Sigma_0+r_1\Sigma_1\right)^{-1}
		(m_0-m_1).
		\label{eq:Amean}
	\end{equation}
	Moreover \(\Ashape_{\rm cov}\geq0\), and \(\Ashape(\mu_0,\mu_1)=0\) if and only if \(\mu_0=\mu_1\).
\end{lemma}

\begin{proof}
	The Gaussian Wasserstein and KL formulas separate the mean variables from the covariance variables.  The mean problem is
	\[
	\min_{u,v}
	\left\{
	\norm{u-v}^2
	+\frac{\tau_0}{2}(u-m_0)^\top\Sigma_0^{-1}(u-m_0)
	+\frac{\tau_1}{2}(v-m_1)^\top\Sigma_1^{-1}(v-m_1)
	\right\}.
	\]
	Solving its first-order equations recovers \eqref{eq:hstar}--\eqref{eq:uvstar}.  Writing
	\(K=I+r_0\Sigma_0+r_1\Sigma_1\) and \(h=K^{-1}(m_0-m_1)\), substitution reduces the minimum to
	\[
	\norm{h}^2+r_0h^\top\Sigma_0h+r_1h^\top\Sigma_1h
	=h^\top Kh
	=(m_0-m_1)^\top K^{-1}(m_0-m_1),
	\]
	which is \eqref{eq:Amean}.  The covariance component is nonnegative because it is the minimum of nonnegative terms.  If \(\Ashape=0\), all three terms in \eqref{eq:Adef} vanish at the optimizer, so \(p=q=\mu_0=\mu_1\).  The converse is immediate.
\end{proof}

\subsection{Gaussian identities and KL sublevels}

The compactness and boundary arguments use the following standard Gaussian identities.  For \(\mu_0=\N(m_0,\Sigma_0)\) and \(\mu_1=\N(m_1,\Sigma_1)\), the squared \(2\)-Wasserstein distance is
\begin{equation}
	W_2^2(\mu_0,\mu_1)
	=
	\norm{m_0-m_1}^2
	+
	\tr\!\left(
	\Sigma_0+\Sigma_1
	-2(\Sigma_0^{1/2}\Sigma_1\Sigma_0^{1/2})^{1/2}
	\right).
\end{equation}
Its covariance part will be denoted by
\begin{equation}
	d_{\BW}^2(P,Q)
	:=
	\tr\!\left(P+Q-2(P^{1/2}QP^{1/2})^{1/2}\right),
	\qquad P,Q\in\Spp^d.
	\label{eq:Burescovdistance}
\end{equation}
The positive-definite optimal map from covariance \(P\) to covariance \(Q\) is
\begin{equation}
	L=P^{-1/2}(P^{1/2}QP^{1/2})^{1/2}P^{-1/2},
	\qquad Q=LPL.
	\label{eq:GaussianW2map}
\end{equation}
For fixed \(P\), the standard differential formula for the Bures--Wasserstein covariance distance reads
\begin{equation}
	\nabla_Q d_{\BW}^2(P,Q)=I-L^{-1},
\end{equation}
where \(L\) is the map in \eqref{eq:GaussianW2map}; see, e.g.,~\cite{BhatiaJainLim2019}.

For Gaussian relative entropy we use
\begin{equation}
	\KL\bigl(\N(u,P)\mid\N(m,\Sigma)\bigr)
	=\tfrac12\left[\tr(\Sigma^{-1}P)+\norm{\Sigma^{-1/2}(u-m)}^2-d+\log\tfrac{\lvert\Sigma\rvert}{\lvert P\rvert}\right].
	\label{eq:GaussianKL}
\end{equation}

Both the global-attainment argument and the small-penalty analysis use the following uniform Gaussian KL sublevel estimate.

\begin{lemma}\label{lem:KLsublevel}
	Fix \(R<\infty\), \(0<\underline\sigma\leq\overline\sigma<\infty\), and \(C<\infty\).  There exist constants
	\[
	R_C<\infty,
	\qquad
	0<\underline p_C\leq\overline p_C<\infty
	\]
	such that, whenever
	\[
	\norm m\leq R,
	\qquad
	\underline\sigma I\preceq\Sigma\preceq\overline\sigma I,
	\qquad
	\KL\bigl(\N(u,P)\mid\N(m,\Sigma)\bigr)\leq C,
	\]
	one has
	\[
	\norm u\leq R_C,
	\qquad
	\underline p_C I\preceq P\preceq\overline p_C I.
	\]
\end{lemma}

\begin{proof}
	Write \(B=\Sigma^{-1/2}P\Sigma^{-1/2}\).  By \eqref{eq:GaussianKL}, the relative-entropy identity reads
	\begin{equation}
		2\KL\bigl(\N(u,P)\mid\N(m,\Sigma)\bigr)
		=
		\norm{\Sigma^{-1/2}(u-m)}^2
		+\tr B-\log\det B-d.
		\label{eq:KLsublevelidentity}
	\end{equation}
	If \(\lambda_1(B),\ldots,\lambda_d(B)\) are the eigenvalues of \(B\), then
	\[
	\tr B-\log\det B-d
	=
	\sum_{j=1}^d
	\bigl(\lambda_j(B)-\log\lambda_j(B)-1\bigr).
	\]
	The scalar function \(f(t)=t-\log t-1\) is nonnegative and diverges as \(t\downarrow0\) or \(t\to\infty\).  Since \(\sum_j f(\lambda_j(B))\le 2C\), every eigenvalue belongs to the compact scalar level set \(\{t>0:f(t)\le 2C\}\).  Writing this set as \([\ell_C,u_C]\Subset(0,\infty)\) makes the spectral constants independent of the particular reference Gaussian in the stated spectral box.  The spectral bounds on \(\Sigma\) transfer this control to two-sided bounds for \(P\).  The first term in \eqref{eq:KLsublevelidentity} then bounds \(u\) uniformly because \(\norm m\le R\) and \(\Sigma\preceq\overline\sigma I\).  Thus the corresponding Gaussian KL sublevels are uniformly compact in mean--covariance coordinates.
\end{proof}

\section{Variational structure and stationarity}

Fix \(n\geq1\), weights \(\lambda_i>0\) with \(\sum_i\lambda_i=1\), and finite Gaussian inputs
\begin{equation}
	\alpha_i=a_i\mu_i,
	\qquad
	\mu_i=\N(m_i,\Sigma_i),
	\qquad
	a_i>0.
	\label{eq:inputs}
\end{equation}
For \(\beta=b\nu\in\mathcal G_d^+\), define
\begin{equation}
	\Fbar(b,\nu)
	=
	\sum_{i=1}^n\lambda_i
	\GKU_{\tau_0,\tau_1}(a_i\mu_i,b\nu).
	\label{eq:Fbar}
\end{equation}
The functional \eqref{eq:Fbar} is the objective for the right barycenter problem; its admissible class is the open cone of finite nondegenerate Gaussian measures \(\mathcal G_d^+\).

For each candidate \(\nu\in\mathcal G_d^1\), set
\begin{equation}
	A_i(\nu)=\Ashape(\mu_i,\nu),
	\qquad
	z_i(\nu)=a_i^\theta e^{-A_i(\nu)/s},
	\qquad
	\Zfun(\nu)=\sum_{i=1}^n\lambda_i z_i(\nu),
	\label{eq:zZ}
\end{equation}
and let \(\bar a:=\sum_{i=1}^n\lambda_i a_i\).

Table~\ref{tab:notation} summarizes the endpoint notation.
\begin{table}[h!]
	\centering
	\small
	\caption{Notation for the Gaussian barycenter problem.}
	\label{tab:notation}
	\begin{tabular}{@{}ll@{}}
		\toprule
		Symbol & Meaning \\
		\midrule
		\(\mu_i=\N(m_i,\Sigma_i)\) & normalized input Gaussian \\
		\(\nu=\N(m,\Sigma)\) & normalized candidate barycenter Gaussian \\
		\(p_i=\N(u_i,P_i)\) & adjusted source-side endpoint marginal \\
		\(q_i=\N(v_i,Q_i)\) & adjusted target marginal (barycenter side) \\
		\(A_i(\nu)\) & normalized pairwise endpoint value \(\Ashape(\mu_i,\nu)\) \\
		\(z_i,\Zfun\) & unnormalized and aggregated Gibbs factors in \eqref{eq:zZ} \\
		\(b(\nu)\) & profiled barycenter mass \\
		\(\omega_i(\nu)\) & normalized Gibbs weights introduced in \eqref{eq:gibbsweights} \\
		\bottomrule
	\end{tabular}
\end{table}

\subsection{Mass profiling and Gibbs representation}

\begin{theorem}\label{thm:massprofile}
	For every fixed \(\nu\in\mathcal G_d^1\), the map \(b\mapsto\Fbar(b,\nu)\) is strictly convex on \((0,\infty)\) and has the unique minimizer
	\begin{equation}
		b(\nu)=\Zfun(\nu)^{1/\theta}.
		\label{eq:bprofile}
	\end{equation}
	At this minimizing value, the profiled objective is
	\begin{equation}
		\Phi(\nu):=\inf_{b>0}\Fbar(b,\nu)=\tau_0\left(\bar a-\Zfun(\nu)^{1/\theta}\right).
		\label{eq:profiledF}
	\end{equation}
	The Gaussian KL-UOT barycenter functional is therefore minimized exactly at maximizers of \(\Zfun\):
	\begin{equation}
		\operatorname*{argmin}_{b>0,\,\nu\in\mathcal G_d^1}\Fbar(b,\nu)
		=
		\left\{
		\left(\Zfun(\nu)^{1/\theta},\nu\right):
		\nu\in\operatorname*{argmax}_{\mathcal G_d^1}\Zfun
		\right\},
	\end{equation}
\end{theorem}

\begin{proof}
	By \Cref{prop:endpoint}, the objective takes the form
	\[
	\Fbar(b,\nu)
	=
	\tau_0\bar a+\tau_1b
	-sb^{1-\theta}\Zfun(\nu).
	\]
	Using \(s(1-\theta)=\tau_1\), the first two derivatives with respect to \(b\) are
	\[
	\partial_b\Fbar(b,\nu)
	=\tau_1\left(1-\Zfun(\nu)b^{-\theta}\right),
	\qquad
	\partial_b^2\Fbar(b,\nu)
	=\tau_1\theta\Zfun(\nu)b^{-\theta-1}>0.
	\]
	Strict convexity then gives \eqref{eq:bprofile} as the unique minimizer.  At this point \(\Zfun(\nu)=b(\nu)^\theta\), so
	\[
	\Fbar(b(\nu),\nu)
	=
	\tau_0\bar a+\tau_1b(\nu)-sb(\nu)
	=
	\tau_0(\bar a-b(\nu)).
	\]
\end{proof}

Exact mass profiling therefore reduces every local or global variational question to the normalized shape variable \(\nu\), with the mass recovered uniquely from \eqref{eq:bprofile}.  The resulting log-sum-exp functional admits a useful Gibbs variational representation.  Set
\begin{equation}
	Z_0=\sum_{i=1}^n\lambda_i a_i^\theta,
	\qquad
	\pi_i=\frac{\lambda_i a_i^\theta}{Z_0}.
\end{equation}
Then \(\pi=(\pi_i)\) belongs to the probability simplex
\(\Delta_n:=\{\omega=(\omega_1,\ldots,\omega_n)\in[0,1]^n:\sum_{i=1}^n\omega_i=1\}\).

\begin{theorem}[Gibbs representation]
	For every \(\nu\in\mathcal G_d^1\),
	\begin{equation}
		-s\log\frac{\Zfun(\nu)}{Z_0}
		=
		\min_{\omega\in\Delta_n}
		\left\{
		\sum_{i=1}^n\omega_iA_i(\nu)
		+s\KL(\omega\mid\pi)
		\right\},
	\end{equation}
	where
	\[
	\KL(\omega\mid\pi)=\sum_i\omega_i\log\frac{\omega_i}{\pi_i},
	\qquad 0\log(0/\pi_i):=0.
	\]
	The minimizer is unique and equals
	\begin{equation}
		\omega_i(\nu)
		=
		\frac{\lambda_i a_i^\theta e^{-A_i(\nu)/s}}
		{\sum_{j=1}^n\lambda_j a_j^\theta e^{-A_j(\nu)/s}}.
		\label{eq:gibbsweights}
	\end{equation}
\end{theorem}

\begin{proof}
	For fixed \(\nu\), consider
	\[
	G(\omega)
	=\sum_i\omega_iA_i+s\sum_i\omega_i\log\frac{\omega_i}{\pi_i}.
	\]
	Strict convexity makes the minimizer unique.  Introducing a multiplier \(c\) for \(\sum_i\omega_i=1\), the first-order equations are
	\[
	A_i+s\left(\log\frac{\omega_i}{\pi_i}+1\right)+c=0,
	\]
	so \(\omega_i\propto\pi_i e^{-A_i/s}\).  Normalization fixes the weights in \eqref{eq:gibbsweights}, while substitution evaluates the minimum as
	\[
	\min_{\omega\in\Delta_n}G(\omega)
	=-s\log\sum_i\pi_i e^{-A_i/s}
	=-s\log\frac{\Zfun(\nu)}{Z_0}.
	\]
\end{proof}

The Gibbs weights factor into the external weights, mass factors \(a_i^\theta\), and shape compatibilities \(e^{-A_i(\nu)/s}\).  Thus they respond to mass imbalance and geometric mismatch; the latter factor drives the remote-input decay in \Cref{prop:redescending}.

Fix a candidate \(\nu=\N(m,\Sigma)\).  For the pair \((\mu_i,\nu)\), let \(p_i=\N(u_i,P_i)\) and \(q_i=\N(v_i,Q_i)\) be the unique adjusted Gaussian marginals in \eqref{eq:Adef}.  By \Cref{thm:pairwiseinterface}, all these quantities depend smoothly on \((m,\Sigma)\).

\begin{lemma}\label{lem:envelope}
	For \(A_i(m,\Sigma)=\Ashape(\mu_i,\N(m,\Sigma))\),
	\begin{equation}
		\nabla_m A_i
		=
		\tau_1\Sigma^{-1}(m-v_i),
		\label{eq:gradmA}
	\end{equation}
	and, with the Frobenius pairing on symmetric matrices,
	\begin{equation}
		\nabla_\Sigma A_i
		=
		\frac{\tau_1}{2}
		\left[
		\Sigma^{-1}
		-\Sigma^{-1}R_i\Sigma^{-1}
		\right],
		\label{eq:gradSigmaA}
	\end{equation}
	where \(R_i:=Q_i+(v_i-m)(v_i-m)^\top\).
\end{lemma}

\begin{proof}
	By \Cref{thm:pairwiseinterface}, the unique optimizer depends smoothly on \((m,\Sigma)\).  The envelope argument cancels its implicit derivatives, leaving only \(\tau_1\KL(q_i\mid\N(m,\Sigma))\).  Differentiating \eqref{eq:GaussianKL} with respect to \(m\) and \(\Sigma\) gives the gradients in \eqref{eq:gradmA}--\eqref{eq:gradSigmaA}.
\end{proof}

\subsection{Global attainment}

Loss of compactness in Gaussian mean--covariance coordinates can occur through mean escape, divergence of the largest covariance eigenvalue, or collapse of the smallest one.  Endpoint sublevels control the first two mechanisms; the boundary requires a separate argument.

\begin{lemma}\label{lem:uppercompact}
	Fix \(\mu=\N(m_0,\Sigma_0)\in\mathcal G_d^1\).  For every \(C<\infty\), there exist constants \(R_C,\Lambda_C<\infty\) such that
	\[
	\Ashape(\mu,\N(m,\Sigma))\le C
	\]
	implies
	\begin{equation}
		\norm{m}\le R_C,
		\qquad
		\Sigma\preceq \Lambda_C I.
	\end{equation}
\end{lemma}

\begin{proof}
	Let \(p=\N(u,P)\) and \(q=\N(v,Q)\) be the adjusted Gaussian marginals for \((\mu,\nu)\), where \(\nu=\N(m,\Sigma)\).  Nonnegativity of the three terms in \eqref{eq:Adef} bounds them individually:
	\begin{equation}
		W_2^2(p,q)\le C,
		\qquad
		\KL(p\mid\mu)\le C/\tau_0,
		\qquad
		\KL(q\mid\nu)\le C/\tau_1.
		\label{eq:sublevelthreebounds}
	\end{equation}
	The fixed-reference Gaussian KL sublevel estimate (see \Cref{lem:KLsublevel}) implies
	\[
	\norm u\le C_1,
	\qquad
	\underline p I\preceq P\preceq \overline p I
	\]
	for positive constants depending only on \((C,\mu,\tau_0)\).  The mean part of \(W_2^2(p,q)\) bounds \(v\) as well.  For the covariance part, the trace Cauchy--Schwarz inequality
	\[
	\tr\!\left(P^{1/2}QP^{1/2}\right)^{1/2}
	\le \sqrt{\tr P\,\tr Q}
	\]
	and the Bures formula imply
	\[
	d_{\BW}^2(P,Q)
	\ge
	\bigl(\sqrt{\tr P}-\sqrt{\tr Q}\bigr)^2.
	\]
	The Bures bound controls \(\tr Q\) uniformly; positivity then implies \(Q\preceq(\tr Q)I\).
	
	Set \(B:=\Sigma^{-1/2}Q\Sigma^{-1/2}\).
	By the Gaussian KL formula,
	\begin{equation}
		2\KL(q\mid\nu)
		=
		\norm{\Sigma^{-1/2}(v-m)}^2
		+\tr B-\log\det B-d.
		\label{eq:relativeBbound}
	\end{equation}
	Because \(t-\log t-1\to+\infty\) as \(t\downarrow0\) or \(t\uparrow\infty\), \eqref{eq:sublevelthreebounds} confines the spectrum of \(B\) to a fixed compact subinterval of \((0,\infty)\).  In particular, there is \(\beta_->0\) such that \(Q\succeq \beta_-\Sigma\).
	The uniform upper bound on \(Q\), together with \(Q\succeq\beta_-\Sigma\), forces \(\Sigma\preceq\Lambda_C I\).  The first term in \eqref{eq:relativeBbound}, combined with this spectral bound and the bound on \(v\), also controls \(m\).
\end{proof}

\begin{lemma}\label{lem:boundaryresponse}
	Fix \(\mu\in\mathcal G_d^1\) and \(C<\infty\).  There exist \(t_0,c_0>0\), depending only on \((\mu,C,\tau_0,\tau_1)\), with the following property.  Suppose
	\[
	\nu=\N(m,\Sigma)\in\mathcal G_d^1,
	\qquad
	\Ashape(\mu,\nu)\le C,
	\]
	and let \(e\) be a unit eigenvector of \(\Sigma\) with \(\Sigma e=te\) for some \(t\in(0,t_0)\).  If \(Q\) is the covariance of the adjusted target marginal, then
	\begin{equation}
		e^\top Qe\ge t\bigl(1+c_0\sqrt t\bigr).
		\label{eq:boundaryresponse}
	\end{equation}
\end{lemma}

\begin{proof}
	Let \(p=\N(u,P)\) and \(q=\N(v,Q)\) be the adjusted endpoint marginals, and let \(L\) be the positive-definite Wasserstein map from \(P\) to \(Q\), so \(Q=LPL\).  The endpoint bound implies \(\KL(p\mid\mu)\le C/\tau_0\).  Applying \Cref{lem:KLsublevel} to this KL sublevel bounds the adjusted source covariance away from zero: there is \(\underline p>0\), depending only on \((\mu,C,\tau_0)\), such that
	\begin{equation}
	P\succeq \underline p I.
		\label{eq:Plowerinward}
	\end{equation}
	
	The covariance first-order condition for the adjusted target marginal is
	\begin{equation}
		I-L^{-1}
		+\frac{\tau_1}{2}\bigl(\Sigma^{-1}-Q^{-1}\bigr)=0.
		\label{eq:Qfirstorder}
	\end{equation}
	Here \(\nabla_Qd_{\BW}^2(P,Q)=I-L^{-1}\), while the covariance derivative of \(\tau_1\KL(q\mid\nu)\) is \(\tfrac{\tau_1}{2}(\Sigma^{-1}-Q^{-1})\).
	
	Put \(q_e=e^\top Qe\).  Combining \(Q=LPL\) with \eqref{eq:Plowerinward} bounds \(q_e\) as
	\[
	q_e\ge \underline p\,\norm{Le}^2
	\ge \underline p\,(e^\top Le)^2.
	\]
	For every positive-definite matrix \(A\), Cauchy--Schwarz implies
	\[
	(e^\top Ae)(e^\top A^{-1}e)\ge1.
	\]
	With \(A=L\) and \(A=Q\), the same inequality becomes
	\begin{equation}
		e^\top L^{-1}e
		\ge \sqrt{\frac{\underline p}{q_e}},
		\qquad
		e^\top Q^{-1}e\ge \frac1{q_e}.
		\label{eq:inverseCSbounds}
	\end{equation}
	Evaluating \eqref{eq:Qfirstorder} along \(e\) and inserting \eqref{eq:inverseCSbounds} reduces the stationarity condition to
	\begin{equation}
		0
		\le
		1-\sqrt{\frac{\underline p}{q_e}}
		+\frac{\tau_1}{2}\left(\frac1t-\frac1{q_e}\right).
		\label{eq:qebasic}
	\end{equation}
	Set \(c_0=\sqrt{\underline p}/(2\tau_1)\).  If \(q_e\le t(1+c_0\sqrt t)\), then, for \(c_0\sqrt t\le1\), the right-hand side of \eqref{eq:qebasic} is at most
	\[
	1-\frac{\sqrt{\underline p}}{\sqrt{2t}}
	+\frac{\tau_1c_0}{2\sqrt t}
	=
	1-\frac{\sqrt{\underline p}}{\sqrt t}
	\left(\frac1{\sqrt2}-\frac14\right).
	\]
	For sufficiently small \(t\), this upper bound is negative, contradicting \eqref{eq:qebasic}.  After decreasing \(t_0\) if necessary, we obtain \eqref{eq:boundaryresponse}.
\end{proof}

\begin{lemma}\label{lem:inward}
	Under the hypotheses of \Cref{lem:boundaryresponse}, there exist \(t_0,c>0\), depending only on \((\mu,C,\tau_0,\tau_1)\), such that
	\begin{equation}
		D_\Sigma\Ashape(\mu,\nu)[ee^\top]
		\le -\frac{c}{\sqrt t}
		\qquad (0<t<t_0).
		\label{eq:inwardderivative}
	\end{equation}
\end{lemma}

\begin{proof}
	By \Cref{thm:pairwiseinterface}, the optimizer varies differentiably with the target covariance.  The optimizer first-order conditions cancel the implicit derivative terms, leaving only the explicit covariance dependence of the target KL term.  Thus, for every symmetric perturbation \(H\),
	\[
	D_\Sigma\Ashape(\mu,\nu)[H]
	=
	\frac{\tau_1}{2}\tr\!\left[
	\left(\Sigma^{-1}-\Sigma^{-1}R\Sigma^{-1}\right)H
	\right],
	\qquad
	R=Q+(v-m)(v-m)^\top.
	\]
	With \(H=ee^\top\), \(\Sigma e=te\), and \(R\succeq Q\),
	\[
	D_\Sigma\Ashape(\mu,\nu)[ee^\top]
	\le
	\frac{\tau_1}{2}
	\left(\frac1t-\frac{e^\top Qe}{t^2}\right).
	\]
	The boundary response \eqref{eq:boundaryresponse} turns this bound into
	\[
	D_\Sigma\Ashape(\mu,\nu)[ee^\top]
	\le -\frac{\tau_1c_0}{2\sqrt t},
	\]
	giving \eqref{eq:inwardderivative} with \(c=\tau_1c_0/2\).
\end{proof}

On every fixed box with bounded mean and upper covariance spectrum, the competitor bound \(A_i(\nu)\le W_2^2(\mu_i,\nu)\) places all endpoints in common sublevels.  Since the input family is finite, the constants in \Cref{lem:boundaryresponse,lem:inward} can therefore be chosen uniformly in \(i\), as used in \Cref{lem:profiledinward}.

\begin{lemma}[Inward ascent]\label{lem:profiledinward}
	Fix \(R,\Lambda<\infty\).  There exist constants \(t_0,c>0\), depending only on the fixed inputs, the penalties, \(R\), and \(\Lambda\), such that the following holds.  If
	\[
	\nu=\N(m,\Sigma),
	\qquad
	\norm m\le R,
	\qquad
	\Sigma\preceq\Lambda I,
	\]
	and \(e\) is a unit eigenvector satisfying \(\Sigma e=te\) with \(0<t<t_0\), then
	\begin{equation}
		D_\Sigma\log\Zfun(\nu)[ee^\top]
		\ge \frac{c}{\sqrt t}.
		\label{eq:profiledinward}
	\end{equation}
\end{lemma}

\begin{proof}
	For each input \(\mu_i\), testing \eqref{eq:Adef} with \(p=\mu_i\) and \(q=\nu\) shows
	\[
	A_i(\nu)\le W_2^2(\mu_i,\nu).
	\]
	The right-hand side is uniformly bounded on the box \(\norm m\le R\), \(\Sigma\preceq\Lambda I\), even without a lower spectral bound.  By \Cref{lem:inward}, applied to the finite input family and these common endpoint sublevels, there are constants \(t_0>0\) and \(c_i>0\) such that
	\[
	D_\Sigma A_i(\nu)[ee^\top]
	\le -\frac{c_i}{\sqrt t}
	\qquad (i=1,\ldots,n).
	\]
	Writing
	\[
	\omega_i(\nu)
	=
	\frac{\lambda_i a_i^\theta e^{-A_i(\nu)/s}}{\Zfun(\nu)},
	\qquad
	\sum_i\omega_i(\nu)=1,
	\]
	the directional derivative of \eqref{eq:zZ} along \(ee^\top\) is
	\[
	D_\Sigma\log\Zfun(\nu)[ee^\top]
	=
	-\frac1s\sum_i\omega_i(\nu)
	D_\Sigma A_i(\nu)[ee^\top].
	\]
	Taking \(c=s^{-1}\min_i c_i>0\) proves \eqref{eq:profiledinward}.  Thus, once the mean and upper spectrum are controlled, the lemma provides a uniform ascent direction away from the positive-semidefinite boundary.
\end{proof}

\begin{theorem}[Global attainment and compactness]\label{thm:globalexistence}
	Assume \(\tau_0,\tau_1>0\), and let positive weights \(\lambda_i\) satisfy \(\sum_i\lambda_i=1\).  Let the inputs in \eqref{eq:inputs} satisfy \(a_i>0\) and \(\Sigma_i\in\Spp^d\).  Then the profiled shape objective \(\Phi\) attains its global minimum on the full Gaussian manifold \(\mathcal G_d^1\).  Equivalently, \(\Zfun\) attains its global maximum.  In particular, the infimum
	\[
	\inf_{b>0,\,\nu\in\mathcal G_d^1}\Fbar(b,\nu)
	\]
	is attained.  Its minimizer set is
	\begin{equation}
		\operatorname*{argmin}_{b>0,\,\nu\in\mathcal G_d^1}\Fbar(b,\nu)
		=
		\left\{
		\bigl(\Zfun(\nu)^{1/\theta},\nu\bigr):
		\nu\in\operatorname*{argmax}_{\mathcal G_d^1}\Zfun
		\right\}.
		\label{eq:globalbaryset}
	\end{equation}
	The maximizing shape set \(\operatorname*{argmax}_{\mathcal G_d^1}\Zfun\) is also compact in mean--covariance coordinates; consequently, the finite Gaussian barycenter set in \eqref{eq:globalbaryset} is compact.
\end{theorem}

\begin{proof}
	\emph{Step 1: upper coercivity.}
	Set \(Z^*:=\sup_{\nu\in\mathcal G_d^1}\Zfun(\nu)\).  Since every term in \(\Zfun\) is positive and bounded above by \(\lambda_i a_i^\theta\), we have \(0<Z^*<\infty\).  Let \(\nu_k=\N(m_k,\Sigma_k)\) be a maximizing sequence.  For all sufficiently large \(k\), \(\Zfun(\nu_k)\ge Z^*/2\).  For each such \(k\), at least one index \(i(k)\) must satisfy
	\[
	\lambda_{i(k)}a_{i(k)}^\theta
	e^{-A_{i(k)}(\nu_k)/s}
	\ge
	\frac{Z^*}{2n}.
	\]
	For every index that can occur in this way, whenever \(i(k)=i\) the preceding inequality implies the finite bound
	\[
	A_i(\nu_k)
	\le
	C_i^*:=s\log\!\left(\frac{2n\lambda_i a_i^\theta}{Z^*}\right).
	\]
	The index \(i(k)\) may depend on \(k\), but it takes values in a finite set.  Applying \Cref{lem:uppercompact} to each relevant pair \((\mu_i,C_i^*)\) and taking the largest resulting constants confines the sequence to the uniform bounds
	\begin{equation}
		\norm{m_k}\le R,
		\qquad
		\Sigma_k\preceq \Lambda I
		\label{eq:maxsequpperbounds}
	\end{equation}
	for all sufficiently large \(k\).
	
	\emph{Step 2: exclusion of covariance collapse.}
	Suppose, after passing to a subsequence, that
	\[
	t_k:=\lambda_{\min}(\Sigma_k)\downarrow0.
	\]
	Choose a unit eigenvector \(e_k\) with \(\Sigma_ke_k=t_ke_k\), and for \(t\ge t_k\) set
	\[
	\Sigma_k(t)=\Sigma_k+(t-t_k)e_ke_k^\top,
	\qquad
	\nu_k(t)=\N(m_k,\Sigma_k(t)).
	\]
	Along this rank-one path, \(e_k\) remains an eigenvector with eigenvalue \(t\); it need not remain the minimum-eigenvalue direction after an eigenvalue crossing.  Set \(\Lambda_+:=\max\{\Lambda,1\}\).  The path satisfies \(\norm{m_k}\le R\) and \(\Sigma_k(t)\preceq\Lambda_+I\) for \(t\in[t_k,1]\).  Apply \Cref{lem:profiledinward} on this fixed upper box and choose \(\varepsilon:=\frac12\min\{1,t_0\}\), where \(t_0\) is the corresponding boundary threshold.  For all sufficiently large \(k\), \(t_k<\varepsilon\), and therefore
	\[
	\frac{d}{dt}\log\Zfun(\nu_k(t))
	\ge \frac{c}{\sqrt t}
	\qquad
	(t_k\le t\le\varepsilon).
	\]
	Integrating over the rank-one segment establishes
	\begin{equation}
		\log\frac{\Zfun(\nu_k(\varepsilon))}{\Zfun(\nu_k)}
		\ge
		2c\bigl(\sqrt\varepsilon-\sqrt{t_k}\bigr).
		\label{eq:Zprofiledintegrated}
	\end{equation}
	The right-hand side is eventually bounded below by a positive constant.  Since \(\Zfun(\nu_k)\to Z^*\), \eqref{eq:Zprofiledintegrated} would then force \(\Zfun(\nu_k(\varepsilon))>Z^*\) for all large \(k\), contradicting the definition of \(Z^*\).  Thus the chosen maximizing sequence stays uniformly away from covariance collapse.  After absorbing finitely many initial terms, there is a sequence-dependent constant \(\underline\sigma_{\nu_\bullet}>0\) such that
	\[
	\lambda_{\min}(\Sigma_k)\ge\underline\sigma_{\nu_\bullet}>0
	\qquad(k\ge0).
	\]
	
	\emph{Step 3: attainment.}
	Combined with \eqref{eq:maxsequpperbounds}, this lower spectral bound makes the chosen maximizing sequence precompact in mean--covariance coordinates.  The same argument applies to every maximizing sequence; hence each maximizing sequence is precompact.  A convergent subsequence and continuity of \(\Zfun\) then produce a maximizer \(\nu_*\in\mathcal G_d^1\).  The characterization \eqref{eq:globalbaryset} follows from \Cref{thm:massprofile}.
	
	\emph{Step 4: compactness of the optimizer set.}
	The lower spectral bound is uniform over the entire maximizing set.  Otherwise one could choose maximizers \(\widehat\nu_k=\N(\widehat m_k,\widehat\Sigma_k)\) with \(\lambda_{\min}(\widehat\Sigma_k)\le k^{-1}\).  For these maximizers, \(\Zfun(\widehat\nu_k)=Z^*\), while the bounds from Step~1 place them in a common box with bounded mean and upper covariance spectrum.  Applying the profiled inward-ascent estimate of \Cref{lem:profiledinward} on that box would produce a value larger than \(Z^*\), again a contradiction.  Consequently, common constants \(R,\underline\sigma,\Lambda\) place the whole maximizing set in the compact parameter set
	\[
	\mathcal K_*=
	\left\{\N(m,\Sigma):\norm m\le R,\ \underline\sigma I\preceq\Sigma\preceq\Lambda I\right\}.
	\]
	The maximizing set can therefore be written as
	\[
	\operatorname*{argmax}_{\mathcal G_d^1}\Zfun
	=
	\Zfun^{-1}(\{Z^*\})\cap\mathcal K_*,
	\]
	and continuity of \(\Zfun\) makes this level set closed in the compact set \(\mathcal K_*\).  Hence the maximizing set is compact.  The continuous mass profile \(\nu\mapsto\Zfun(\nu)^{1/\theta}\) maps it to a compact finite Gaussian barycenter set.
\end{proof}

\begin{corollary}\label{cor:nearmaxcompact}
	Let
	\[
	Z^*=\sup_{\nu\in\mathcal G_d^1}\Zfun(\nu),
	\qquad
	\Phi_*:=\min_{\nu\in\mathcal G_d^1}\Phi(\nu)
	=\tau_0\left(\bar a-(Z^*)^{1/\theta}\right).
	\]
	Then:
	\begin{enumerate}[label=(\alph*)]
		\item there exist \(\delta>0\), \(R<\infty\), and \(0<\underline\sigma\leq\Lambda<\infty\), depending only on the fixed inputs and penalties, such that
		\begin{equation}
			\Zfun(\N(m,\Sigma))\ge Z^*-\delta
			\quad\Longrightarrow\quad
			\norm m\le R,
			\qquad
			\underline\sigma I\preceq\Sigma\preceq\Lambda I;
		\end{equation}
		\item for
		\begin{equation}
			\varepsilon_*
			=\tau_0\left[(Z^*)^{1/\theta}-(Z^*-\delta)^{1/\theta}\right]>0,
			\label{eq:nearmin-epsilon}
		\end{equation}
		the sublevel set
		\begin{equation}
			\left\{\nu\in\mathcal G_d^1:
			\Phi(\nu)\le \Phi_*+\varepsilon_*\right\}
		\end{equation}
		is compact in mean--covariance coordinates.
	\end{enumerate}
\end{corollary}

\begin{proof}
	As in Step~1 of the proof of \Cref{thm:globalexistence}, the superlevel set \(\{\Zfun\ge Z^*/2\}\) lies in a common box \(\norm m\le R\), \(\Sigma\preceq\Lambda I\).  Apply \Cref{lem:profiledinward} on the enlarged box with \(\Lambda_+=\max\{\Lambda,1\}\), let \(t_0,c\) be the resulting constants, and set
	\[
	\varepsilon=\frac12\min\{1,t_0\},
	\qquad
	t_\star=\frac{\varepsilon}{4},
	\qquad
	\eta=2c\bigl(\sqrt\varepsilon-\sqrt{t_\star}\bigr)>0.
	\]
	If \(\lambda_{\min}(\Sigma)<t_\star\), the rank-one path used in Step~2 of \Cref{thm:globalexistence} satisfies
	\[
	\Zfun(\nu(\varepsilon))\ge e^\eta\Zfun(\nu).
	\]
	Choose \(\delta=\frac{Z^*}{2}(1-e^{-\eta})\).  Then every \(\nu\) with \(\Zfun(\nu)\ge Z^*-\delta\) lies in the preceding superlevel box; if its smallest covariance eigenvalue were below \(t_\star\), then
	\[
	\Zfun(\nu(\varepsilon))
	\ge e^\eta(Z^*-\delta)
	=\frac{e^\eta+1}{2}Z^*>Z^*.
	\]
	This contradicts the definition of \(Z^*\).  Hence (a) holds with \(\underline\sigma=t_\star\).
	
	For (b), the exact profile formula and \eqref{eq:nearmin-epsilon} show that \(\Phi(\nu)\le\Phi_*+\varepsilon_*\) implies \(\Zfun(\nu)\ge Z^*-\delta\).  Part (a) places the sublevel in one compact Gaussian parameter box, and continuity of \(\Phi\) makes it closed there.
\end{proof}

Together, \Cref{thm:globalexistence,cor:nearmaxcompact} establish compactness for minimizers and near-minimizers on the open manifold \(\R^d\times\Spp^d\); compactness of MM trajectories is proved separately in Section~\ref{sec:algorithm}.

\begin{corollary}\label{cor:massbound}
	For every \(\nu\in\mathcal G_d^1\),
	\begin{equation}
		0<b(\nu)
		\leq
		\left(\sum_{i=1}^n\lambda_i a_i^\theta\right)^{1/\theta}
		\leq
		\bar a.
		\label{eq:massbound}
	\end{equation}
	The last inequality is strict unless all \(a_i\) are equal.  Equality \(b(\nu)=\bar a\) holds if and only if all finite Gaussian inputs \(a_i\mu_i\) are identical and \(\nu\) equals their common normalized shape.
\end{corollary}

\begin{proof}
	Since \(A_i(\nu)\geq0\) and \(x\mapsto x^\theta\) is strictly concave on \((0,\infty)\), Jensen's inequality gives
	\[
	\Zfun(\nu)
	\leq \sum_i\lambda_i a_i^\theta
	\leq
	\left(\sum_i\lambda_i a_i\right)^\theta,
	\]
	where equality in the second inequality occurs precisely when all \(a_i\) are equal.  Taking the \(1/\theta\) power and using \eqref{eq:bprofile} proves \eqref{eq:massbound}.  If \(b(\nu)=\bar a\), both displayed inequalities must be equalities.  The equality conditions force all masses to coincide and \(A_i(\nu)=0\) for every \(i\); \Cref{lem:meancomponent} then gives \(\mu_i=\nu\) for all \(i\).  The reverse implication is immediate.
\end{proof}

As a special case, if \(\mu_i=\mu\) for all \(i\), then \Cref{lem:meancomponent} shows that the unique maximizing shape is \(\nu=\mu\), and the profiled barycenter mass is
\[
b_*
=
\left(\sum_{i=1}^n\lambda_i a_i^\theta\right)^{1/\theta}.
\]
Indeed, \(A_i(\nu)=\Ashape(\mu,\nu)\geq0\) with equality only at \(\nu=\mu\), while the mass formula follows from \eqref{eq:bprofile}.

\subsection{Stationarity and moment equations}

With global attainment established, differentiating the profiled shape functional converts the Gibbs representation into explicit moment equations.

\begin{theorem}[Stationarity and fixed points]\label{thm:fixedpoint}
	For \(\nu=\N(m,\Sigma)\in\mathcal G_d^1\), define
	\begin{equation}
		\omega_i(\nu)
		=
		\frac{\lambda_i a_i^\theta e^{-A_i(\nu)/s}}
		{\sum_j\lambda_j a_j^\theta e^{-A_j(\nu)/s}},
	\end{equation}
	and let \(q_i(\nu)=\N(v_i,Q_i)\) be the adjusted target Gaussian for \((\mu_i,\nu)\).  Then \(\nu\) is a stationary point of \(\Phi\) if and only if
	\begin{equation}
		m=\sum_{i=1}^n\omega_i(\nu)v_i,
		\label{eq:fixedm}
	\end{equation}
	\begin{equation}
		\Sigma
		=
		\sum_{i=1}^n\omega_i(\nu)
		\left[
		Q_i+(v_i-m)(v_i-m)^\top
		\right].
		\label{eq:fixedSigma}
	\end{equation}
	Equivalently, the fixed points of the moment map in \Cref{sec:algorithm} are exactly the stationary points of the profiled functional.  At every such point, with
	\begin{equation}
		b_*=\Zfun(\nu)^{1/\theta},
		\qquad
		M_i^*
		=
		a_i^\theta b_*^{1-\theta}e^{-A_i(\nu)/s},
	\end{equation}
	the pairwise transported masses satisfy
	\begin{equation}
		b_*=\sum_{i=1}^n\lambda_iM_i^*,
		\qquad
		\omega_i(\nu)=\frac{\lambda_iM_i^*}{b_*}.
		\label{eq:massbalance}
	\end{equation}
\end{theorem}

\begin{proof}
	Since
	\[
	\nabla \Zfun(\nu)
	=-\frac{\Zfun(\nu)}{s}\sum_i\omega_i(\nu)\nabla A_i(\nu),
	\]
	and \(\Phi=\tau_0(\bar a-\Zfun^{1/\theta})\), stationarity of \(\Phi\) is equivalent to
	\[
	\sum_i\omega_i(\nu)\nabla A_i(\nu)=0.
	\]
	Using \Cref{lem:envelope}, the mean and covariance components of this identity are respectively
	\[
	\tau_1\Sigma^{-1}\left(m-\sum_i\omega_i(\nu)v_i\right)=0
	\]
	and
	\[
	\frac{\tau_1}{2}
	\left[
	\Sigma^{-1}
	-\Sigma^{-1}\left(\sum_i\omega_i(\nu)R_i\right)\Sigma^{-1}
	\right]=0,
	\qquad
	R_i=Q_i+(v_i-m)(v_i-m)^\top.
	\]
	Because \(\Sigma\succ0\), these identities are equivalent to \eqref{eq:fixedm}--\eqref{eq:fixedSigma}, proving both directions at once.  Finally, the transported masses satisfy the balance identity
	\[
	\sum_i\lambda_iM_i^*
	=b_*^{1-\theta}\Zfun(\nu)
	=b_*,
	\]
	which is exactly \eqref{eq:massbalance}.
\end{proof}

Equivalently, \(\nu\) is the unique Gaussian matching the first two moments of the mixture \(\sum_i\omega_i(\nu)q_i(\nu)\); unlike the Gaussian Wasserstein barycenter, this mixture involves adjusted marginals, endogenous Gibbs weights, and their between-mean covariance.

\subsection{Remote-input attenuation}

The endogenous weights decay quantitatively when one input mean moves far from a compact candidate region.

\begin{proposition}[Remote-input attenuation]\label{prop:redescending}
	Fix a compact candidate set
	\begin{equation}
		\mathcal K
		=
		\left\{\N(m,\Sigma):\norm{m}\le R,\ \underline\sigma I\preceq\Sigma\preceq\overline\sigma I\right\}
		\subset\mathcal G_d^1.
	\end{equation}
	Assume that among the non-outlying inputs there is an anchor index \(j\neq o\) such that
	\[
	\lambda_j a_j^\theta\ge c_*>0,
	\qquad
	\mu_j\in\mathcal K_{\rm anc},
	\]
	where \(c_*>0\) and the compact Gaussian set \(\mathcal K_{\rm anc}\Subset\mathcal G_d^1\) are fixed.  Let the outlying input have mass \(a_o>0\), covariance \(\Sigma_o\in\Spp^d\) satisfying \(\Sigma_o\preceq\overline\sigma_o I\), and mean \(m_o\).  Then there exist constants \(c,C>0\), depending only on \(\mathcal K\), \(\mathcal K_{\rm anc}\), \(c_*\), \(\tau_0,\tau_1\), \(\lambda_o\), and \(\overline\sigma_o\), but independent of \(m_o\), \(a_o\), the choice of admissible anchor, and all remaining inputs, such that
	\begin{equation}
		\sup_{\nu\in\mathcal K}\omega_o(\nu)
		\le C a_o^\theta e^{-c\norm{m_o}^2},
		\label{eq:redescendingweight}
	\end{equation}
	and
	\begin{equation}
		\sup_{\nu\in\mathcal K}\omega_o(\nu)\norm{v_o(\nu)-m}
		\le C a_o^\theta(1+\norm{m_o})e^{-c\norm{m_o}^2}.
		\label{eq:redescendinginfluence}
	\end{equation}
	For the stationary covariance contribution \(R_o(\nu):=Q_o(\nu)+(v_o(\nu)-m)(v_o(\nu)-m)^\top\), one also has
	\begin{equation}
		\sup_{\nu\in\mathcal K}\omega_o(\nu)\norm{R_o(\nu)}_F
		\le C a_o^\theta(1+\norm{m_o}^2)e^{-c\norm{m_o}^2}.
		\label{eq:redescendingsecondmoment}
	\end{equation}
\end{proposition}

\begin{proof}
	\Cref{lem:meancomponent} gives the mean lower bound
	\[
	A_o(\nu)\ge(m_o-m)^\top\left(I+r_0\Sigma_o+r_1\Sigma\right)^{-1}(m_o-m).
	\]
	On the prescribed spectral sets, the inverse matrix obeys the uniform bound
	\[
	\left(I+r_0\Sigma_o+r_1\Sigma\right)^{-1}\succeq c_0I,
	\qquad c_0=\frac{1}{1+r_0\overline\sigma_o+r_1\overline\sigma}.
	\]
	Since \(\norm m\le R\), there are constants \(c_1,C_1>0\), independent of \(m_o\), such that
	\[
	A_o(\nu)\ge c_1\norm{m_o}^2-C_1
	\]
	uniformly over \(\nu\in\mathcal K\).  The numerator in \eqref{eq:gibbsweights} is therefore bounded by \(C_2a_o^\theta e^{-c_1\norm{m_o}^2/s}\).
	
	For the denominator, continuity of \(A\) on the compact product \(\mathcal K_{\rm anc}\times\mathcal K\) ensures
	\[
	C_{\rm anc}:=\sup_{\mu\in\mathcal K_{\rm anc},\,\nu\in\mathcal K}A(\mu,\nu)<\infty.
	\]
	The admissible anchor term is bounded below by \(c_*e^{-C_{\rm anc}/s}\), uniformly over the anchor family.  This proves \eqref{eq:redescendingweight}.
	
	The mean formula \eqref{eq:uvstar} reads
	\[
	v_o-m=r_1\Sigma\left(I+r_0\Sigma_o+r_1\Sigma\right)^{-1}(m_o-m),
	\]
	whose norm is at most \(C_3(1+\norm{m_o})\) on \(\mathcal K\).  Combining this with \eqref{eq:redescendingweight} proves \eqref{eq:redescendinginfluence}.
	
	It remains to control \(Q_o\).  Writing \(A_{o,{\rm cov}}\) for the covariance component in \Cref{lem:meancomponent}, this subproblem is independent of the means.  Testing the covariance subproblem with the unchanged covariances bounds
	\[
	A_{o,{\rm cov}}(\Sigma_o,\Sigma)
	\le d_{\BW}^2(\Sigma_o,\Sigma)
	\le \tr\Sigma_o+\tr\Sigma
	\le d(\overline\sigma_o+\overline\sigma).
	\]
	Hence the target covariance KL term at the optimum is uniformly bounded:
	\[
	\frac{\tau_1}{2}\left[\tr(\Sigma^{-1}Q_o)-\log\det(\Sigma^{-1}Q_o)-d\right]\le C_{\rm cov}.
	\]
	The coercivity of \(t-\log t-1\), together with the two-sided spectral bounds on \(\Sigma\), implies \(Q_o\preceq C_QI\) uniformly on \(\mathcal K\).  Therefore
	\[
	\norm{R_o(\nu)}_F\le C_4(1+\norm{m_o}^2),
	\]
	and \eqref{eq:redescendingsecondmoment} follows from \eqref{eq:redescendingweight}.
\end{proof}

Because \(m\) is bounded on \(\mathcal K\), the outlier weight and its first two moment contributions vanish whenever
\[
a_o^\theta(1+\norm{m_o}^2)e^{-c\norm{m_o}^2}\longrightarrow0,
\]
in particular for bounded \(a_o\).  Related KL-relaxed robust OT barycenters are discussed in~\cite{LeEtAl2021}.

\section{Local BW--FR geometry}

The diagonal second variation yields a BW--FR parallel-sum tensor.  Under \(\tau_j=\kappa\bar\tau_j\), it approaches the Bures--Wasserstein shape form as \(\kappa\to\infty\) and, after the natural small-penalty rescaling, the Fisher--Rao form as \(\kappa\downarrow0\), matching the global regimes of \Cref{sec:limits}.  Here, \emph{quadratic coefficient} means the coefficient of \(\varepsilon^2\); unless stated otherwise, \(o(\varepsilon^2)\) is local to the fixed base point and tangent directions.

\subsection{Bures--Wasserstein and Fisher--Rao forms}

We use the standard Gaussian Bures--Wasserstein and Fisher--Rao tensors; see, e.g.,~\cite{Takatsu2011,BhatiaJainLim2019} for the former and~\cite{CalvoOller1990,CalvoOller1991} for the multivariate-Gaussian Fisher geometry.  For \(\mu=\N(m,\Sigma)\), write a tangent vector as \(\xi=(\dot m,H)\in\R^d\times\Sym^d\) and set \(\mathcal L_\Sigma(X)=\Sigma X+X\Sigma\).  Then
\begin{equation}
	g_{\BW,\mu}(\xi,\xi)
	=
	\norm{\dot m}^2
	+\frac12\tr\left(H\mathcal L_\Sigma^{-1}H\right),
\end{equation}
and
\begin{equation}
	g_{\FR,\mu}(\xi,\xi)
	=
	\dot m^\top\Sigma^{-1}\dot m
	+\frac12\tr(\Sigma^{-1}H\Sigma^{-1}H).
	\label{eq:FRmetric}
\end{equation}
Their diagonal second-order expansions are
\begin{equation}
	W_2^2(\mu,\mu_\varepsilon)
	=
	\varepsilon^2g_{\BW,\mu}(\xi,\xi)+o(\varepsilon^2),
	\qquad
	\KL(\mu_\varepsilon\mid\mu)
	=
	\frac{\varepsilon^2}{2}g_{\FR,\mu}(\xi,\xi)+o(\varepsilon^2),
	\label{eq:KLlocal}
\end{equation}
for any smooth Gaussian curve with tangent \(\xi\) at \(\varepsilon=0\).  More generally, two smooth Gaussian curves through the same base point with tangent vectors \(x\) and \(y\) satisfy the joint expansions
\begin{equation}
	\begin{aligned}
		W_2^2(\mu_\varepsilon^x,\mu_\varepsilon^y)
		&=\varepsilon^2g_{\BW,\mu}(x-y,x-y)+o(\varepsilon^2),\\
		\KL(\mu_\varepsilon^x\mid\mu_\varepsilon^y)
		&=\frac{\varepsilon^2}{2}g_{\FR,\mu}(x-y,x-y)+o(\varepsilon^2).
	\end{aligned}
	\label{eq:twocurvelocal}
\end{equation}
For two positive-definite quadratic forms \(G,H\) on a finite-dimensional vector space, define their \emph{parallel sum} by the quadratic variational identity
\begin{equation}
	(G:H)(\xi,\xi)
	=
	\inf_{\eta}
	\left\{
	G(\eta,\eta)+H(\xi-\eta,\xi-\eta)
	\right\}.
	\label{eq:parallelsum}
\end{equation}
In matrix coordinates, if \(G,H\) are represented by positive-definite matrices, then
\[
G:H=(G^{-1}+H^{-1})^{-1}.
\]

Set \(\rBF:=\tau_0\tau_1/[2(\tau_0+\tau_1)]\).  We refer to \(g_{\BW}:\rBF g_{\FR}\) as the \emph{BW--FR parallel-sum metric}.

\subsection{Local endpoint expansion}

The exact Gaussian endpoint optimizer yields the following second-order envelope.

\begin{lemma}\label{lem:secondorderenvelope}
	Let \(\mu_\varepsilon\) be a smooth Gaussian curve through \(\mu\) with tangent \(\xi\) at \(\varepsilon=0\).  Denote by \((p_\varepsilon,q_\varepsilon)\) the unique Gaussian optimizer in \eqref{eq:Adef} for \((\mu,\mu_\varepsilon)\).  Then
	\begin{equation}
		\Ashape(\mu,\mu_\varepsilon)
		=
		\varepsilon^2
		\min_{x,y\in T_\mu\mathcal G_d^1}
		\mathcal Q_\xi(x,y)
		+o(\varepsilon^2),
		\label{eq:secondorderenvelope}
	\end{equation}
	where
	\begin{equation}
		\mathcal Q_\xi(x,y)
		=
		g_{\BW,\mu}(x-y,x-y)
		+\frac{\tau_0}{2}g_{\FR,\mu}(x,x)
		+\frac{\tau_1}{2}g_{\FR,\mu}(y-\xi,y-\xi).
	\end{equation}
	If \(\dot p_0,\dot q_0\in T_\mu\mathcal G_d^1\) are the tangent velocities of the optimizing curves \(p_\varepsilon\) and \(q_\varepsilon\), then \((\dot p_0,\dot q_0)\) is the unique minimizer of \(\mathcal Q_\xi\).
\end{lemma}

\begin{proof}
	Choose a smooth local chart \(\chi\) centered at \(\mu\) and identify coordinate velocities with tangent vectors through \(D\chi^{-1}(0)\).  Then \(\chi(\mu_\varepsilon)=\varepsilon\xi+O(\varepsilon^2)\).  Let \(F(\varepsilon,z_p,z_q)\) be the Gaussian endpoint objective in these coordinates.  At \((0,0,0)\) its unique optimizer is \(p=q=\mu\), the first \(\varepsilon\)-derivative vanishes, and the Hessian in \((z_p,z_q)\) is twice
	\[
	g_{\BW,\mu}(x-y,x-y)
	+\frac{\tau_0}{2}g_{\FR,\mu}(x,x)
	+\frac{\tau_1}{2}g_{\FR,\mu}(y,y),
	\]
	which is positive definite.  Applying the finite-dimensional implicit-function theorem to the stationarity equations in \((z_p,z_q)\) therefore gives a unique smooth stationary branch near \((0,0,0)\), with \((z_p,z_q)=O(\varepsilon)\).  By \Cref{thm:pairwiseinterface}, however, the exact endpoint optimizer is unique and depends real analytically on the two Gaussian endpoints; at \(\varepsilon=0\) it equals \((p,q)=(\mu,\mu)\).  Hence, for sufficiently small \(\varepsilon\), the exact optimizer lies in the implicit-function neighborhood, satisfies the stationarity equations, and must coincide with this local stationary branch.  A locally uniform second-order Taylor expansion, combined with the two-curve expansions \eqref{eq:twocurvelocal}, takes the form
	\[
	F(\varepsilon,\varepsilon x,\varepsilon y)
	=
	\varepsilon^2\mathcal Q_\xi(x,y)+o(\varepsilon^2).
	\]
	Since \(\mathcal Q_\xi\) is strictly convex, evaluating this expansion along the exact optimizer branch establishes \eqref{eq:secondorderenvelope}; its rescaled velocities converge to the unique minimizer \((\dot p_0,\dot q_0)\).  Minimizing over these endpoint tangent variables gives the optimized quadratic coefficient in the next lemma.
\end{proof}

\begin{lemma}\label{lem:localA}
	Let \(\mu_\varepsilon\) be a smooth Gaussian curve with \(\mu_0=\mu\) and tangent \(\xi\).  Then
	\begin{equation}
		\Ashape(\mu,\mu_\varepsilon)
		=
		\varepsilon^2
		\left(g_{\BW,\mu}:\rBF g_{\FR,\mu}\right)(\xi,\xi)
		+o(\varepsilon^2).
	\end{equation}
\end{lemma}

\begin{proof}
	By \Cref{lem:secondorderenvelope}, the second-order coefficient is the minimum
	\begin{equation}
		\inf_{x,y}
		\left\{
		g_{\BW}(x-y,x-y)
		+\frac{\tau_0}{2}g_{\FR}(x,x)
		+\frac{\tau_1}{2}g_{\FR}(y-\xi,y-\xi)
		\right\}.
		\label{eq:localquadraticxy}
	\end{equation}
	For fixed \(z=x-y\), minimizing the two Fisher terms over \(x\) has the solution
	\[
	x=\frac{\tau_1}{\tau_0+\tau_1}(z+\xi),
	\qquad
	y-\xi=-\frac{\tau_0}{\tau_0+\tau_1}(z+\xi),
	\]
	and the minimum of the Fisher part equals
	\[
	\frac{\tau_0\tau_1}{2(\tau_0+\tau_1)}
	g_{\FR}(z+\xi,z+\xi)
	=\rBF g_{\FR}(z+\xi,z+\xi).
	\]
	Replacing \(z\) by \(-\eta\), the remaining minimum is exactly
	\[
	\inf_{\eta}
	\left\{
	g_{\BW}(\eta,\eta)
	+\rBF g_{\FR}(\xi-\eta,\xi-\eta)
	\right\},
	\]
	which equals \((g_{\BW}:\rBF g_{\FR})(\xi,\xi)\) by \eqref{eq:parallelsum}.
\end{proof}

\subsection{Finite-mass expansion}

Let
\[
\alpha=a\mu,
\qquad
\alpha_\varepsilon
=(a+\varepsilon\dot a)\mu_\varepsilon,
\]
where \(\mu_\varepsilon\) has tangent \(\xi\) at \(\mu\).  For \(\lvert\varepsilon\rvert\) sufficiently small, \(a+\varepsilon\dot a>0\), so \(\alpha_\varepsilon\in\mathcal G_d^+\).

\begin{theorem}[BW--FR parallel-sum local quadratic form]
	As \(\varepsilon\to0\),
	\begin{equation}
		\GKU_{\tau_0,\tau_1}(\alpha,\alpha_\varepsilon)
		=
		\varepsilon^2
		g_{\alpha}^{\rm U}
		\big((\dot a,\xi),(\dot a,\xi)\big)
		+o(\varepsilon^2),
	\end{equation}
	where
	\begin{equation}
		g_{\alpha}^{\rm U}\big((\dot a,\xi),(\dot a,\xi)\big)
		=\frac{\rBF}{a}\dot a^2+a\left(g_{\BW,\mu}:\rBF g_{\FR,\mu}\right)(\xi,\xi).
		\label{eq:Umetric}
	\end{equation}
	This positive-definite form defines a smooth Riemannian tensor on \(\mathcal G_d^+\).
\end{theorem}

\begin{proof}
	Let
	\[
	A_\varepsilon=\Ashape(\mu,\mu_\varepsilon)
	=
	\varepsilon^2Q(\xi)+o(\varepsilon^2),
	\qquad
	Q=g_{\BW}:\rBF g_{\FR},
	\]
	by \Cref{lem:localA}.  Set
	\(x=\varepsilon\dot a/a\), so that
	\(a+\varepsilon\dot a=a(1+x)\).  The finite-mass endpoint formula in \Cref{prop:endpoint} then becomes
	\begin{align*}
		\GKU_{\tau_0,\tau_1}(\alpha,\alpha_\varepsilon)
		&=
		\tau_0a+\tau_1a(1+x)
		-sa(1+x)^{1-\theta}e^{-A_\varepsilon/s}.
	\end{align*}
	Using
	\[
	(1+x)^{1-\theta}
	=
	1+(1-\theta)x-\frac{\theta(1-\theta)}{2}x^2+o(x^2),
	\qquad
	e^{-A_\varepsilon/s}=1-\frac{A_\varepsilon}{s}+o(\varepsilon^2),
	\]
	the constant and first-order terms cancel because \(s(1-\theta)=\tau_1\).  The second-order mass contribution is
	\[
	sa\frac{\theta(1-\theta)}{2}x^2
	=
	\frac{\tau_0\tau_1}{2s}\frac{\varepsilon^2\dot a^2}{a}
	=
	\varepsilon^2\frac{\rBF}{a}\dot a^2,
	\]
	while the shape contribution is \(aA_\varepsilon=\varepsilon^2aQ(\xi)+o(\varepsilon^2)\).  Their sum is exactly the coefficient in \eqref{eq:Umetric}.  The mass term and the parallel sum are both positive definite, so the resulting quadratic form is positive definite as well.  Smoothness follows because \(a\mapsto a^{-1}\), the Gaussian tensors \(g_{\BW}\) and \(g_{\FR}\), and the parallel-sum map \((G,H)\mapsto(G^{-1}+H^{-1})^{-1}\) are smooth on their positive-definite domains.
\end{proof}

\begin{remark}
	Only \(\rBF\) carries the penalty dependence of this local tensor.  The construction records the diagonal second-order Riemannian tensor of the asymmetric endpoint functional, rather than a global squared-distance representation of the KL-UOT value.  Its parallel-sum form in \eqref{eq:Umetric} arises from the Schur-complement/infimal-convolution reduction of \eqref{eq:localquadraticxy}, in contrast with the additive Bures--Wasserstein--Fisher--Rao structure of Gaussian dynamic transport--reaction models; see, for example, \cite{ChizatFoCM2018,LieroMielkeTseZhu2026}.
\end{remark}

Absorbing the mass coefficient in the local tensor \eqref{eq:Umetric} into a square-root radial coordinate reveals its cone representation.

\begin{theorem}[Cone representation of the local KL-UOT tensor]
	Define the normalized-shape metric
	\begin{equation}
		h_{\rBF,\mu}
		:=
		\frac{1}{4\rBF}
		\left(g_{\BW,\mu}:\rBF g_{\FR,\mu}\right)
		\qquad(\mu\in\mathcal G_d^1),
	\end{equation}
	and introduce the radial mass coordinate
	\begin{equation}
		\varrho=2\sqrt{\rBF a},
		\qquad
		a=\frac{\varrho^2}{4\rBF}.
		\label{eq:coneradial}
	\end{equation}
	Under the diffeomorphism
	\[
	\mathcal G_d^+\ni a\mu
	\longmapsto
	(\varrho,\mu)\in(0,\infty)\times\mathcal G_d^1,
	\]
	the tensor \eqref{eq:Umetric} becomes
	\begin{equation}
		g^{\rm U}=d\varrho^2+\varrho^2h_{\rBF}.
		\label{eq:conemetric}
	\end{equation}
	Thus the finite Gaussian manifold equipped with the local tensor \(g^{\rm U}\) is isometric to the open Riemannian cone over \((\mathcal G_d^1,h_{\rBF})\).  In particular, along a fixed normalized shape \(\mu\), radial curves are geodesics for the metric induced by this local tensor and
	\begin{equation}
		d_{g^{\rm U}}(a_0\mu,a_1\mu)
		=2\sqrt{\rBF}\,
		\left|\sqrt{a_1}-\sqrt{a_0}\right|.
		\label{eq:radialdistance}
	\end{equation}
	The metric completion of the corresponding local tensor geometry in the radial direction adds a cone apex at zero mass, whose distance from \(a\mu\) is \(2\sqrt{\rBF a}\).
\end{theorem}

\begin{proof}
	Differentiating \eqref{eq:coneradial} gives \(da=\varrho(2\rBF)^{-1}\,d\varrho\).  Hence the mass and shape terms in \eqref{eq:Umetric} satisfy, respectively,
	\[
	\frac{\rBF}{a}\,da^2=d\varrho^2,
	\qquad
	a\left(g_{\BW}:\rBF g_{\FR}\right)
	=\varrho^2 h_{\rBF}.
	\]
	This proves \eqref{eq:conemetric}.  For any piecewise \(C^1\) curve \(\gamma(t)=(\varrho(t),\mu(t))\) joining two points with the same normalized shape, the cone metric assigns the length
	\[
	L(\gamma)
	=\int\sqrt{\dot\varrho(t)^2+\varrho(t)^2\norm{\dot\mu(t)}_{h_{\rBF}}^2}\,dt
	\ge \int\left\lvert \dot\varrho(t)\right\rvert\,dt
	\ge \left\lvert \varrho_1-\varrho_0\right\rvert.
	\]
	The radial curve with \(\mu(t)\equiv\mu\) attains this lower bound, so the distance is exactly \(|\varrho_1-\varrho_0|\), which is \eqref{eq:radialdistance}.  To identify the zero-radius limit across different shapes, let \(\mu_0,\mu_1\in\mathcal G_d^1\) and choose a piecewise \(C^1\) curve \(\eta\) joining them with finite \(h_{\rBF}\)-length.  At fixed radius \(\varrho>0\), the curve \(t\mapsto(\varrho,\eta(t))\) has \(g^{\rm U}\)-length \(\varrho L_{h_{\rBF}}(\eta)\), and hence
	\[
	d_{g^{\rm U}}\bigl((\varrho,\mu_0),(\varrho,\mu_1)\bigr)
	\le \varrho L_{h_{\rBF}}(\eta)
	\longrightarrow 0
	\qquad(\varrho\downarrow0).
	\]
	Since \(\mathcal G_d^1\) is connected and \(h_{\rBF}\) is a smooth positive-definite metric, such a finite-length curve exists between any two shapes.  Thus all radial rays determine the same zero-mass completion point, whose distance from \(a\mu\) is \(2\sqrt{\rBF a}\).  This proves the cone-apex assertion.
\end{proof}

\begin{remark}
	The theorem identifies the common radial cone apex.  Additional boundary points may arise from incompleteness of the base metric \(h_{\rBF}\), whose completeness properties therefore enter a full characterization of the metric completion.
\end{remark}

\begin{corollary}
	For \(\xi=(\dot m,H)\), the mean part of \eqref{eq:Umetric} is
	\begin{equation}
		\left(g_{\BW}:\rBF g_{\FR}\right)_{\rm mean}(\dot m,\dot m)
		=
		\dot m^\top
		\left(I+\rBF^{-1}\Sigma\right)^{-1}
		\dot m.
	\end{equation}
	If
	\[
	\Sigma=V\operatorname{diag}(\sigma_1,\ldots,\sigma_d)V^\top,
	\qquad
	\widetilde H=V^\top HV,
	\]
	then the covariance part is
	\begin{equation}
		\left(g_{\BW}:\rBF g_{\FR}\right)_{\rm cov}(H,H)
		=
		\sum_{j,k=1}^d
		\frac{\widetilde H_{jk}^2}
		{2(\sigma_j+\sigma_k)+2\rBF^{-1}\sigma_j\sigma_k}.
		\label{eq:covmetricexplicit}
	\end{equation}
\end{corollary}

\begin{proof}
	In the mean block, \(g_{\BW}\) and \(\rBF g_{\FR}\) are represented by \(I\) and \(\rBF\Sigma^{-1}\), so their parallel sum is \((I+\rBF^{-1}\Sigma)^{-1}\).  Diagonalizing \(\Sigma\), the covariance coefficients become \([2(\sigma_j+\sigma_k)]^{-1}\) and \(\rBF/(2\sigma_j\sigma_k)\); their scalar parallel sum recovers \eqref{eq:covmetricexplicit}.  Because the sum in \eqref{eq:covmetricexplicit} runs over ordered pairs \((j,k)\), the off-diagonal entries of a symmetric \(H\) occur twice, consistently with the Frobenius pairing.
\end{proof}

\begin{corollary}
	Fix \(\bar\tau_0,\bar\tau_1>0\) and set
	\(\tau_j=\kappa\bar\tau_j\) for \(j\in\{0,1\}\).  Let
	\[
	\bar{\rBF}=\frac{\bar\tau_0\bar\tau_1}{2(\bar\tau_0+\bar\tau_1)},
	\qquad
	\rBF_\kappa=\kappa\bar{\rBF}.
	\]
	Then:
	\begin{enumerate}[label=(\alph*)]
		\item as \(\kappa\to\infty\),
		\[
		g_{\BW}:\rBF_\kappa g_{\FR}\longrightarrow g_{\BW},
		\]
		while the mass coefficient \(\rBF_\kappa/a\to\infty\);
		\item as \(\kappa\downarrow0\),
		\begin{equation}
			\frac{1}{\rBF_\kappa}g_{\alpha}^{\rm U}
			\longrightarrow
			\frac{\dot a^2}{a}
			+a\,g_{\FR,\mu}(\xi,\xi).
		\end{equation}
	\end{enumerate}
\end{corollary}

\begin{proof}
	This follows directly from \(G:(cH)=(G^{-1}+c^{-1}H^{-1})^{-1}\): as \(c\to\infty\), \(G:(cH)\to G\), while as \(c\downarrow0\), \(c^{-1}(G:(cH))\to H\).  Apply this with \(G=g_{\BW}\), \(H=g_{\FR}\), \(c=\rBF_\kappa\), together with the mass term in \eqref{eq:Umetric}.  These are the infinitesimal counterparts of the two global regimes established in \Cref{sec:limits}.  Even when \(\tau_0\neq\tau_1\), the local quadratic form depends only on \(\rBF\), which is invariant under exchanging the two penalties; hence the quadratic term is symmetric under \(\tau_0\leftrightarrow\tau_1\).
\end{proof}

\section{MM iteration and convergence}
\label{sec:algorithm}

The stationary moment equations define a profiled reverse-KL MM iteration whose update is the Gaussian centroid of the adjusted target marginals.  We establish Lyapunov descent and spectral noncollapse, followed by sufficient-decrease and relative-error estimates that yield the full-sequence convergence result in \Cref{thm:cluster}.  A local surrogate Hessian then gives the contraction rate near nondegenerate stationary points.  For general fixed penalties, \Cref{thm:cluster} gives convergence to a stationary fixed point; \Cref{thm:largebranch} establishes global uniqueness in the sufficiently large common-penalty regime.

\subsection{Profiled reverse-KL MM iteration}

The MM step rests on the following reverse-KL Pythagorean identity for Gaussian centroids.

\begin{lemma}\label{lem:KLcentroid}
	Let \(q_i=\N(v_i,Q_i)\in\mathcal G_d^1\), and let positive weights \(\omega_i\) satisfy \(\sum_i\omega_i=1\).  Define
	\begin{equation}
		\bar m=\sum_i\omega_iv_i,
		\qquad
		\bar\Sigma
		=
		\sum_i\omega_i\left[Q_i+(v_i-\bar m)(v_i-\bar m)^\top\right],
		\label{eq:gausscentroid}
	\end{equation}
	and \(\bar q=\N(\bar m,\bar\Sigma)\).  Then \(\bar\Sigma\succ0\), and \(\bar q\) is the unique minimizer of
	\[
	\nu\longmapsto\sum_i\omega_i\KL(q_i\mid\nu)
	\qquad\text{over }\nu\in\mathcal G_d^1,
	\]
	and for every \(\nu\in\mathcal G_d^1\),
	\begin{equation}
		\sum_i\omega_i\KL(q_i\mid\nu)
		=
		\sum_i\omega_i\KL(q_i\mid\bar q)
		+\KL(\bar q\mid\nu).
		\label{eq:KLPythagorean}
	\end{equation}
\end{lemma}

\begin{proof}
	The covariance \(\bar\Sigma\) is positive definite by \eqref{eq:gausscentroid}.  Let \(\eta=\sum_i\omega_iq_i\).  By construction, \(\eta\) and \(\bar q\) have the same first two moments, and hence integrate every Gaussian log-density identically.  Expanding the relative entropies and cancelling the common terms leaves
	\[
	\sum_i\omega_i\KL(q_i\mid\nu)-\sum_i\omega_i\KL(q_i\mid\bar q)
	=
	\int\log\frac{\bar q}{\nu}\,d\eta
	=
	\KL(\bar q\mid\nu),
	\]
	which is \eqref{eq:KLPythagorean}.  Since \(\KL(\bar q\mid\nu)=0\) if and only if \(\nu=\bar q\), the minimizer is unique.
\end{proof}

Applying this identity to the adjusted target marginals and Gibbs weights gives the profiled MM update as their Gaussian reverse-KL centroid.

Given \(\nu^k=\N(m^k,\Sigma^k)\), compute for every \(i\) the endpoint quantities \(A_i^k:=A_i(\nu^k)\) and \(q_i^k:=\N(v_i^k,Q_i^k)\).
Set
\begin{equation}
	z_i^k=a_i^\theta e^{-A_i^k/s},
	\qquad
	Z^k=\sum_i\lambda_i z_i^k,
	\qquad
	\omega_i^k=\frac{\lambda_i z_i^k}{Z^k}.
	\label{eq:algorithmweights}
\end{equation}
Then update
\begin{equation}
	m^{k+1}=\sum_i\omega_i^kv_i^k,
	\label{eq:algm}
\end{equation}
\begin{equation}
	\Sigma^{k+1}
	=
	\sum_i\omega_i^k
	\left[
	Q_i^k+(v_i^k-m^{k+1})(v_i^k-m^{k+1})^\top
	\right],
	\label{eq:algSigma}
\end{equation}
and set \(\nu^{k+1}=\N(m^{k+1},\Sigma^{k+1})\).  The profiled mass is
\begin{equation}
	b^k=(Z^k)^{1/\theta}.
	\label{eq:algmass}
\end{equation}

The shape update \eqref{eq:algm}--\eqref{eq:algSigma} does not use the current mass explicitly.  The adjusted probability marginals \(q_i^k\) depend only on normalized shapes; mass information enters through the factors \(a_i^\theta\) in the Gibbs weights and through the profiled mass \eqref{eq:algmass}.

For stopping, after each shape update define the complete step residual
\begin{equation}
	\operatorname{res}_k=\max\left\{
	\sqrt{\KL(\nu^{k+1}\mid\nu^k)},
	\frac{\norm{m^{k+1}-m^k}}{1+\norm{m^k}},
	\frac{\norm{\Sigma^{k+1}-\Sigma^k}_F}{1+\norm{\Sigma^k}_F}
	\right\}.
	\label{eq:stoppingresidual}
\end{equation}
For any returned terminal shape \(\nu^{\rm term}\), the mass is reprofiled once by
\begin{equation}
	b^{\rm term}=\Zfun(\nu^{\rm term})^{1/\theta}.
	\label{eq:terminalmass}
\end{equation}
In the pseudocode below, \(\operatorname{LSE}(x_1,\ldots,x_n)=\log\sum_i e^{x_i}\) denotes the log-sum-exp operator.

\begin{algorithm}[!ht]
	\caption{Profiled reverse-KL MM iteration}
	\small
	\begin{algorithmic}[1]
		\Require Data \(\{a_i,m_i,\Sigma_i,\lambda_i\}_{i=1}^n\), penalties \(\tau_0,\tau_1>0\), initial shape \(\nu^0=\N(m^0,\Sigma^0)\), tolerance \(\varepsilon_{\rm tol}>0\), and iteration cap \(K_{\max}\).
		\Ensure Terminal stationary Gaussian candidate \(\beta^{\rm term}\) and a residual-based convergence flag.
		\For{\(k=0,\ldots,K_{\max}-1\)}
		\For{\(i=1,\ldots,n\)}
		\State Evaluate \(A_i^k\) and \(q_i^k=\N(v_i^k,Q_i^k)\).
		\State \(\ell_i^k\gets\log\lambda_i+\theta\log a_i-A_i^k/s\).
		\EndFor
		\State \(\log Z^k\gets\operatorname{LSE}(\ell_1^k,\ldots,\ell_n^k)\).
		\State \(\omega_i^k\gets\exp(\ell_i^k-\log Z^k)\) for \(i=1,\ldots,n\).
		\State \(\log b^k\gets(\log Z^k)/\theta\) and \(b^k\gets\exp(\log b^k)\) if the current mass is required.
		\State Update \(m^{k+1}\) and \(\Sigma^{k+1}\) using \eqref{eq:algm}--\eqref{eq:algSigma}.
		\State \(\nu^{k+1}\gets\N(m^{k+1},\Sigma^{k+1})\); evaluate \(\operatorname{res}_k\) by \eqref{eq:stoppingresidual}.
		\If{\(\operatorname{res}_k\le\varepsilon_{\rm tol}\)}
		\State \(\nu^{\rm term}\gets\nu^{k+1}\); reprofile \(b^{\rm term}\) by \eqref{eq:terminalmass}.
		\State \(\beta^{\rm term}\gets b^{\rm term}\nu^{\rm term}\).
		\State \Return \((\beta^{\rm term},\texttt{converged})\).
		\EndIf
		\EndFor
		\State \(\nu^{\rm term}\gets\nu^{K_{\max}}\); reprofile \(b^{\rm term}\) by \eqref{eq:terminalmass}.
		\State \(\beta^{\rm term}\gets b^{\rm term}\nu^{\rm term}\).
		\State \Return \((\beta^{\rm term},\texttt{not converged})\).
	\end{algorithmic}
\end{algorithm}

The flag records attainment of the prescribed complete step residual.  The stationarity conclusion in \Cref{thm:cluster} concerns the exact infinite MM trajectory; global optimality at intermediate penalties requires separate information.  Each dense iteration costs \(O(nd^3)\), dominated by the pairwise positive-definite matrix operations.

\subsection{Lyapunov descent and convergence}

To quantify descent, we freeze the exact pairwise plans at the current iterate and regard the barycenter variable as the free MM block.  For fixed \(\beta\in\mathcal G_d^+\), define the lifted objective
\begin{equation}
	\mathcal J(\beta;\gamma_1,\ldots,\gamma_n)
	=\!{\textstyle\sum_i\lambda_i[\int\!\norm{x-y}^2d\gamma_i+\tau_0\KL((\gamma_i)_0\!\mid\!\alpha_i)+\tau_1\KL((\gamma_i)_1\!\mid\!\beta)]}.
\end{equation}
Then \(\Fbar(\beta)=\inf_{\gamma_1,\ldots,\gamma_n}\mathcal J(\beta;\gamma_1,\ldots,\gamma_n)\).

\begin{proposition}\label{prop:analyticmomentmap}
	Let
	\[
	\mathcal X:=\R^d\times\Spp^d,
	\qquad x=(m,\Sigma),
	\qquad \nu_x=\N(m,\Sigma).
	\]
	The update \eqref{eq:algm}--\eqref{eq:algSigma} defines a well-posed real-analytic self-map \(T:\mathcal X\to\mathcal X\).  For \(x,y\in\mathcal X\), let \(\gamma_i(x)\) be the KL-UOT optimizer between \(\alpha_i\) and \(b(x)\nu_x\), where \(b(x)=\Zfun(\nu_x)^{1/\theta}\), and define
	\begin{equation}
		\mathcal S(y\mid x)
		:=\mathcal J\bigl(b(x)\nu_y;\gamma_1(x),\ldots,\gamma_n(x)\bigr).
		\label{eq:globalMMsurrogate}
	\end{equation}
	Then \(\mathcal S\) is real analytic on \(\mathcal X\times\mathcal X\) and satisfies
	\begin{equation}
		\Phi(\nu_y)\le \mathcal S(y\mid x),
		\qquad
		\Phi(\nu_x)=\mathcal S(x\mid x),
		\label{eq:globalMMmajorization}
	\end{equation}
	while
	\begin{equation}
		T(x)=\operatorname*{argmin}_{y\in\mathcal X}\mathcal S(y\mid x)
		\label{eq:globalMMargmin}
	\end{equation}
	is the unique minimizer.
\end{proposition}

\begin{proof}
	For each \(x\), the pairwise interface defines unique adjusted targets \(q_i(x)=\N(v_i(x),Q_i(x))\) with real-analytic dependence on \(x\); the Gibbs weights are positive and analytic.  The covariance update is a positive weighted sum of the matrices \(Q_i(x)\succ0\), plus a positive-semidefinite between-mean term, so \(T(x)\in\mathcal X\).  The explicit moment formulas then show analyticity of \(T\).
	
	For the surrogate, pairwise optimality and exact mass profiling imply the two inequalities in \eqref{eq:globalMMmajorization}, with equality on the diagonal.  Write the terminal marginal of \(\gamma_i(x)\) as \(M_i(x)q_i(x)\).  The scaled-KL identity separates the \(y\)-dependence explicitly as
	\[
	\mathcal S(y\mid x)
	=C(x)+\tau_1\sum_i\lambda_iM_i(x)\KL(q_i(x)\mid\nu_y),
	\]
	where \(C(x)\) is independent of \(y\).  This finite-dimensional representation shows directly that \(\mathcal S\) is real analytic.  The Gaussian reverse-KL centroid identity in \Cref{lem:KLcentroid} gives the moment update as the unique minimizer \eqref{eq:algm}--\eqref{eq:algSigma}.
\end{proof}

\begin{corollary}
	The moment map has at least one fixed point in \(\mathcal G_d^1\).  Every global minimizer of the Gaussian-restricted barycenter problem is such a fixed point.
\end{corollary}

\begin{proof}
	By \Cref{thm:globalexistence}, the smooth profiled functional \(\Phi\) has a global minimizer in the open manifold \(\mathcal G_d^1\).  Every such minimizer is stationary, and \Cref{thm:fixedpoint} shows that stationary points are exactly the fixed points of the moment map.
\end{proof}

The surrogate also implies a quantitative Lyapunov decrease.

\begin{theorem}[MM descent]\label{thm:descent}
	Let \(\nu^{k+1}\) be obtained from \(\nu^k\) by \eqref{eq:algorithmweights}--\eqref{eq:algSigma}, and let \(b^k=b(\nu^k)\).  Then
	\begin{equation}
		\Phi(\nu^k)-\Phi(\nu^{k+1})
		\geq
		\tau_1 b^k\KL(\nu^{k+1}\mid\nu^k)
		\geq0.
		\label{eq:descent}
	\end{equation}
	The descent is strict whenever \(\nu^{k+1}\neq\nu^k\); equivalently,
	\[
	\Phi(\nu^{k+1})=\Phi(\nu^k)
	\quad\Longleftrightarrow\quad
	\nu^{k+1}=\nu^k.
	\]
\end{theorem}

\begin{proof}
	Let \(\beta^k=b^k\nu^k\), and let \(\gamma_i^k=M_i^k\Pi_i^k\) be the exact pairwise KL-UOT optimizer between \(\alpha_i\) and \(\beta^k\).  The transported-mass formula \eqref{eq:Mstar} gives \(M_i^k=(b^k)^{1-\theta}z_i^k\).
	Using \(b^k=(Z^k)^{1/\theta}\), these masses satisfy the normalization identities
	\begin{equation}
		\sum_i\lambda_iM_i^k
		=(b^k)^{1-\theta}Z^k=b^k,
		\qquad
		\frac{\lambda_iM_i^k}{b^k}=\omega_i^k.
		\label{eq:massidentityk}
	\end{equation}
	The terminal marginal of \(\gamma_i^k\) is \(M_i^kq_i^k\).
	
	Fix the mass at \(b^k\) and set \(\widetilde\beta^{k+1}=b^k\nu^{k+1}\).  With the plans \(\gamma_i^k\) fixed, replacing \(\beta^k\) by \(\widetilde\beta^{k+1}\) changes only the terminal KL terms.  Using \eqref{eq:scaledKL}, the mass-dependent terms cancel and
	\[
		\mathcal J(\beta^k;\gamma_1^k,\ldots,\gamma_n^k)-\mathcal J(\widetilde\beta^{k+1};\gamma_1^k,\ldots,\gamma_n^k)
		=\tau_1\sum_i\lambda_iM_i^k\bigl[\KL(q_i^k\mid\nu^k)-\KL(q_i^k\mid\nu^{k+1})\bigr].
	\]
	The normalized masses in \eqref{eq:massidentityk} are precisely \(\omega_i^k\).  The centroid identity in \Cref{lem:KLcentroid} then turns the last display into \(\tau_1b^k\KL(\nu^{k+1}\mid\nu^k)\).
	At the current iterate, pairwise optimality makes the lifted and profiled objectives coincide:
	\[
	\Fbar(\beta^k)
	=\mathcal J(\beta^k;\gamma_1^k,\ldots,\gamma_n^k).
	\]
	Reusing the same plans at the updated shape bounds the new objective from above:
	\[
	\Fbar(\widetilde\beta^{k+1})
	\leq
	\mathcal J(\widetilde\beta^{k+1};\gamma_1^k,\ldots,\gamma_n^k).
	\]
	Reprofiling the mass at the updated shape does not increase the objective:
	\[
	\Phi(\nu^{k+1})
	=\inf_{b>0}\Fbar(b,\nu^{k+1})
	\leq
	\Fbar(\widetilde\beta^{k+1}).
	\]
	Combining these inequalities proves \eqref{eq:descent}.  Strictness follows from \(\KL(\nu^{k+1}\mid\nu^k)>0\) whenever \(\nu^{k+1}\ne\nu^k\).
\end{proof}

\begin{corollary}\label{cor:globaliteration}
	For every initialization \(\nu^0\in\mathcal G_d^1\), the profiled reverse-KL MM iteration satisfies:
	\begin{enumerate}[label=(\alph*)]
		\item \(\Phi(\nu^k)\) is nonincreasing and the profiled masses \(b^k=b(\nu^k)\) are nondecreasing.  Moreover,
		\begin{equation}
			b^0\le b^k\uparrow b_\infty
			\le
			\left(\sum_{i=1}^n\lambda_i a_i^\theta\right)^{1/\theta};
		\end{equation}
		\item the KL steps are summable:
		\begin{equation}
			\sum_{k=0}^{\infty}\KL(\nu^{k+1}\mid\nu^k)
			\le
			\frac{\tau_0}{\tau_1 b^0}(b_\infty-b^0)<\infty,
			\label{eq:KLsummable}
		\end{equation}
		and hence \(\KL(\nu^{k+1}\mid\nu^k)\to0\);
		\item there exist constants \(R<\infty\) and \(\Lambda<\infty\), depending only on the data, penalties, and initialization, such that
		\begin{equation}
			\norm{m^k}\le R,
			\qquad
			\Sigma^k\preceq\Lambda I
			\qquad(k\ge0).
			\label{eq:trajectoryuppercompact}
		\end{equation}
	\end{enumerate}
\end{corollary}

\begin{proof}
	By \eqref{eq:profiledF},
	\[
	\Phi(\nu^k)-\Phi(\nu^{k+1})
	=\tau_0(b^{k+1}-b^k).
	\]
	Insert the descent bound of \Cref{thm:descent} into this identity:
	\begin{equation}
		\tau_0(b^{k+1}-b^k)
		\ge
		\tau_1b^k\KL(\nu^{k+1}\mid\nu^k)\ge0.
		\label{eq:massKLstep}
	\end{equation}
	The nonnegative right-hand side makes \(b^k\) nondecreasing, and \Cref{cor:massbound} supplies a finite upper bound; hence \(b^k\to b_\infty\).  Since \(b^k\ge b^0>0\), summing \eqref{eq:massKLstep} establishes the finite KL budget in \eqref{eq:KLsummable}.
	
	To obtain \eqref{eq:trajectoryuppercompact}, set \(Z_{\rm init}:=\Zfun(\nu^0)=(b^0)^\theta>0\).
	Since \(b^k\) is nondecreasing, \(\Zfun(\nu^k)=(b^k)^\theta\ge Z_{\rm init}\) for every \(k\).  One term in this finite sum must satisfy
	\[
	\lambda_{i(k)}a_{i(k)}^\theta e^{-A_{i(k)}(\nu^k)/s}
	\ge \frac{Z_{\rm init}}{n}.
	\]
	Whenever the selected index equals a fixed \(i\), that term obeys
	\[
	A_i(\nu^k)
	\le
	s\log\!\left(\frac{n\lambda_i a_i^\theta}{Z_{\rm init}}\right).
	\]
	Only finitely many indices can occur.  Taking the largest constants from \Cref{lem:uppercompact} over those endpoint sublevels bounds the whole trajectory by common \(R,\Lambda\), proving \eqref{eq:trajectoryuppercompact}.
\end{proof}

The following pairwise covariance inequality supplies the lower spectral floor not provided by these upper-sublevel bounds.

\begin{lemma}[Spectral transfer]\label{lem:spectraltransfer}
	Consider a pairwise Gaussian endpoint with target \(\nu=\N(m,\Sigma)\in\mathcal G_d^1\), adjusted source covariance \(P\), and adjusted target covariance \(Q\).  Suppose that
	\begin{equation}
		P\succeq \underline p I
		\qquad\text{for some }\underline p>0.
		\label{eq:spectraltransferPlower}
	\end{equation}
	Set
	\[
	t=\lambda_{\min}(\Sigma),
	\qquad
	q=\lambda_{\min}(Q),
	\qquad
	\eta_1=\frac{\tau_1}{2}.
	\]
	Then
	\begin{equation}
		q\ge \psi_{\underline p,\eta_1}(t)
		:=
		\left[
		\frac{
			\sqrt{\underline p+4\eta_1(1+\eta_1/t)}-\sqrt{\underline p}
		}{2\eta_1}
		\right]^{-2}.
		\label{eq:spectraltransfer}
	\end{equation}
	Moreover,
	\begin{equation}
		\psi_{\underline p,\eta_1}(t)
		\ge \min\{t,\underline p\}
		\qquad(t>0),
		\label{eq:psibarrier}
	\end{equation}
	with strict inequality relative to the smaller quantity whenever \(t\ne\underline p\).  In particular, if \(t<\underline p\), then \(\lambda_{\min}(Q)>t\).  If, in addition, the endpoint satisfies \(\Ashape(\mu,\nu)\le C\) for a fixed source \(\mu\), then \Cref{lem:KLsublevel} supplies a choice of \(\underline p\) depending only on \((\mu,C,\tau_0)\).
\end{lemma}

\begin{proof}
	Let \(L\succ0\) be the Wasserstein map from \(P\) to \(Q\), so \(Q=LPL\).  The lower bound \eqref{eq:spectraltransferPlower} implies
	\[
	Q^{-1}=L^{-1}P^{-1}L^{-1}
	\preceq \underline p^{-1}L^{-2}.
	\]
	Operator monotonicity of the principal square root transforms this into
	\begin{equation}
		\sqrt{\underline p}\,Q^{-1/2}\preceq L^{-1}.
		\label{eq:LinvQinvhalf}
	\end{equation}
	The adjusted-target covariance first-order condition \eqref{eq:Qfirstorder} can be written as
	\[
	I+\eta_1\Sigma^{-1}
	=L^{-1}+\eta_1Q^{-1}.
	\]
	Combining this identity with \eqref{eq:LinvQinvhalf} and taking largest eigenvalues reduces the estimate to
	\begin{equation}
		1+\frac{\eta_1}{t}
		\ge
		\frac{\sqrt{\underline p}}{\sqrt q}
		+\frac{\eta_1}{q}.
	\end{equation}
	Writing \(x=q^{-1/2}\), the left side of
	\[
	\eta_1x^2+\sqrt{\underline p}\,x
	\le 1+\frac{\eta_1}{t}
	\]
	is strictly increasing for \(x>0\).  Solving the corresponding quadratic equality explicitly recovers \eqref{eq:spectraltransfer}.  To compare the barrier with \(t\) and \(\underline p\), evaluate
	\[
	f(x)=\eta_1x^2+\sqrt{\underline p}\,x
	\]
	at \(x=t^{-1/2}\) and \(x=\underline p^{-1/2}\).  If \(t<\underline p\), then
	\[
	f(t^{-1/2})
	=\frac{\eta_1}{t}+\sqrt{\frac{\underline p}{t}}
	>1+\frac{\eta_1}{t},
	\]
	so the positive root in \eqref{eq:spectraltransfer} is smaller than \(t^{-1/2}\), hence \(\psi_{\underline p,\eta_1}(t)>t\).  If \(t>\underline p\), then
	\[
	f(\underline p^{-1/2})
	=1+\frac{\eta_1}{\underline p}
	>1+\frac{\eta_1}{t},
	\]
	so \(\psi_{\underline p,\eta_1}(t)>\underline p\).  Equality occurs at \(t=\underline p\), proving \eqref{eq:psibarrier}.
	
	If \(\Ashape(\mu,\nu)\le C\), nonnegativity of the endpoint objective bounds \(\KL(p\mid\mu)\) by \(C/\tau_0\).  Applying \Cref{lem:KLsublevel} to the fixed reference Gaussian \(\mu\) establishes \eqref{eq:spectraltransferPlower} with a constant depending only on \((\mu,C,\tau_0)\).
\end{proof}

\begin{corollary}[Covariance noncollapse]\label{cor:trajectorynoncollapse}
	For every initialization \(\nu^0=\N(m^0,\Sigma^0)\in\mathcal G_d^1\), there exists \(\underline\sigma_{\rm orb}>0\), depending on the data, penalties, and initialization, such that
	\begin{equation}
		\Sigma^k\succeq \underline\sigma_{\rm orb}I
		\qquad(k\ge0).
		\label{eq:trajectorynoncollapse}
	\end{equation}
	More precisely, one may choose a constant \(\underline p_{\rm orb}>0\) for which
	\begin{equation}
		\lambda_{\min}(\Sigma^{k+1})
		\ge
		\min\{\lambda_{\min}(\Sigma^k),\underline p_{\rm orb}\},
		\label{eq:spectralrecurrence}
	\end{equation}
	and therefore
	\begin{equation}
		\underline\sigma_{\rm orb}
		=\min\{\lambda_{\min}(\Sigma^0),\underline p_{\rm orb}\}>0.
		\label{eq:orbitalfloor}
	\end{equation}
\end{corollary}

\begin{proof}
	By \Cref{cor:globaliteration}, the orbit satisfies common bounds \(\norm{m^k}\le R\) and \(\Sigma^k\preceq\Lambda I\).  For each fixed input, testing with \(p=\mu_i\), \(q=\nu^k\) shows
	\[
	A_i(\nu^k)\le W_2^2(\mu_i,\nu^k)\le C_i
	\]
	uniformly in \(k\).  Thus the adjusted source covariance \(P_i^k\) belongs to a fixed KL sublevel relative to \(\mu_i\), and there is \(\underline p_i>0\) such that
	\[
	P_i^k\succeq\underline p_i I
	\qquad(k\ge0).
	\]
	Set \(\underline p_{\rm orb}=\min_i\underline p_i\) and \(t_k=\lambda_{\min}(\Sigma^k)\).  Applying \Cref{lem:spectraltransfer} to every adjusted target covariance bounds its smallest eigenvalue by
	\[
	\lambda_{\min}(Q_i^k)
	\ge \min\{t_k,\underline p_i\}
	\ge \min\{t_k,\underline p_{\rm orb}\}.
	\]
	Since the between-mean term in \eqref{eq:algSigma} is positive semidefinite,
	\[
		\lambda_{\min}(\Sigma^{k+1})\ge
		\lambda_{\min}\!\left(\sum_i\omega_i^kQ_i^k\right)\ge
		\sum_i\omega_i^k\lambda_{\min}(Q_i^k)\ge
		\min\{t_k,\underline p_{\rm orb}\},
	\]
	which is \eqref{eq:spectralrecurrence}.  Induction proves \eqref{eq:orbitalfloor}.  If \(t_k<\underline p_{\rm orb}\), then \Cref{lem:spectraltransfer} implies \(\lambda_{\min}(Q_i^k)>t_k\) for every \(i\); thus the boundary is strictly repelling in the minimum-eigenvalue sense.
\end{proof}

With the orbit confined to a two-sided spectral box, Gaussian KL controls displacement in the ambient parameter metric.

\begin{lemma}\label{lem:KLseparation}
	Fix \(0<\underline\sigma\le\overline\sigma<\infty\).  If
	\[
	\nu=\N(m,\Sigma),\qquad \nu'=\N(m',\Sigma'),
	\qquad
	\underline\sigma I\preceq\Sigma,\Sigma'\preceq\overline\sigma I,
	\]
	with arbitrary means \(m,m'\in\R^d\), then there exists
	\(c_{\underline\sigma,\overline\sigma}>0\) such that
	\begin{equation}
		\KL(\nu'\mid\nu)
		\ge c_{\underline\sigma,\overline\sigma}\,
		d_{\rm par}(\nu,\nu')^2.
		\label{eq:KLquadraticcontrol}
	\end{equation}
	Consequently, for every \(\varepsilon>0\), the infimum of
	\(\KL(\nu'\mid\nu)\) over pairs satisfying the same spectral bounds and
	\(d_{\rm par}(\nu,\nu')\ge\varepsilon\) is positive.  In particular, for
	sequences with common two-sided covariance bounds,
	\(\KL(\nu_k'\mid\nu_k)\to0\) implies
	\(d_{\rm par}(\nu_k,\nu_k')\to0\).
\end{lemma}

\begin{proof}
	Set \(B:=\Sigma^{-1/2}\Sigma'\Sigma^{-1/2}\).
	The Gaussian KL formula reads
	\[
	2\KL(\nu'\mid\nu)
	=\norm{\Sigma^{-1/2}(m'-m)}^2
	+\tr B-\log\det B-d.
	\]
	The mean term is bounded below by
	\(\overline\sigma^{-1}\norm{m'-m}^2\).  The spectrum of \(B\) lies in the fixed interval
	\([\underline\sigma/\overline\sigma,\overline\sigma/\underline\sigma]\).  For
	\(f(t)=t-\log t-1\), the quotient \(f(t)/(t-1)^2\), continuously extended by the value \(1/2\) at \(t=1\), has a positive minimum \(c_f\) on this interval.  Hence
	\[
	\tr B-\log\det B-d
	=\sum_j f(\lambda_j(B))
	\ge c_f\norm{B-I}_F^2.
	\]
	Using \(\Sigma'-\Sigma=\Sigma^{1/2}(B-I)\Sigma^{1/2}\), submultiplicativity gives
	\[
	\norm{\Sigma'-\Sigma}_F
	\le \overline\sigma\norm{B-I}_F.
	\]
	Combining these estimates proves \eqref{eq:KLquadraticcontrol}, where one may take \(c_{\underline\sigma,\overline\sigma}:=\frac12\min\{\overline\sigma^{-1},c_f\overline\sigma^{-2}\}\).
	The remaining assertions follow immediately.
\end{proof}

The Lyapunov and spectral estimates provide compactness and square-summable KL steps, but full-sequence convergence additionally requires sufficient decrease and relative error.  We establish these estimates first in mean--covariance coordinates and then transfer them through \Cref{lem:KLlogcoords} to global log-covariance coordinates, where the Kurdyka--\L{}ojasiewicz theorem applies directly.

\begin{lemma}\label{lem:KLrelativeerror}
	Let \((\nu^k)_{k\geq0}\) be any profiled MM trajectory generated from \(\nu^0\in\mathcal G_d^1\).  Regard the parameter domain \(\mathcal X:=\R^d\times\Spp^d\) as an open subset of the Euclidean space \(\R^d\times\Sym^d\), equipped with the norm underlying \(d_{\rm par}\), and let \(\nabla_{\rm par}\Phi\) denote the corresponding gradient.  Then there exist constants \(a,L>0\) such that the following estimates hold for every \(k\):
	\begin{align}
		\Phi(\nu^k)-\Phi(\nu^{k+1})
		&\ge a\,d_{\rm par}(\nu^k,\nu^{k+1})^2,
		\label{eq:KLsufficientdecrease}\\
		\norm{\nabla_{\rm par}\Phi(\nu^{k+1})}
		&\le L\,d_{\rm par}(\nu^k,\nu^{k+1}).
		\label{eq:KLrelativeerror}
	\end{align}
\end{lemma}

\begin{proof}
	By \Cref{cor:globaliteration,cor:trajectorynoncollapse}, there are constants \(R<\infty\) and \(0<\underline\sigma\le\Lambda<\infty\) such that the whole trajectory satisfies \(\norm{m^k}\le R\) and \(\underline\sigma I\preceq\Sigma^k\preceq\Lambda I\).  Its closure is therefore a compact subset \(\mathcal K\Subset\mathcal G_d^1\).  Combining \Cref{thm:descent,lem:KLseparation} with \(b^k\ge b^0>0\), the sufficient-decrease estimate becomes
	\[
	\Phi(\nu^k)-\Phi(\nu^{k+1})
	\ge \tau_1b^0c_{\underline\sigma,\Lambda}\,
	d_{\rm par}(\nu^k,\nu^{k+1})^2,
	\]
	which is \eqref{eq:KLsufficientdecrease}.
	
	For the relative-error estimate, use the global surrogate \(\mathcal S(y\mid x)\) from \Cref{prop:analyticmomentmap} and set
	\[
	\mathcal R(y,x)
	:=\nabla_y\mathcal S(y\mid x)-\nabla_{\rm par}\Phi(\nu_y).
	\]
	The majorization relation \eqref{eq:globalMMmajorization} is tangent on the diagonal, hence \(\mathcal R(x,x)=0\) for every \(x\in\mathcal X\).  Choose a slightly larger compact convex box \(\mathcal K^+\Subset\mathcal X\).  Analyticity of \(\mathcal S\) and \(\Phi\) bounds \(D_y\mathcal R\) uniformly over \(\mathcal K^+\times\mathcal K^+\).  The segment from \(x^k\) to \(x^{k+1}\) lies in \(\mathcal K^+\); the fundamental theorem of calculus, together with \(\mathcal R(x^k,x^k)=0\), gives
	\[
	\mathcal R(x^{k+1},x^k)
	=
	\int_0^1
	D_y\mathcal R\!\left(x^k+t(x^{k+1}-x^k),x^k\right)
	[x^{k+1}-x^k] \,dt.
	\]
	Hence
	\[
	\norm{\mathcal R(x^{k+1},x^k)}
	\le L\norm{x^{k+1}-x^k}.
	\]
	Finally, \eqref{eq:globalMMargmin} states
	\(\nabla_y\mathcal S(x^{k+1}\mid x^k)=0\), and therefore
	\[
	\norm{\nabla_{\rm par}\Phi(\nu^{k+1})}
	=\norm{\mathcal R(x^{k+1},x^k)}
	\le Ld_{\rm par}(\nu^k,\nu^{k+1}),
	\]
	which proves \eqref{eq:KLrelativeerror}.
\end{proof}

To apply a global Euclidean K\L{} convergence theorem without changing local charts along the orbit, we pass to log-covariance coordinates.

\begin{lemma}\label{lem:KLlogcoords}
	Let \((\nu^k)_{k\geq0}\) be any profiled MM trajectory generated from \(\nu^0\in\mathcal G_d^1\).  Let
	\[
	\mathcal E:=\R^d\times\Sym^d,
	\qquad
	\norm{(h,H)}_{\mathcal E}^2:=\norm{h}^2+\norm{H}_F^2,
	\]
	and define the global coordinate map and pulled-back objective by
	\begin{equation}
		\Xi(m,S):=\N\!\left(m,\exp(S)\right),
		\qquad
		\Psi(m,S):=\Phi\bigl(\Xi(m,S)\bigr).
	\end{equation}
	Then \(\Xi:\mathcal E\to\mathcal G_d^1\) is a real-analytic diffeomorphism, with inverse \(\Xi^{-1}(\N(m,\Sigma))=(m,\log\Sigma)\), and \(\Psi:\mathcal E\to\R\) is real analytic.  If
	\[
	w^k=(m^k,S^k),
	\qquad
	S^k=\log\Sigma^k,
	\]
	then there exist constants \(c_-,c_+,\widetilde a,\widetilde L>0\) and a compact convex neighborhood \(\mathcal W\subset\mathcal E\) of the orbit \(\{w^k:k\geq0\}\) such that, for \(w,w'\in\mathcal W\),
	\begin{equation}
		c_-\norm{w-w'}_{\mathcal E}
		\le
		d_{\rm par}\bigl(\Xi(w),\Xi(w')\bigr)
		\le
		c_+\norm{w-w'}_{\mathcal E},
		\label{eq:logcovbilipschitz}
	\end{equation}
	and, for every \(k\),
	\begin{align}
		\Psi(w^k)-\Psi(w^{k+1})
		&\ge \widetilde a\norm{w^{k+1}-w^k}_{\mathcal E}^2,
		\label{eq:ABSH1}\\
		\norm{\nabla\Psi(w^{k+1})}_{\mathcal E}
		&\le \widetilde L\norm{w^{k+1}-w^k}_{\mathcal E}.
		\label{eq:ABSH2}
	\end{align}
\end{lemma}

\begin{proof}
	The matrix exponential is a real-analytic diffeomorphism from \(\Sym^d\) onto \(\Spp^d\), with inverse the principal logarithm~\cite{Higham2008}; hence \(\Xi\) and \(\Psi=\Phi\circ\Xi\) are real analytic.  By \Cref{cor:globaliteration,cor:trajectorynoncollapse}, the trajectory lies in a fixed compact mean--spectral box.  On a slightly larger box, \(\exp\), \(\log\), and their first derivatives are uniformly bounded, so the mean-value theorem establishes the bi-Lipschitz comparison \eqref{eq:logcovbilipschitz}.  The sufficient-decrease estimate \eqref{eq:ABSH1} follows from \Cref{lem:KLrelativeerror} and the left inequality in \eqref{eq:logcovbilipschitz}.  Finally, the chain rule gives
	\[
	\nabla\Psi(w)=D\Xi(w)^*\nabla_{\rm par}\Phi(\Xi(w)),
	\]
	and the uniform bound on \(D\Xi\), together with \eqref{eq:KLrelativeerror} and the right inequality in \eqref{eq:logcovbilipschitz}, proves \eqref{eq:ABSH2}.  These estimates match the hypotheses of the K\L{} convergence theorem.
\end{proof}

\begin{theorem}[Finite-length MM convergence]\label{thm:cluster}
	For every initialization \(\nu^0\in\mathcal G_d^1\), the profiled MM trajectory is contained in a compact subset of \(\mathcal G_d^1\), has finite length in mean--covariance parameter space,
	\begin{equation}
		\sum_{k=0}^{\infty}
		d_{\rm par}(\nu^{k+1},\nu^k)<\infty,
		\label{eq:finiteparameterlength}
	\end{equation}
	and converges to a single stationary fixed point \(\nu_\infty\in\mathcal G_d^1\).  Moreover, with \(b^k=b(\nu^k)\) and \(\beta^k=b^k\nu^k\),
	\begin{equation}
		\sum_{k=0}^{\infty}
		\left(|b^{k+1}-b^k|+d_{\rm par}(\nu^{k+1},\nu^k)\right)<\infty,
		\label{eq:finitefullparameterlength}
	\end{equation}
	and the full finite Gaussian iterates converge in the mass, mean, and covariance parameters to
	\[
	\beta_\infty=b(\nu_\infty)\nu_\infty.
	\]
\end{theorem}

\begin{proof}
	The upper bounds \eqref{eq:trajectoryuppercompact} and the lower bound \eqref{eq:trajectorynoncollapse} place the entire trajectory in a compact subset \(\mathcal K\Subset\mathcal G_d^1\).  Introduce the global log-covariance coordinates from \Cref{lem:KLlogcoords},
	\[
	w^k=(m^k,\log\Sigma^k)\in\mathcal E,
	\qquad
	\Psi=\Phi\circ\Xi.
	\]
	Since \(\Psi\) is finite-valued and real analytic on the Euclidean space \(\mathcal E\), it is proper in the extended-real-valued sense, continuous (hence lower semicontinuous), and satisfies the Kurdyka--\L{}ojasiewicz property at every point.  In the present smooth setting, its limiting subdifferential is \(\partial\Psi(w)=\{\nabla\Psi(w)\}\), so the relative-error condition in Theorem~2.9 of Attouch, Bolte, and Svaiter~\cite{AttouchBolteSvaiter2013} is precisely \eqref{eq:ABSH2}.  The lower bound \(\Phi\ge0\) controls the objective values.  Hypotheses H1 and H2 are \eqref{eq:ABSH1} and \eqref{eq:ABSH2}, respectively.  For H3, compactness of the orbit provides a convergent subsequence \(w^{k_j}\to\bar w\), and continuity gives \(\Psi(w^{k_j})\to\Psi(\bar w)\).

	Thus all hypotheses of Theorem~2.9 are satisfied, and its finite-length conclusion gives
	\[
	\sum_{k=0}^{\infty}\norm{w^{k+1}-w^k}_{\mathcal E}<\infty,
	\qquad
	w^k\longrightarrow w_\infty,
	\]
	where \(w_\infty\) is a critical point of \(\Psi\).  Since \(\Xi\) is a diffeomorphism, the chain rule implies that \(\nu_\infty:=\Xi(w_\infty)\) is stationary for \(\Phi\).  The upper Lipschitz bound in \eqref{eq:logcovbilipschitz} transfers finite length back to the original parameters and proves \eqref{eq:finiteparameterlength}.  By \Cref{cor:globaliteration}, \(b^k\uparrow b_\infty\), while continuity of \(b(\nu)=\Zfun(\nu)^{1/\theta}\) and \(\nu^k\to\nu_\infty\) gives \(b_\infty=b(\nu_\infty)\).  Hence \(\sum_k|b^{k+1}-b^k|=b(\nu_\infty)-b^0<\infty\); together with \eqref{eq:finiteparameterlength}, this proves \eqref{eq:finitefullparameterlength} and convergence of the full finite Gaussian iterates.  Finally, \Cref{thm:fixedpoint} shows that \(\nu_\infty\) is a fixed point of the moment map.
\end{proof}

\subsection{Local linear convergence}

If a stationary fixed point is locally nondegenerate, the MM surrogate determines a quantitative linear rate.

\begin{lemma}\label{lem:MMsurrogate}
	Let \(\nu_*\in\mathcal G_d^1\) be a fixed point of the moment map and set \(b_*=b(\nu_*)\).  Choose a smooth local chart \(\chi\) whose image \(V\) contains \(x_*=\chi(\nu_*)\), and write \(\nu_x=\chi^{-1}(x)\) and \(b(x)=b(\nu_x)\).  By continuity of the moment map, shrink to an open neighborhood \(U\subset V\) of \(x_*\) such that \(\chi(T(\nu_x))\in V\) for all \(x\in U\).  Pull back the global surrogate in \eqref{eq:globalMMsurrogate} to these coordinates and denote the pullback again by \(\mathcal S(y\mid x)\).  With \(\gamma_i(x)\) denoting the optimizer attached to \(\nu_x\),
	\begin{equation}
		\mathcal S(y\mid x)
		=
		\mathcal J\bigl(b(x)\nu_y;\gamma_1(x),\ldots,\gamma_n(x)\bigr).
	\end{equation}
	Then:
	\begin{enumerate}[label=(\alph*)]
		\item \(\Phi(\nu_y)\le \mathcal S(y\mid x)\) for all \(x\in U\) and \(y\in V\), with equality at \(y=x\);
		\item the fixed-point update is the unique minimizer of the surrogate over \(V\), that is,
		\[
		\chi(T(\nu_x))=\operatorname*{argmin}_{y\in V}\mathcal S(y\mid x);
		\]
		\item at the fixed point,
		\begin{equation}
			\mathcal S(y\mid x_*)
			=
			K_*+\tau_1b_*\KL(\nu_*\mid\nu_y)
			\label{eq:MMfixedform}
		\end{equation}
		for a constant \(K_*\) independent of \(y\).
	\end{enumerate}
	Moreover, set
	\[
	B_*:=D_{yy}^2\mathcal S(x_*\mid x_*),
	\qquad
	H_*:=D^2(\Phi\circ\chi^{-1})(x_*).
	\]
	Then
	\begin{equation}
		B_*=\tau_1b_*G_{\FR,*},
		\label{eq:BstarFR}
	\end{equation}
	where \(G_{\FR,*}\) is the matrix of the Gaussian Fisher--Rao metric in the chart, and
	\begin{equation}
		B_*-H_*\succeq0.
		\label{eq:HlessB}
	\end{equation}
	If, in addition, \(\nu_*\) is a local minimizer of \(\Phi\), then \(0\preceq H_*\preceq B_*\).
\end{lemma}

\begin{proof}
	For fixed \(x\), the optimal plans \(\gamma_i(x)\) satisfy
	\[
	\Phi(\nu_y)
	=
	\inf_{b>0}\Fbar(b,\nu_y)
	\le
	\Fbar(b(x),\nu_y)
	\le
	\mathcal S(y\mid x).
	\]
	At \(y=x\), exact mass profiling and pairwise optimality make both inequalities equalities, proving (a).  With the mass and plans fixed, only the terminal reverse-KL terms depend on the shape.  By \Cref{lem:KLcentroid}, their unique Gaussian minimizer is the moment update \eqref{eq:algm}--\eqref{eq:algSigma}.  The resulting minimizer lies in \(V\), which proves (b).
	
	At the fixed point, \eqref{eq:massidentityk} and \Cref{thm:fixedpoint} imply
	\[
	\sum_i\lambda_iM_i^*=b_*,
	\qquad
	\omega_i^*=\frac{\lambda_iM_i^*}{b_*},
	\]
	and \(\nu_*\) is the reverse-KL Gaussian centroid of the adjusted terminal marginals \(q_i^*\).  Applying \eqref{eq:KLPythagorean} reduces this sum to
	\[
	\sum_i\lambda_iM_i^*\KL(q_i^*\mid\nu_y)
	=
	C'_*
	+b_*\KL(\nu_*\mid\nu_y),
	\]
	which is \eqref{eq:MMfixedform}.  At the diagonal, the Hessian of \(y\mapsto\KL(\nu_*\mid\nu_y)\) is the Fisher information matrix, so \(B_*=\tau_1b_*G_{\FR,*}\) as in \eqref{eq:BstarFR}.
	
	Part (a) also shows that the function
	\[
	y\longmapsto \mathcal S(y\mid x_*)-\Phi(\nu_y)
	\]
	is nonnegative and vanishes at \(y=x_*\).  Its Hessian at \(x_*\) is positive semidefinite; equivalently, \(B_*-H_*\succeq0\) as in \eqref{eq:HlessB}.  If \(\nu_*\) is also a local minimizer, then \(H_*\succeq0\), so \(0\preceq H_*\preceq B_*\).
\end{proof}

For a positive-definite matrix \(B\), write \(\norm{h}_{B}^2=h^\top Bh\) and let \(\norm{A}_{B}=\sup_{h\ne0}\norm{Ah}_{B}/\norm{h}_{B}\) denote the induced operator norm.

\begin{theorem}[MM Jacobian and local linear convergence]\label{thm:localrate}
	Let \(\nu_*\) be a fixed point of the moment map.  Choose a smooth local chart \(\chi\) and write \(x_*=\chi(\nu_*)\).  Assume
	\begin{equation}
		H_*=D^2(\Phi\circ\chi^{-1})(x_*)\succ0
		\label{eq:nondegenerateH}
	\end{equation}
	for one (equivalently, every) smooth local chart.  Then \(\nu_*\) is a strict local minimizer.  Let \(B_*:=\tau_1b_*G_{\FR,*}\) as in \Cref{lem:MMsurrogate}.  The generalized spectrum of \((H_*,B_*)\), and hence the quantities \(\mu_*\) and \(\rho_*\) below, is chart independent.  The derivative of the moment fixed-point map satisfies
	\begin{equation}
		D\widehat T(x_*)
		=
		I-B_*^{-1}H_*,
		\label{eq:exactJacobian}
	\end{equation}
	where \(\widehat T=\chi\circ T\circ\chi^{-1}\).  Define
	\begin{equation}
		\mu_*
		:=
		\lambda_{\min}\!\left(B_*^{-1/2}H_*B_*^{-1/2}\right)
		\in(0,1],
		\qquad
		\rho_*:=1-\mu_*\in[0,1),
		\qquad
		\norm{D\widehat T(x_*)}_{B_*}=\rho_*.
		\label{eq:localmu}
	\end{equation}
	For every \(q\in(\rho_*,1)\), there exist a neighborhood \(V\) of \(x_*\) in the chart domain and \(r>0\) such that
	\[
	\Phi(\nu_x)\geq\Phi(\nu_*)\qquad(x\in V),
	\]
	and
	\[
	\overline{\mathcal U}_r:=\{x:\norm{x-x_*}_{B_*}\le r\}\subset V,
	\qquad
	\mathcal U_r:=\{x:\norm{x-x_*}_{B_*}<r\}.
	\]
	The open ball \(\mathcal U_r\) is invariant under \(\widehat T\).  Moreover, whenever \(x_0\in\mathcal U_r\), the iterates satisfy
	\begin{equation}
		\norm{x_k-x_*}_{B_*}
		\le
		q^k\norm{x_0-x_*}_{B_*}.
		\label{eq:localQlinear}
	\end{equation}
	The mass and objective errors satisfy
	\begin{equation}
		|b(\nu_{x_k})-b_*|=O(q^k),
		\qquad
		\Phi(\nu_{x_k})-\Phi(\nu_*)=O(q^{2k}).
		\label{eq:localmassobjective}
	\end{equation}
\end{theorem}

\begin{proof}
	By \Cref{thm:fixedpoint}, the fixed point \(\nu_*\) is stationary.  At \(x_*\), a coordinate change with Jacobian \(J\) sends \(H_*\) and \(B_*\) to \(J^\top H_*J\) and \(J^\top B_*J\), respectively.  Hence positive definiteness of \(H_*\) and the generalized spectrum of \((H_*,B_*)\) are coordinate invariant.  The pairwise interface and explicit Gaussian centroid formulas also show that the local surrogate and moment map are real analytic near \(x_*\).  Their surrogate first-order condition is
	\begin{equation}
		D_y\mathcal S\bigl(\widehat T(x)\mid x\bigr)=0.
	\end{equation}
	For fixed \(x\), the function \(y\mapsto\mathcal S(y\mid x)-\Phi(\nu_y)\) has a zero local minimum at \(y=x\).  Its first derivative therefore vanishes there:
	\begin{equation}
		D_y\mathcal S(x\mid x)=D(\Phi\circ\chi^{-1})(x).
		\label{eq:MMtangency}
	\end{equation}
	Differentiating \eqref{eq:MMtangency} at \(x_*\) splits the Hessian as
	\begin{equation}
		H_*=B_*+C_*,
		\qquad
		C_*:=D_{yx}^2\mathcal S(x_*\mid x_*).
		\label{eq:HBC}
	\end{equation}
	The derivative of the surrogate first-order condition at the fixed point is
	\[
	B_*D\widehat T(x_*)+C_*=0.
	\]
	Since \(B_*\succ0\), combining this identity with \eqref{eq:HBC} proves \eqref{eq:exactJacobian}.
	
	By \Cref{lem:MMsurrogate} and \eqref{eq:nondegenerateH},
	\[
	0\prec H_*\preceq B_*.
	\]
	All eigenvalues of \(\Xi_*:=B_*^{-1/2}H_*B_*^{-1/2}\) belong to \((0,1]\).  In the \(B_*\)-inner product, \(D\widehat T(x_*)\) is similar to the symmetric matrix \(I-\Xi_*\), and hence
	\[
	\norm{D\widehat T(x_*)}_{B_*}
	=
	\max_j|1-\lambda_j(\Xi_*)|
	=
	1-\lambda_{\min}(\Xi_*)
	=
	\rho_*<1,
	\]
	identifying the intrinsic contraction factor in \eqref{eq:localmu}.
	
	Real analyticity of the endpoint formulas, Gibbs weights, and Gaussian moment update makes \(D\widehat T\) continuous.  Since \(H_*\succ0\) and \(\nu_*\) is stationary, the second-order Taylor expansion yields a chart neighborhood \(V\) on which \(\Phi(\nu_x)\geq\Phi(\nu_*)\), with strict inequality for \(x\neq x_*\) sufficiently close.  Fix \(q\in(\rho_*,1)\).  Shrinking \(r>0\) if necessary, require \(\overline{\mathcal U_r}\subset V\) and
	\[
	\sup_{x\in\overline{\mathcal U_r}}\norm{D\widehat T(x)}_{B_*}\le q.
	\]
	Because the ball is convex and \(\widehat T(x_*)=x_*\), the mean-value theorem along the segment from \(x_*\) to \(x\in\mathcal U_r\) implies
	\[
	\norm{\widehat T(x)-x_*}_{B_*}
	\le q\norm{x-x_*}_{B_*}
	<qr<r.
	\]
	The bound keeps \(\widehat T(x)\) inside \(\mathcal U_r\).  Iterating this bound along the orbit proves \eqref{eq:localQlinear}.  Local Lipschitz continuity of the profiled mass proves the first estimate in \eqref{eq:localmassobjective}.  Because the orbit stays in \(V\), stationarity and the local \(C^2\) expansion of \(\Phi\circ\chi^{-1}\) imply
	\[
	0\le \Phi(\nu_{x_k})-\Phi(\nu_*)
	\le C\norm{x_k-x_*}^2,
	\]
	which is the second estimate in \eqref{eq:localmassobjective}.
\end{proof}

\begin{remark}
	The contraction factor is governed by the generalized curvature ratio
	\[
	\mu_*
	=
	\inf_{h\ne0}
	\frac{H_*[h,h]}{\tau_1b_*\,g_{\FR,\nu_*}(h,h)}.
	\]
	Thus \(\mu_*\) is the smallest profiled curvature in the Fisher metric induced by the reverse-KL surrogate and \(\rho_*=\norm{D\widehat T(x_*)}_{B_*}\) is the operator-norm contraction factor.  \Cref{thm:localrate} gives local \(Q\)-linear convergence for every \(q>\rho_*\); an individual orbit may converge faster if it does not activate the slowest spectral mode.
\end{remark}

\section{Large- and small-penalty limits}
\label{sec:limits}

We scale both marginal penalties by a common parameter \(\kappa>0\):
\begin{equation}
	\tau_0^{(\kappa)}=\kappa\bar\tau_0,
	\qquad
	\tau_1^{(\kappa)}=\kappa\bar\tau_1,
	\qquad
	\bar s=\bar\tau_0+\bar\tau_1,
	\qquad
	\theta=\frac{\bar\tau_0}{\bar s}.
\end{equation}
The exponent \(\theta\) is independent of \(\kappa\).  We write \(A_i^{(\kappa)}\), \(\Zfun_\kappa\), \(b_\kappa\), and \(\Phi_\kappa\) for the corresponding quantities.

The analysis combines endpoint expansions uniform on compact sets with localization bounds for global minimizers that are uniform in the penalty scale.  Large penalties yield a Wasserstein limit and a first-order correction, whereas small penalties yield a Chernoff affinity limit.  Localization is essential because finite Wasserstein sublevels may approach singular covariance at finite cost.

\subsection{Large-penalty Wasserstein limit and localization}

For brevity, set \(\mathcal W(p,q):=W_2^2(p,q)\) and \(\mathcal W_i(\nu):=\mathcal W(\mu_i,\nu)\).
On compact positive-definite spectral boxes, inversion and the principal square root are uniformly Lipschitz~\cite{Higham2008}.  We use any fixed matrix norm, and \(C^r(\mathcal K)\) refers to ambient mean--covariance coordinates.

\begin{lemma}\label{lem:largeanalyticextension}
	Fix an input \(\mu_i=\N(m_i,\Sigma_i)\) and a compact Gaussian parameter set \(\mathcal K\Subset\mathcal G_d^1\).  Put \(\varepsilon=\kappa^{-1}\).  There exist an open neighborhood \(U\supset\mathcal K\) and \(\varepsilon_0>0\) such that the endpoint formulas
	\[
	(\varepsilon,\nu)\longmapsto
	\bigl(L_i^{(1/\varepsilon)},P_i^{(1/\varepsilon)},Q_i^{(1/\varepsilon)},
	u_i^{(1/\varepsilon)},v_i^{(1/\varepsilon)}\bigr),
	\qquad \varepsilon>0,
	\]
	extend real analytically to \((-\varepsilon_0,\varepsilon_0)\times U\), with value
	\[
	\bigl(L_i^W(\nu),\Sigma_i,\Sigma,m_i,m\bigr)
	\]
	at \(\varepsilon=0\), where \(L_i^W(\nu)\) is the Gaussian Wasserstein map from \(\mu_i\) to \(\nu=\N(m,\Sigma)\).  In particular,
	\begin{align}
		L_i^{(\kappa)}&=L_i^{W}(\nu)+O(\kappa^{-1}),
		&
		P_i^{(\kappa)}&=\Sigma_i+O(\kappa^{-1}),\notag\\
		Q_i^{(\kappa)}&=\Sigma+O(\kappa^{-1}),
		&
		u_i^{(\kappa)}&=m_i+O(\kappa^{-1}),
		\qquad
		v_i^{(\kappa)}=m+O(\kappa^{-1}).
		\label{eq:largeparameterexpansion}
	\end{align}
	For \(\varepsilon>0\), write \(p_i^\varepsilon:=p_i^{(1/\varepsilon)}\) and \(q_i^\varepsilon:=q_i^{(1/\varepsilon)}\).  The endpoint value
	\begin{equation}
		\widehat A_i(\varepsilon,\nu)
		:=W_2^2(p_i^\varepsilon,q_i^\varepsilon)
		+\frac{\bar\tau_0}{\varepsilon}\KL(p_i^\varepsilon\mid\mu_i)
		+\frac{\bar\tau_1}{\varepsilon}\KL(q_i^\varepsilon\mid\nu)
		\label{eq:analyticAhat}
	\end{equation}
	also extends real analytically through \(\varepsilon=0\), with \(\widehat A_i(0,\nu)=\mathcal W_i(\nu)\).  For every fixed integer \(r\ge0\), Taylor remainders in \(\varepsilon\) are uniform in the \(C^r(\mathcal K)\) norm, and in particular
	\begin{equation}
		\sup_{\nu\in\mathcal K}
		\left|A_i^{(\kappa)}(\nu)-W_2^2(\mu_i,\nu)\right|
		=O(\kappa^{-1}).
		\label{eq:largeA}
	\end{equation}
\end{lemma}

\begin{proof}
	In \eqref{eq:Sstar}--\eqref{eq:uvstar}, \(r_j=2\varepsilon/\bar\tau_j\) and \(\delta\) are analytic in \(\varepsilon\).  At \(\varepsilon=0\), all matrices entering inverses and principal square roots are positive definite; compactness of \(\mathcal K\) allows one common neighborhood on which these operations remain real analytic~\cite{Higham2008}.  Evaluating the covariance equation at this limit reduces it to \(L_0\Sigma^{-1}L_0=\Sigma_i^{-1}\), whose positive-definite solution is the Gaussian Wasserstein map \(L_i^W(\nu)\); the remaining endpoint limits and \eqref{eq:largeparameterexpansion} follow directly.
	
	For \eqref{eq:analyticAhat}, set
	\[
	R_{0,i}(\varepsilon,\nu)=\KL(p_i^\varepsilon\mid\mu_i),
	\qquad
	R_{1,i}(\varepsilon,\nu)=\KL(q_i^\varepsilon\mid\nu).
	\]
	After the endpoint parameters have been analytically extended, both \(R_{j,i}\) are jointly real analytic near \(\{0\}\times\mathcal K\).  At \(\varepsilon=0\), one has \(p_i^0=\mu_i\) and \(q_i^0=\nu\), so \(R_{j,i}(0,\nu)=0\).  The first \(\varepsilon\)-derivatives also vanish because the differential of Gaussian relative entropy with respect to its moving argument is zero on the diagonal.  Taylor's formula with integral remainder gives the explicit factorization
	\[
	R_{j,i}(\varepsilon,\nu)=\varepsilon^2\widetilde R_{j,i}(\varepsilon,\nu),
	\qquad
	\widetilde R_{j,i}(\varepsilon,\nu)
	=\int_0^1(1-t)\,\partial_\varepsilon^2R_{j,i}(t\varepsilon,\nu)\,dt.
	\]
	The integrand is jointly analytic, so \(\widetilde R_{j,i}\) is jointly analytic as well.
	Consequently \(\varepsilon^{-1}R_{j,i}=\varepsilon\widetilde R_{j,i}\) extends analytically through zero.  The transport term is already jointly analytic and reduces at \(\varepsilon=0\) to \(\mathcal W_i(\nu)\).  Hence \(\widehat A_i\) is jointly analytic on a neighborhood of \(\{0\}\times\mathcal K\); compactness then gives the stated uniform \(C^r(\mathcal K)\) Taylor bounds and \eqref{eq:largeA}.
\end{proof}

This analytic continuation provides the expansion at \(\varepsilon=0\), while the transport problem itself corresponds to \(\varepsilon>0\).  Substituting the endpoint expansion into the exact mass profile gives the limiting barycenter functional.  Set
\begin{equation}
	Z_0=\sum_i\lambda_i a_i^\theta,
	\qquad
	\widetilde\lambda_i
	:=\pi_i
	=
	\frac{\lambda_i a_i^\theta}{Z_0}.
	\label{eq:largeweights}
\end{equation}

\begin{theorem}\label{thm:largelimit}
	For every compact \(\mathcal K\Subset\mathcal G_d^1\), uniformly for \(\nu\in\mathcal K\),
	\begin{equation}
		\Phi_\kappa(\nu)=\kappa\bar\tau_0\left(\bar a-Z_0^{1/\theta}\right)
		+Z_0^{1/\theta-1}\sum_{i=1}^n\lambda_i a_i^\theta W_2^2(\mu_i,\nu)+O(\kappa^{-1}).
		\label{eq:largeprofileexpansion}
	\end{equation}
	Define the shifted profiled functional
	\begin{equation}
		\widetilde\Phi_\kappa(\nu)
		=
		\Phi_\kappa(\nu)
		-\kappa\bar\tau_0\left(\bar a-Z_0^{1/\theta}\right).
		\label{eq:shiftedlargeprofile}
	\end{equation}
	Then \(\widetilde\Phi_\kappa\) converges uniformly on \(\mathcal K\) to
	\begin{equation}
		\Phi_\infty(\nu)
		=
		Z_0^{1/\theta}
		\sum_{i=1}^n\widetilde\lambda_iW_2^2(\mu_i,\nu).
		\label{eq:BWlimitfunctional}
	\end{equation}
	Moreover, the profiled mass satisfies
	\begin{equation}
		b_\kappa(\nu)\longrightarrow Z_0^{1/\theta}
		\label{eq:largemasslimit}
	\end{equation}
	uniformly on \(\mathcal K\).
\end{theorem}

\begin{proof}
	\Cref{lem:largeanalyticextension} gives the uniform expansion
	\[
	e^{-A_i^{(\kappa)}(\nu)/(\kappa\bar s)}
	=
	1-\frac{W_2^2(\mu_i,\nu)}{\kappa\bar s}
	+O(\kappa^{-2}).
	\]
	Substituting this expansion into \(\Zfun_\kappa\) produces the explicit expansion
	\begin{equation}
		\Zfun_\kappa(\nu)
		=
		Z_0-\frac{1}{\kappa\bar s}
		\sum_i\lambda_i a_i^\theta W_2^2(\mu_i,\nu)
		+O(\kappa^{-2}).
	\end{equation}
	Taylor expansion of \(x\mapsto x^{1/\theta}\) about \(Z_0\) gives
	\begin{align*}
		\Zfun_\kappa(\nu)^{1/\theta}
		={}&
		Z_0^{1/\theta}
		-
		\frac{1}{\theta\kappa\bar s}
		Z_0^{1/\theta-1}
		\sum_i\lambda_i a_i^\theta W_2^2(\mu_i,\nu)
		+O(\kappa^{-2}).
	\end{align*}
	Now \(\theta\bar s=\bar\tau_0\), and
	\[
	\Phi_\kappa(\nu)
	=
	\kappa\bar\tau_0
	\left(\bar a-\Zfun_\kappa(\nu)^{1/\theta}\right),
	\]
	which is \eqref{eq:largeprofileexpansion}.  Applying the same expansion to \(b_\kappa=\Zfun_\kappa^{1/\theta}\) proves \eqref{eq:largemasslimit}.
\end{proof}

Uniform convergence on compact sets does not control global minimizers without two additional ingredients: nondegeneracy of the Wasserstein limit and localization bounds uniform in the penalty scale.  These are provided by \Cref{lem:BWlimitHessian,lem:largelocalization}.

\begin{lemma}[Hessian of the Wasserstein limit]\label{lem:BWlimitHessian}
	In global mean--covariance coordinates \((m,\Sigma)\in\R^d\times\Spp^d\), the limit functional \(\Phi_\infty\) in \eqref{eq:BWlimitfunctional} has positive-definite Euclidean Hessian at every point.  Hence, viewed as a function of \((m,\Sigma)\), it is strictly convex on the convex affine parameter domain \(\R^d\times\Spp^d\).  Every critical point is nondegenerate and any minimizer is unique.  Moreover, for every compact \(\mathcal K\Subset\mathcal G_d^1\) there exists \(\mu_{\mathcal K}>0\) such that \(D^2\Phi_\infty(\nu)\succeq\mu_{\mathcal K}I\) throughout \(\mathcal K\) in mean--covariance coordinates.  At a critical point, positive definiteness also holds in every smooth local chart.
\end{lemma}

\begin{proof}
	It suffices to prove pointwise Euclidean Hessian positivity; strict convexity then follows because \(\R^d\times\Spp^d\) is convex.  The mean and covariance variables separate in \(W_2^2\), and the mean block of \(D^2\Phi_\infty\) is \(2Z_0^{1/\theta}I\).  For the covariance block, fix \(A\in\Spp^d\) and set
	\[
	F_A(X)=d_{\BW}^2(A,X)
	=\tr A+\tr X-2\tr\!\left(A^{1/2}XA^{1/2}\right)^{1/2}.
	\]
	For \(H\in\Sym^d\), diagonalize
	\[
	Y=A^{1/2}XA^{1/2}=U\operatorname{diag}(\lambda_1,\ldots,\lambda_d)U^\top,
	\qquad
	\widehat H=U^\top A^{1/2}HA^{1/2}U.
	\]
	For a trace spectral function, the second-derivative formula~\cite{LewisSendov2001} reads
	\[
	D^2\tr f(Y)[K,K]
	=
	\sum_i f''(\lambda_i)K_{ii}^2
	+2\sum_{i<j}
	\frac{f'(\lambda_i)-f'(\lambda_j)}{\lambda_i-\lambda_j}
	K_{ij}^2,
	\]
	with the quotient understood by continuity when eigenvalues coincide.  For \(f(t)=\sqrt t\),
	\[
	f''(\lambda_i)=-\frac{1}{4\lambda_i^{3/2}},
	\qquad
	\frac{f'(\lambda_i)-f'(\lambda_j)}{\lambda_i-\lambda_j}
	=-\frac{1}{2\sqrt{\lambda_i\lambda_j}(\sqrt{\lambda_i}+\sqrt{\lambda_j})}.
	\]
	Since \(F_A(X)=\tr A+\tr X-2\tr f(Y)\), the affine terms have zero Hessian and therefore
	\[
	D^2F_A(X)[H,H]
	=
	\sum_{i=1}^d\frac{\widehat H_{ii}^2}{2\lambda_i^{3/2}}
	+
	2\sum_{i<j}
	\frac{\widehat H_{ij}^2}
	{\sqrt{\lambda_i\lambda_j}(\sqrt{\lambda_i}+\sqrt{\lambda_j})}.
	\]
	Every coefficient is positive, and \(A^{1/2}HA^{1/2}=0\) only when \(H=0\) because \(A^{1/2}\) is invertible.  Therefore \(D^2F_A(X)[H,H]>0\) for every nonzero \(H\).  Positive weights and the absence of mean--covariance cross terms make \(D^2\Phi_\infty\) positive definite on \(\R^d\times\Spp^d\).  Continuity of the Hessian on a compact set makes its smallest eigenvalue uniformly positive, proving the asserted lower bound on every compact \(\mathcal K\Subset\mathcal G_d^1\).
	
	At a critical point the gradient vanishes, so under a smooth coordinate change the Hessian transforms by congruence.  Positive definiteness is therefore chart independent there and supplies the limit-side nondegeneracy used in \Cref{thm:largebranch}.
\end{proof}

\begin{lemma}[Large-penalty localization]\label{lem:largelocalization}
	Let
	\[
	\nu_\kappa\in
	\operatorname*{argmin}_{\nu\in\mathcal G_d^1}\Phi_\kappa(\nu)
	\]
	be any global Gaussian minimizer.  There exist \(\kappa_0<\infty\) and constants \(R<\infty\), \(0<\underline\sigma\le\Lambda<\infty\), independent of the choice of minimizer and of \(\kappa\ge\kappa_0\), such that
	\begin{equation}
		\norm{m_\kappa}\le R,
		\qquad
		\underline\sigma I\preceq\Sigma_\kappa\preceq\Lambda I.
		\label{eq:largegloballocalization}
	\end{equation}
\end{lemma}

\begin{proof}
	Write \(c_i=\lambda_i a_i^\theta\), so \(Z_0=\sum_i c_i\).  Fix one reference Gaussian \(\nu_{\rm ref}\).  By \Cref{lem:largeanalyticextension} on the singleton \(\{\nu_{\rm ref}\}\),
	\[
	\Zfun_\kappa(\nu_{\rm ref})
	\ge Z_0-\frac{C}{\kappa}
	\]
	for large \(\kappa\).  Since minimizing \(\Phi_\kappa\) is equivalent to maximizing \(\Zfun_\kappa\),
	\[
	0\le Z_0-\Zfun_\kappa(\nu_\kappa)\le \frac{C}{\kappa}.
	\]
	Each summand in
	\[
	Z_0-\Zfun_\kappa(\nu_\kappa)
	=\sum_i c_i
	\left[
	1-\exp\!\left(-\frac{A_i^{(\kappa)}(\nu_\kappa)}{\kappa\bar s}\right)
	\right]
	\]
	is nonnegative, and for every \(i\),
	\[
	1-\exp\!\left(-\frac{A_i^{(\kappa)}(\nu_\kappa)}{\kappa\bar s}\right)
	\le \frac{C_i}{\kappa}.
	\]
	After increasing \(\kappa_0\), the right-hand side is at most \(1/2\), and \(-\log(1-x)\le2x\) converts this into an endpoint bound uniform in \(\kappa\)
	\begin{equation}
		A_i^{(\kappa)}(\nu_\kappa)\le C_i'
		\qquad(i=1,\ldots,n,\ \kappa\ge\kappa_0).
		\label{eq:largeendpointglobalbound}
	\end{equation}
	The proof of \Cref{lem:uppercompact} is uniform for \(\kappa\ge\kappa_0\): indeed \eqref{eq:largeendpointglobalbound} directly bounds \(W_2^2(p_i,q_i)\) by \(C_i'\),
	\[
	\KL(p_i\mid\mu_i)\le \frac{C_i'}{\kappa_0\bar\tau_0},
	\qquad
	\KL(q_i\mid\nu_\kappa)\le \frac{C_i'}{\kappa_0\bar\tau_1}.
	\]
	The same argument therefore bounds the means and largest covariance eigenvalues uniformly by common constants \(R,\Lambda\), as required in \eqref{eq:largegloballocalization}.  The first KL bound also places the adjusted source covariance in a fixed KL sublevel, so
	\[
	P_i^{(\kappa)}\succeq \underline p_i I
	\]
	with \(\underline p_i>0\) independent of \(\kappa\).  Let \(\underline p=\min_i\underline p_i\).
	
	Every global minimizer is stationary for the smooth profiled functional, hence satisfies the fixed-point covariance equation \eqref{eq:fixedSigma}.  Put \(t_\kappa=\lambda_{\min}(\Sigma_\kappa)\).  Apply \Cref{lem:spectraltransfer} with the scaled target penalty \(\tau_1^{(\kappa)}=\kappa\bar\tau_1\).  If \(t_\kappa<\underline p\), then
	\[
	\lambda_{\min}(Q_i^{(\kappa)})>t_\kappa
	\qquad\text{for every }i.
	\]
	The positive-semidefinite between-mean term in \eqref{eq:fixedSigma} would then imply
	\[
	t_\kappa
	\ge
	\sum_i\omega_i(\nu_\kappa)\lambda_{\min}(Q_i^{(\kappa)})
	>t_\kappa.
	\]
	This contradiction rules out \(t_\kappa<\underline p\).  Hence \(t_\kappa\ge\underline p\), completing \eqref{eq:largegloballocalization}.
\end{proof}

\begin{theorem}[Large-penalty barycenter limit]\label{thm:largeglobal}
	Let \(\nu_{\BW}\) denote the Gaussian Wasserstein barycenter of \(\mu_1,\ldots,\mu_n\) with weights \(\widetilde\lambda_i\) from \eqref{eq:largeweights}.  Its existence is classical~\cite{AguehCarlier2011,AlvarezEsteban2016}, while uniqueness follows from \Cref{lem:BWlimitHessian}.  Choose any \(\nu_\kappa\in\operatorname*{argmin}_{\nu\in\mathcal G_d^1}\Phi_\kappa(\nu)\).  Then
	\begin{equation}
		\nu_\kappa\longrightarrow\nu_{\BW}
		\qquad(\kappa\to\infty).
		\label{eq:largeglobalshape}
	\end{equation}
	Moreover,
	\begin{equation}
		b_\kappa(\nu_\kappa)\longrightarrow Z_0^{1/\theta},
		\label{eq:largeglobalmass}
	\end{equation}
	and
	\begin{equation}
		\min_{\nu\in\mathcal G_d^1}\widetilde\Phi_\kappa(\nu)
		\longrightarrow
		\Phi_\infty(\nu_{\BW}).
		\label{eq:largeglobalvalue}
	\end{equation}
\end{theorem}

\begin{proof}
	By \Cref{lem:largelocalization}, every large-penalty global minimizer belongs to one compact set \(\mathcal K_*\Subset\mathcal G_d^1\).  On \(\mathcal K_*\), \Cref{thm:largelimit} establishes uniform convergence \(\widetilde\Phi_\kappa\to\Phi_\infty\).  Every cluster point of a minimizing sequence therefore minimizes \(\Phi_\infty\).  By \Cref{lem:BWlimitHessian}, \(\Phi_\infty\) is strictly convex in global mean--covariance coordinates, so its minimizer is unique; hence every cluster point equals \(\nu_{\BW}\), proving \eqref{eq:largeglobalshape}.  The mass limit \eqref{eq:largemasslimit} on \(\mathcal K_*\) gives \eqref{eq:largeglobalmass}, and uniform convergence gives the limiting minimum value in \eqref{eq:largeglobalvalue}.
\end{proof}

\begin{remark}
	Along \(\nu_t=\N(0,tI)\) with \(t\downarrow0\), each \(W_2^2(\mu_i,\nu_t)\) has a finite limit.  Finite sublevels of \(\Phi_\infty\) may therefore approach singular covariances even though the global minimizers remain in a compact subset of the open Gaussian manifold \(\mathcal G_d^1\).  Rank-changing limits naturally lead to a completion containing singular Gaussians.
\end{remark}

\subsection{First-order expansion and eventual uniqueness}

The zeroth-order limit leaves two questions: the first correction and eventual uniqueness.  Both are resolved by the analytic endpoint expansion.

\begin{proposition}\label{prop:largefirstcorrection}
	For a covector \(\ell\in T_\mu^*\mathcal G_d^1\), write \(\norm{\ell}_{\FR,\mu,*}\) for the dual norm induced by the Gaussian Fisher--Rao metric \eqref{eq:FRmetric}.  Define
	\begin{equation}
		\mathcal C_i(\nu)
		:=
		-\frac{1}{2\bar\tau_0}
		\norm{D_1\mathcal W(\mu_i,\nu)}_{\FR,\mu_i,*}^2
		-\frac{1}{2\bar\tau_1}
		\norm{D_2\mathcal W(\mu_i,\nu)}_{\FR,\nu,*}^2.
		\label{eq:largeendpointcorrection}
	\end{equation}
	Then, on every compact \(\mathcal K\Subset\mathcal G_d^1\), the endpoint admits the expansion
	\begin{equation}
		A_i^{(\kappa)}(\nu)
		=\mathcal W_i(\nu)+\frac{1}{\kappa}\mathcal C_i(\nu)
		+O(\kappa^{-2}),
		\label{eq:largeAsecond}
	\end{equation}
	where the remainder is uniform together with derivatives in \(\nu\) up to order two.
\end{proposition}

\begin{proof}
	Set \(\varepsilon=\kappa^{-1}\) and use the global mean--covariance coordinates on the open set \(\R^d\times\Spp^d\).  Write \(\operatorname{par}(\N(u,P))=(u,P)\).  By \Cref{lem:largeanalyticextension}, the unique adjusted marginals have parameter expansions
	\[
	\operatorname{par}(p_i^\varepsilon)
	=\operatorname{par}(\mu_i)+\varepsilon x_*+O(\varepsilon^2),
	\qquad
	\operatorname{par}(q_i^\varepsilon)
	=\operatorname{par}(\nu)+\varepsilon y_*+O(\varepsilon^2),
	\]
	where \(x_*\in T_{\mu_i}\mathcal G_d^1\) and \(y_*\in T_\nu\mathcal G_d^1\) under these coordinates, with coefficients depending smoothly (indeed analytically) on \(\nu\).  For arbitrary tangent displacements \((x,y)\), the coefficient of \(\varepsilon\) in the endpoint objective is
	\begin{equation}
		\mathcal Q_i^{(1)}(x,y)
		=D_1\mathcal W(\mu_i,\nu)[x]
		+D_2\mathcal W(\mu_i,\nu)[y]
		+\frac{\bar\tau_0}{2}g_{\FR,\mu_i}(x,x)
		+\frac{\bar\tau_1}{2}g_{\FR,\nu}(y,y).
	\end{equation}
	Indeed, the transport term contributes its two first variations, whereas the diagonal expansion \eqref{eq:KLlocal} contributes the Fisher--Rao quadratic forms after multiplication by \(\varepsilon^{-1}\bar\tau_j\).  The linearized stationarity equations make this identification explicit.  For every \(\xi\in T_{\mu_i}\mathcal G_d^1\) and \(\eta\in T_\nu\mathcal G_d^1\), differentiation of the two endpoint first-order equations at \(\varepsilon=0\) gives
	\begin{align}
		D_1\mathcal W(\mu_i,\nu)[\xi]
		+\bar\tau_0 g_{\FR,\mu_i}(x_*,\xi)&=0,
		\label{eq:large-linearized-source}\\
		D_2\mathcal W(\mu_i,\nu)[\eta]
		+\bar\tau_1 g_{\FR,\nu}(y_*,\eta)&=0.
		\label{eq:large-linearized-target}
	\end{align}
	Here the Hessian of the moving-argument Gaussian KL divergence on the diagonal is precisely the Fisher--Rao metric by \eqref{eq:KLlocal}.  Thus \(x_*\) and \(y_*\) are the negative Fisher--Rao Riesz representatives of \(D_1\mathcal W/\bar\tau_0\) and \(D_2\mathcal W/\bar\tau_1\), respectively.  Equations \eqref{eq:large-linearized-source}--\eqref{eq:large-linearized-target} are also the Euler equations for minimizing \(\mathcal Q_i^{(1)}\).  Its strict convexity makes the minimizer unique, and completing the square in each variable evaluates the two infima as
	\[
	\inf_x\left\{D_1\mathcal W[x]+\frac{\bar\tau_0}{2}g_{\FR}(x,x)\right\}
	=-\frac{1}{2\bar\tau_0}\norm{D_1\mathcal W}_{\FR,*}^2,
	\]
	and the analogous identity in the second slot.  The coefficient of \(\varepsilon\) is exactly \(\mathcal C_i\) in \eqref{eq:largeendpointcorrection}.
	
	This derivation is coordinate independent because the covectors and Fisher--Rao Riesz representatives in \eqref{eq:large-linearized-source}--\eqref{eq:large-linearized-target} are intrinsic.  Joint analyticity provides uniform control of the differentiated remainder on compact sets.  By \Cref{lem:largeanalyticextension}, \(\widehat A_i(\varepsilon,\nu)\) is real analytic on a neighborhood of \(\{0\}\times\mathcal K\).  Taylor's theorem in \(\varepsilon\), with derivatives in \(\nu\) taken in the fixed ambient mean--covariance coordinates, bounds the remainder uniformly:
	\[
	\left\|\widehat A_i(\varepsilon,\cdot)-\mathcal W_i-\varepsilon\mathcal C_i\right\|_{C^2(\mathcal K)}
	\le C_{\mathcal K}\varepsilon^2.
	\]
	Substituting \(\varepsilon=\kappa^{-1}\) proves \eqref{eq:largeAsecond} with the claimed uniform \(C^2\) remainder.
\end{proof}

\begin{corollary}\label{cor:largeprofilecorrection}
	Set \(p_\theta=1/\theta\), \(c_i=\lambda_i a_i^\theta\), and
	\begin{equation}
		M_1(\nu)=\sum_i c_i\mathcal W_i(\nu),
		\qquad
		M_2(\nu)=\sum_i c_i\mathcal W_i(\nu)^2,
		\qquad
		\mathcal C(\nu)=\sum_i c_i\mathcal C_i(\nu).
	\end{equation}
	Then the profiled functional admits the following \(C^2\)-uniform expansion on compact Gaussian sets:
	\begin{equation}
		\widetilde\Phi_\kappa(\nu)
		=\Phi_\infty(\nu)
		+\frac{1}{\kappa}\Phi_1(\nu)
		+O(\kappa^{-2}),
		\label{eq:largeprofilefirstcorrection}
	\end{equation}
	where
	\begin{equation}
		\Phi_1(\nu)=Z_0^{p_\theta-1}\mathcal C(\nu)-\frac{Z_0^{p_\theta-1}}{2\bar s}M_2(\nu)-\frac{p_\theta-1}{2\bar s}Z_0^{p_\theta-2}M_1(\nu)^2.
		\label{eq:Phi1}
	\end{equation}
	The profiled mass has the expansion
	\begin{equation}
		b_\kappa(\nu)
		=Z_0^{1/\theta}
		-\frac{Z_0^{1/\theta-1}}{\bar\tau_0\kappa}M_1(\nu)
		+O(\kappa^{-2}).
		\label{eq:largemassfirstcorrection}
	\end{equation}
\end{corollary}

\begin{proof}
	\Cref{prop:largefirstcorrection} yields the second-order exponential expansion
	\[
	\exp\!\left[-\frac{A_i^{(\kappa)}}{\kappa\bar s}\right]
	=1-\frac{\mathcal W_i}{\kappa\bar s}
	+\frac{1}{\kappa^2}
	\left(\frac{\mathcal W_i^2}{2\bar s^2}-\frac{\mathcal C_i}{\bar s}\right)
	+O(\kappa^{-3})
	\]
	in \(C^2\) on compact sets.  After summing over \(i\), \(\Zfun_\kappa\) has the corresponding second-order expansion.  Taylor expanding \(z\mapsto z^{p_\theta}\), using \(\bar\tau_0p_\theta=\bar s\), and substituting into \eqref{eq:shiftedlargeprofile} establishes \eqref{eq:largeprofilefirstcorrection}--\eqref{eq:Phi1}.  Applying the same expansion to \(\Zfun_\kappa^{1/\theta}\) establishes \eqref{eq:largemassfirstcorrection}.
\end{proof}

Here \(\mathcal C\) is the correction due to endpoint relaxation, while \(M_2\) and \(M_1^2\) arise, respectively, from the Gibbs exponential and the nonlinear mass profile.

\begin{theorem}[Large-penalty uniqueness and analytic branch]\label{thm:largebranch}
	Let \(x\) be any real-analytic local chart near the Gaussian Wasserstein barycenter \(\nu_{\BW}\) (for example, mean--log-covariance coordinates), and write \(x_{\BW}\) for its coordinate.  Set
	\begin{equation}
		H_\infty
		:=D_x^2(\Phi_\infty\circ x^{-1})(x_{\BW})\succ0,
		\label{eq:largeHnondegenerate}
	\end{equation}
	where positive definiteness follows from \Cref{lem:BWlimitHessian}.  Then there exists \(\kappa_*<\infty\) such that for every \(\kappa\ge\kappa_*\) the Gaussian-restricted KL-UOT barycenter shape over \(\mathcal G_d^1\) is unique.  Since the mass profile is unique for each shape by \Cref{thm:massprofile}, the corresponding finite Gaussian barycenter is also unique.  Its coordinate \(x_\kappa\) depends real analytically on \(\varepsilon=\kappa^{-1}\) near \(0\) and satisfies
	\begin{equation}
		x_\kappa=x_{\BW}-\frac{1}{\kappa}H_\infty^{-1}\nabla_x\bigl(\Phi_1\circ x^{-1}\bigr)(x_{\BW})+O(\kappa^{-2}).
		\label{eq:largebranchcorrection}
	\end{equation}
	The corresponding mass satisfies
	\begin{equation}
		b_\kappa(\nu_\kappa)
		=Z_0^{1/\theta}
		-\frac{Z_0^{1/\theta-1}}{\bar\tau_0\kappa}
		M_1(\nu_{\BW})
		+O(\kappa^{-2}).
		\label{eq:largebranchmass}
	\end{equation}
\end{theorem}

\begin{proof}
	For all sufficiently large \(\kappa\), \Cref{lem:largelocalization} confines every global minimizer to a common set that may be enlarged to the compact mean--covariance box
	\[
	\mathcal K_*=
	\left\{\N(m,\Sigma):\norm m\le R_*,\ 
	\underline\sigma_*I\preceq\Sigma\preceq\overline\sigma_*I\right\}
	\Subset\mathcal G_d^1.
	\]
	This box is convex in the ambient coordinates \((m,\Sigma)\).  On it, \Cref{lem:BWlimitHessian} supplies \(\mu_*>0\) with \(D^2\Phi_\infty\succeq\mu_*I\).  The \(C^2\)-uniform expansion in \Cref{cor:largeprofilecorrection} therefore implies, after increasing \(\kappa\) if necessary,
	\[
	D^2\widetilde\Phi_\kappa\succeq\frac{\mu_*}{2}I
	\qquad\text{on }\mathcal K_*.
	\]
	The shifted functional \(\widetilde\Phi_\kappa\) is strongly convex on the localization box; the same is true of \(\Phi_\kappa\) because the two differ by a constant.  Localization and strong convexity on \(\mathcal K_*\) therefore imply uniqueness of the global minimizer for all sufficiently large \(\kappa\).
	
	It remains to identify this minimizer with an analytic branch and compute its first correction.  By \Cref{lem:largeanalyticextension,cor:largeprofilecorrection}, the shifted objective in local coordinates, \((\varepsilon,x)\mapsto\widetilde\Phi_{1/\varepsilon}(x)\), extends real analytically through \(\varepsilon=0\) with value \(\Phi_\infty(x)\).  Indeed, \(\Zfun_{1/\varepsilon}(x)^{1/\theta}-Z_0^{1/\theta}\) is analytic in \((\varepsilon,x)\) and vanishes at \(\varepsilon=0\); analytic division by \(\varepsilon\) removes the apparent singularity in \eqref{eq:shiftedlargeprofile}.  Its stationarity map is analytic, and its derivative in \(x\) at \((0,x_{\BW})\) is the positive-definite matrix \(H_\infty\) from \eqref{eq:largeHnondegenerate}.  By the analytic implicit-function theorem, there is a unique analytic critical branch \(x(\varepsilon)\) near \(x_{\BW}\).
	
	By \Cref{thm:largeglobal}, the unique global minimizer converges to \(\nu_{\BW}\), so for all sufficiently large \(\kappa\) it lies in the implicit-function neighborhood and coincides with \(x(\kappa^{-1})\).  Differentiating the stationarity equation at \(\varepsilon=0\) and inserting \eqref{eq:largeprofilefirstcorrection} gives the first correction in \eqref{eq:largebranchcorrection}.  Finally, evaluating \eqref{eq:largemassfirstcorrection} at \(x_\kappa=x_{\BW}+O(\kappa^{-1})\) establishes \eqref{eq:largebranchmass}.
\end{proof}

The first-order coefficient in \eqref{eq:largebranchcorrection} is intrinsic: because \(d\Phi_\infty(\nu_{\BW})=0\), \(-H_\infty^{-1}d\Phi_1\) defines a tangent vector independently of the chosen chart.

In the equal-mass case \(a_i=a\) for every \(i\), one has \(\widetilde\lambda_i=\lambda_i\), \(Z_0^{1/\theta}=a\), and the limiting shape functional is \(a\sum_i\lambda_iW_2^2(\mu_i,\nu)\); thus the shape converges to the ordinary Gaussian Wasserstein barycenter and the mass to \(a\).

\subsection{Small-penalty Chernoff limit}

As \(\kappa\downarrow0\), the adjusted endpoint marginals coalesce and the shape criterion becomes affinity based.  For positive densities \(f_\mu,f_\nu\), define the order-\(\theta\) Chernoff affinity~\cite{Chernoff1952}
\begin{equation}
	\rhoCher(\mu,\nu)
	=
	\int_{\R^d}f_\mu(x)^\theta f_\nu(x)^{1-\theta}\,dx.
\end{equation}
For nondegenerate Gaussians, \(0<\rhoCher\leq1\).

\begin{lemma}\label{lem:weightedKL}
	Let \(\mu,\nu\in\mathcal G_d^1\), and define the probability measure \(\zeta_{\mu,\nu}\) by
	\begin{equation}
		f_{\zeta_{\mu,\nu}}
		=
		\frac{f_\mu^\theta f_\nu^{1-\theta}}{\rhoCher(\mu,\nu)}.
		\label{eq:geommeanmeasure}
	\end{equation}
	Then
	\begin{equation}
		\inf_{\eta\in\mathcal P(\R^d)}
		\left\{
		\bar\tau_0\KL(\eta\mid\mu)
		+\bar\tau_1\KL(\eta\mid\nu)
		\right\}
		=
		-\bar s\log\rhoCher(\mu,\nu),
		\label{eq:weightedKLmin}
	\end{equation}
	with unique minimizer \(\eta=\zeta_{\mu,\nu}\); moreover \(\zeta_{\mu,\nu}\in\mathcal G_d^1\).
\end{lemma}

\begin{proof}
	For finite-KL competitors, the standard weighted entropy identity is
	\[
	\theta\KL(\eta\mid\mu)+(1-\theta)\KL(\eta\mid\nu)
	=
	\KL(\eta\mid\zeta_{\mu,\nu})-\log\rhoCher(\mu,\nu),
	\]
	while competitors with an infinite term are irrelevant.  Multiplication by \(\bar s\) proves \eqref{eq:weightedKLmin}; strict positivity of KL away from the diagonal makes \(\eta=\zeta_{\mu,\nu}\) the unique minimizer.  The density in \eqref{eq:geommeanmeasure} is Gaussian because the product \(f_\mu^\theta f_\nu^{1-\theta}\) is proportional to a nondegenerate Gaussian density.
\end{proof}

For Gaussian inputs \(\mu=\N(m_0,\Sigma_0)\) and \(\nu=\N(m_1,\Sigma_1)\), the weighted geometric mean \(\zeta_{\mu,\nu}\) has precision and precision-weighted mean
\begin{equation}
	\Lambda_\theta
	=\theta\Sigma_0^{-1}+(1-\theta)\Sigma_1^{-1},
	\qquad
	h_\theta
	=\theta\Sigma_0^{-1}m_0+(1-\theta)\Sigma_1^{-1}m_1,
\end{equation}
so
\[
\zeta_{\mu,\nu}
=\N\!\left(\Lambda_\theta^{-1}h_\theta,\Lambda_\theta^{-1}\right).
\]
Writing \(D_\theta=(1-\theta)\Sigma_0+\theta\Sigma_1\), the same completion of the square yields the affinity
\begin{equation}
	\rhoCher(\mu,\nu)=\frac{\det(\Sigma_0)^{(1-\theta)/2}\det(\Sigma_1)^{\theta/2}}{\det(D_\theta)^{1/2}}
	\exp\!\left[-\frac{\theta(1-\theta)}{2}(m_0-m_1)^\top D_\theta^{-1}(m_0-m_1)\right].
	\label{eq:ChernoffGaussianGeneral}
\end{equation}
The symmetric specialization \(\theta=1/2\) leads to the Bhattacharyya formula in Corollary~\ref{cor:bhattacharyya}.  This weighted entropy identity is precisely the leading-order endpoint problem under the small-penalty scaling.

\begin{lemma}\label{lem:smallendpoint}
	For every compact \(\mathcal K\Subset\mathcal G_d^1\),
	\begin{equation}
		\sup_{\nu\in\mathcal K}
		\left|
		\frac{A_i^{(\kappa)}(\nu)}{\kappa\bar s}
		+\log\rhoCher(\mu_i,\nu)
		\right|
		\longrightarrow0
		\qquad
		(\kappa\downarrow0).
		\label{eq:smallA}
	\end{equation}
\end{lemma}

\begin{proof}
	Write
	\begin{equation}
		\frac{A_i^{(\kappa)}(\nu)}{\kappa}
		=
		\inf_{p,q}
		\left\{
		\frac{1}{\kappa}W_2^2(p,q)
		+\bar\tau_0\KL(p\mid\mu_i)
		+\bar\tau_1\KL(q\mid\nu)
		\right\}.
	\end{equation}
	For the upper bound, choose \(p=q=\zeta_{\mu_i,\nu}\), where \(\zeta_{\mu_i,\nu}\) is the weighted geometric mean from Lemma~\ref{lem:weightedKL}.  Then the transport term vanishes and
	\begin{equation}
		\frac{A_i^{(\kappa)}(\nu)}{\kappa}
		\leq
		-\bar s\log\rhoCher(\mu_i,\nu).
		\label{eq:smallupper}
	\end{equation}
	
	For the lower bound, let \((p_\kappa,q_\kappa)\) be minimizers.  The right-hand side of \eqref{eq:smallupper} is uniformly bounded on \(\mathcal K\), so each nonnegative term is bounded separately:
	\[
	\bar\tau_0\KL(p_\kappa\mid\mu_i)\le C,
	\qquad
	\bar\tau_1\KL(q_\kappa\mid\nu)\le C,
	\qquad
	W_2^2(p_\kappa,q_\kappa)\le C\kappa.
	\]
	These bounds keep both KL terms uniformly bounded and force the transport discrepancy to zero.  By the Gaussian endpoint reduction, the endpoint minimizers are Gaussian.  Because \(\mathcal K\Subset\mathcal G_d^1\), there are common constants \(R<\infty\) and \(0<\underline\sigma\le\overline\sigma<\infty\) such that every \(\nu=\N(m,\Sigma)\in\mathcal K\) satisfies \(\norm{m}\le R\) and \(\underline\sigma I\preceq\Sigma\preceq\overline\sigma I\).  Thus \Cref{lem:KLsublevel} applies to the target-side bounds with constants independent of the varying reference \(\nu\in\mathcal K\); the fixed source \(\mu_i\) is the corresponding special case for \(p_\kappa\).  Hence both optimizer families lie in one common compact set of mean--covariance parameters.  Every sequence \(\kappa_j\downarrow0\), \(\nu_j\in\mathcal K\) therefore has a subsequence with
	\[
	\nu_j\to\bar\nu,
	\qquad
	p_{\kappa_j}\to\bar p,
	\qquad
	q_{\kappa_j}\to\bar q
	\]
	and \(W_2(p_{\kappa_j},q_{\kappa_j})\to0\) forces \(\bar p=\bar q=\bar\zeta\).  All Gaussian parameters remain in a common compact subset of the positive-definite domain, so joint continuity of the Gaussian KL formula allows passage to the limits
	\[
	\KL(p_{\kappa_j}\mid\mu_i)\to\KL(\bar\zeta\mid\mu_i),
	\qquad
	\KL(q_{\kappa_j}\mid\nu_j)\to\KL(\bar\zeta\mid\bar\nu).
	\]
	Dropping the nonnegative transport term leaves the lower bound
	\[
		\liminf_{j\to\infty}\frac{A_i^{(\kappa_j)}(\nu_j)}{\kappa_j}\geq
		\bar\tau_0\KL(\bar\zeta\mid\mu_i)+\bar\tau_1\KL(\bar\zeta\mid\bar\nu)\geq
		-\bar s\log\rhoCher(\mu_i,\bar\nu).
	\]
	To prove uniformity, assume by contradiction that \eqref{eq:smallA} fails.  Then there exist \(\varepsilon_0>0\), a sequence \(\kappa_j\downarrow0\), and \(\nu_j\in\mathcal K\) such that
	\[
	\left|
	\frac{A_i^{(\kappa_j)}(\nu_j)}{\kappa_j\bar s}
	+\log\rhoCher(\mu_i,\nu_j)
	\right|\geq\varepsilon_0
	\qquad\text{for every }j.
	\]
	By compactness of \(\mathcal K\), pass to a subsequence such that \(\nu_j\to\bar\nu\).  Optimizer compactness then yields a further subsequence with \(p_{\kappa_j}\to\bar\zeta\) and \(q_{\kappa_j}\to\bar\zeta\).  The lower bound above and the universal upper bound \eqref{eq:smallupper}, together with continuity of \(\rhoCher\), force
	\[
	\frac{A_i^{(\kappa_j)}(\nu_j)}{\kappa_j\bar s}
	+\log\rhoCher(\mu_i,\nu_j)\longrightarrow0,
	\]
	contradicting the defining lower bound \(\varepsilon_0\).  This proves the uniform convergence in \eqref{eq:smallA}.
\end{proof}

Substituting the endpoint limit into the Gibbs sum and exact mass profile gives the limit of the barycenter functional on compact sets.

\begin{theorem}\label{thm:smalllimit}
	Define
	\begin{equation}
		\Zfun_0(\nu)
		=
		\sum_{i=1}^n\lambda_i a_i^\theta\rhoCher(\mu_i,\nu).
	\end{equation}
	Then, uniformly on every compact \(\mathcal K\Subset\mathcal G_d^1\),
	\begin{equation}
		\Zfun_\kappa(\nu)\longrightarrow\Zfun_0(\nu),
		\qquad
		b_\kappa(\nu)\longrightarrow\Zfun_0(\nu)^{1/\theta},
		\label{eq:smallZmass}
	\end{equation}
	and
	\begin{equation}
		\frac{1}{\kappa}\Phi_\kappa(\nu)
		\longrightarrow
		\bar\tau_0
		\left[
		\bar a-\Zfun_0(\nu)^{1/\theta}
		\right].
		\label{eq:smallprofilelimit}
	\end{equation}
\end{theorem}

\begin{proof}
	\Cref{lem:smallendpoint} shows that the normalized Gibbs factor converges uniformly as
	\[
	e^{-A_i^{(\kappa)}(\nu)/(\kappa\bar s)}
	\longrightarrow
	\rhoCher(\mu_i,\nu)
	\]
	uniformly on \(\mathcal K\).  The finite sum defining \(\Zfun_\kappa\) converges uniformly to \(\Zfun_0\), which is the first limit in \eqref{eq:smallZmass}.  On \(\mathcal K\), \(\Zfun_\kappa\) and \(\Zfun_0\) lie in a common compact subinterval of \((0,\infty)\), so continuity of \(x^{1/\theta}\) transfers the uniform convergence to the profiled mass.  The profile has the exact form
	\[
	\Phi_\kappa(\nu)
	=
	\kappa\bar\tau_0
	\left[
	\bar a-\Zfun_\kappa(\nu)^{1/\theta}
		\right].
	\]
	Substituting the mass limit into the exact profile proves \eqref{eq:smallprofilelimit}.
\end{proof}

To pass from the compact-set limit to global minimizers, we next establish localization uniform in the small-penalty scale.

\begin{lemma}[Small-penalty localization]\label{lem:smalllocalization}
	Let
	\[
	\nu_\kappa\in
	\operatorname*{argmin}_{\nu\in\mathcal G_d^1}\Phi_\kappa(\nu)
	\]
	for \(\kappa\downarrow0\).  There exist \(\kappa_0>0\) and constants \(R<\infty\), \(0<\underline\sigma\le\Lambda<\infty\), independent of the choice of minimizer and of \(0<\kappa\le\kappa_0\), such that
	\begin{equation}
		\norm{m_\kappa}\le R,
		\qquad
		\underline\sigma I\preceq\Sigma_\kappa\preceq\Lambda I.
		\label{eq:smallgloballocalization}
	\end{equation}
\end{lemma}

\begin{proof}
	Fix \(\nu_{\rm ref}\in\mathcal G_d^1\).  By \Cref{thm:smalllimit} on the singleton \(\{\nu_{\rm ref}\}\),
	\[
	\Zfun_\kappa(\nu_{\rm ref})\longrightarrow
	\Zfun_0(\nu_{\rm ref})>0.
	\]
	For all sufficiently small \(\kappa\), global optimality implies \(\Zfun_\kappa(\nu_\kappa)\ge c>0\).  Since \(\Zfun_\kappa\) is a finite positive sum, for each \(\kappa\) there is an index \(i=i(\kappa)\) such that
	\[
	\lambda_i a_i^\theta
	\exp\!\left[-\frac{A_i^{(\kappa)}(\nu_\kappa)}{\kappa\bar s}\right]
	\ge \frac{c}{n}.
	\]
	The input family is finite, so this implies the uniform bound
	\begin{equation}
		A_i^{(\kappa)}(\nu_\kappa)\le C\kappa
	\end{equation}
	for the selected index.  Let \(p_\kappa,q_\kappa\) be its adjusted Gaussian marginals.  The three nonnegative terms in the rescaled endpoint objective then satisfy
	\begin{equation}
		W_2^2(p_\kappa,q_\kappa)\le C\kappa,
		\qquad
		\KL(p_\kappa\mid\mu_i)\le \frac{C}{\bar\tau_0},
		\qquad
		\KL(q_\kappa\mid\nu_\kappa)\le \frac{C}{\bar\tau_1}.
		\label{eq:smallthreeglobal}
	\end{equation}
	The KL bound for \(p_\kappa\), together with finiteness of the possible selected inputs, confines \(p_\kappa\) to a common compact subset of \(\mathcal G_d^1\).  Meanwhile, \(W_2(p_\kappa,q_\kappa)\to0\), so \(q_\kappa\) inherits uniform mean bounds and an upper covariance bound from the Bures trace estimate.  To obtain the missing lower spectral bound, suppose that a sequence \(\kappa_j\downarrow0\) satisfies \(\lambda_{\min}(Q_{\kappa_j})\to0\).  Compactness of the adjusted sources and the preceding upper bounds then allow a further subsequence with
	\[
	P_{\kappa_j}\to\overline P\succ0,
	\qquad
	Q_{\kappa_j}\to\overline Q\succeq0.
	\]
	By continuity of the matrix square root, the Bures formula \eqref{eq:Burescovdistance} extends continuously to \(\mathbb{S}_{+}^d\times\mathbb{S}_{+}^d\).  The convergence \(W_2(p_{\kappa_j},q_{\kappa_j})\to0\) then forces \(d_{\BW}(\overline P,\overline Q)=0\), so \(\overline Q=\overline P\succ0\), contradicting \(\lambda_{\min}(Q_{\kappa_j})\to0\).  Hence \(q_\kappa\) remains in a common compact subset of \(\mathcal G_d^1\).
	
	It remains to control the reference \(\nu_\kappa=\N(m_\kappa,\Sigma_\kappa)\) in the last KL bound in \eqref{eq:smallthreeglobal}.  Write \(q_\kappa=\N(v_\kappa,Q_\kappa)\) and \(B_\kappa:=\Sigma_\kappa^{-1/2}Q_\kappa\Sigma_\kappa^{-1/2}\).
	The Gaussian KL formula reads
	\[
	2\KL(q_\kappa\mid\nu_\kappa)
	=\norm{\Sigma_\kappa^{-1/2}(v_\kappa-m_\kappa)}^2
	+\tr B_\kappa-\log\det B_\kappa-d.
	\]
	The coercivity of \(t-\log t-1\) confines the spectrum to a fixed interval: there exist \(0<\ell\le u<\infty\) such that
	\[
	\ell I\preceq B_\kappa\preceq uI.
	\]
	Equivalently,
	\[
	\ell\Sigma_\kappa\preceq Q_\kappa\preceq u\Sigma_\kappa,
	\qquad\text{so}\qquad
	u^{-1}Q_\kappa\preceq\Sigma_\kappa\preceq\ell^{-1}Q_\kappa.
	\]
	The common two-sided spectral bounds for \(Q_\kappa\) therefore transfer directly to \(\Sigma_\kappa\).  The mean term then bounds \(m_\kappa\) because \(v_\kappa\) is bounded.  This proves \eqref{eq:smallgloballocalization}.
\end{proof}

Localization and uniform convergence on compact sets can now be combined to pass to global minimizers.  For \(\mathcal A\subset\mathcal G_d^1\), write \(d_{\rm par}(\nu,\mathcal A):=\inf_{\eta\in\mathcal A}d_{\rm par}(\nu,\eta)\).

\begin{theorem}[Small-penalty barycenter limit]\label{thm:smallglobal}
	Every positive superlevel of the limiting affinity functional \(\Zfun_0\) is compact in mean--covariance coordinates.  Consequently, \(\Zfun_0\) attains its global maximum on \(\mathcal G_d^1\), and its maximizer set is nonempty and compact.  Set
	\begin{equation}
		\mathcal A_0
		:=\operatorname*{argmax}_{\nu\in\mathcal G_d^1}\Zfun_0(\nu).
	\end{equation}
	Then the directed Hausdorff excess from the global minimizer set to \(\mathcal A_0\) vanishes:
	\begin{equation}
		\sup_{\nu\in\operatorname*{argmin}_{\eta\in\mathcal G_d^1}\Phi_\kappa(\eta)}
		d_{\rm par}(\nu,\mathcal A_0)
		\longrightarrow0
		\qquad(\kappa\downarrow0).
		\label{eq:smallsetwise}
	\end{equation}
	In particular, every cluster point of any selection of global minimizers belongs to \(\mathcal A_0\); if \(\mathcal A_0=\{\nu_0\}\), then every selection converges to \(\nu_0\).  Moreover,
	\begin{equation}
		\frac1\kappa
		\min_{\nu\in\mathcal G_d^1}\Phi_\kappa(\nu)
		\longrightarrow
		\bar\tau_0\left[
		\bar a-
		\left(\max_{\nu\in\mathcal G_d^1}\Zfun_0(\nu)\right)^{1/\theta}
		\right].
		\label{eq:smallglobalvalue}
	\end{equation}
	For every such selection of global minimizers, the profiled masses satisfy
	\begin{equation}
		b_\kappa(\nu_\kappa)
		\longrightarrow
		\left(\max_{\nu\in\mathcal G_d^1}\Zfun_0(\nu)\right)^{1/\theta}.
		\label{eq:smallglobalmass}
	\end{equation}
\end{theorem}

\begin{proof}
	\emph{Step 1: compactness of the limiting maximizer set.}
	Fix \(\mu_i=\N(m_i,\Sigma_i)\), write \(\nu=\N(m,\Sigma)\), and set \(B_i=\Sigma_i^{-1/2}\Sigma\Sigma_i^{-1/2}\).  By \eqref{eq:ChernoffGaussianGeneral}, the determinant prefactor of \(\rhoCher(\mu_i,\nu)\) is
	\[
	\prod_{j=1}^d\left[\frac{\lambda_j(B_i)^\theta}{(1-\theta)+\theta\lambda_j(B_i)}\right]^{1/2}.
	\]
	Weighted AM--GM shows that every factor is at most one, while \(g(t)=t^\theta/((1-\theta)+\theta t)\) tends to zero as \(t\downarrow0\) or \(t\to\infty\).  Thus the lower bound \(\rhoCher(\mu_i,\nu)\ge\delta>0\) confines the spectrum of \(B_i\) to a compact subinterval of \((0,\infty)\), which also gives two-sided spectral bounds for \(\Sigma\).  On this spectral box, \(D_\theta=(1-\theta)\Sigma_i+\theta\Sigma\) has a uniform upper bound, and the exponential factor in \eqref{eq:ChernoffGaussianGeneral} bounds \(\norm{m-m_i}\).  This proves that \(\{\nu:\rhoCher(\mu_i,\nu)\ge\delta\}\) is compact.
	
	If \(\Zfun_0(\nu)\ge\eta>0\), then for some \(i\),
	\[
	\rhoCher(\mu_i,\nu)\ge\frac{\eta}{n\lambda_i a_i^\theta}.
	\]
	Every positive superlevel of \(\Zfun_0\) is contained in a finite union of such compact sets and is closed by continuity, so it is compact.  Since \(\Zfun_0\) is positive somewhere, every maximizing sequence eventually lies in a fixed positive superlevel and admits a convergent subsequence.  Continuity shows that the limit is a maximizer; \(\mathcal A_0\) is closed in the same compact superlevel.
	
	\emph{Step 2: convergence of global minimizers.}
	By \Cref{lem:smalllocalization}, all sufficiently small-penalty global minimizers lie in one compact Gaussian set \(\mathcal K_*\).  Let \(\kappa_j\downarrow0\) and choose arbitrary global minimizers \(\nu_j\).  A subsequence converges to some \(\nu_0\in\mathcal K_*\).  For any fixed \(\nu\in\mathcal G_d^1\), global maximality of \(\Zfun_{\kappa_j}\), combined with \Cref{thm:smalllimit}, implies
	\[
	\Zfun_0(\nu_0)\ge\Zfun_0(\nu).
	\]
	Every cluster point of every selection of global minimizers belongs to \(\mathcal A_0\).  If \eqref{eq:smallsetwise} failed, there would exist \(\varepsilon>0\), \(\kappa_j\downarrow0\), and global minimizers \(\nu_j\) with \(d_{\rm par}(\nu_j,\mathcal A_0)\ge\varepsilon\).  Compactness would yield a subsequence converging to some \(\overline\nu\in\mathcal A_0\), which forces \(d_{\rm par}(\nu_j,\mathcal A_0)\to0\), contrary to the lower bound \(\varepsilon\).  Hence \eqref{eq:smallsetwise} holds.  For the remaining limits, fix any selection of global minimizers \(\nu_\kappa\).  Let \(Z_0^*:=\max_{\nu\in\mathcal G_d^1}\Zfun_0(\nu)\).
	Fix any global maximizer \(\nu_0^*\in\mathcal A_0\) and enlarge \(\mathcal K_*\), if necessary, to the compact set \(\mathcal K_*\cup\{\nu_0^*\}\).  Since all minimizing shapes lie in this enlarged compact set, \Cref{thm:smalllimit} establishes uniform convergence \(\Zfun_\kappa\to\Zfun_0\) there.  Global maximality of \(\nu_\kappa\) then implies
	\[
	\Zfun_0(\nu_\kappa)
	\ge \Zfun_\kappa(\nu_\kappa)-o(1)
	\ge \Zfun_\kappa(\nu_0^*)-o(1)
	=Z_0^*+o(1).
	\]
	Since \(\Zfun_0(\nu_\kappa)\le Z_0^*\), it follows that \(\Zfun_0(\nu_\kappa)\to Z_0^*\).  The uniform mass limit on \(\mathcal K_*\) therefore establishes \eqref{eq:smallglobalmass}; the uniform profile limit establishes \eqref{eq:smallglobalvalue}.
\end{proof}

\begin{corollary}\label{cor:bhattacharyya}
	If \(\bar\tau_0=\bar\tau_1\), then \(\theta=1/2\) and \(\rhoCher\) becomes the Bhattacharyya affinity~\cite{Bhattacharyya1943}.  For \(\mu_i=\N(m_i,\Sigma_i)\) and \(\nu=\N(m,\Sigma)\), the affinity is
	\begin{equation}
		\rho_{1/2}(\mu_i,\nu)=\frac{\det(\Sigma_i)^{1/4}\det(\Sigma)^{1/4}}
		{\det\left((\Sigma_i+\Sigma)/2\right)^{1/2}}
		\exp\!\left[-\frac18(m_i-m)^\top\left(\frac{\Sigma_i+\Sigma}{2}\right)^{-1}(m_i-m)\right].
		\label{eq:BhattacharyyaGaussian}
	\end{equation}
\end{corollary}

\begin{proof}
	Completing the square in the integral of the geometric mean of the two Gaussian densities evaluates it as \eqref{eq:BhattacharyyaGaussian}.
\end{proof}

\begin{remark}
	For equal penalties, the limiting shape problem is an overlap maximization rather than a quadratic-displacement averaging problem.  The compact limiting affinity maximizer set may contain multiple points.  The directed convergence in \Cref{thm:smallglobal} is one-sided: when \(\mathcal A_0\) is non-singleton, minimizers at positive \(\kappa\) need not approach every element of \(\mathcal A_0\).  If the affinity maximizer is unique, the conclusion reduces to full-sequence convergence.
\end{remark}

\FloatBarrier
\section{Numerical illustrations}

Synthetic Gaussian experiments illustrate the descent and noncollapse properties of the MM iteration, the attenuation and local-contraction mechanisms, and both penalty limits.  The plots distinguish computed series from fitted or asymptotic guides.

\subsection{Setup and MM descent}

The baseline uses \(n=8\), \(d=3\), equal external weights, masses sampled log-uniformly from \([0.55,1.8]\), means from \(\N(0,1.15^2I_3)\), and covariance eigenvalues log-uniformly from \([0.4,2.5]\) with random orthogonal eigenbases.  To probe noncollapse, we initialize with \(m^0=0\) and \(\Sigma^0=\operatorname{diag}(10^{-5},0.8,3)\).
Runs are accepted when the complete residual \eqref{eq:stoppingresidual} is at most \(3\times10^{-6}\).  For each nontrivial step we test \Cref{thm:descent} through
\[
C_k
=
\frac{\Phi(\nu^k)-\Phi(\nu^{k+1})}
{\tau_1b^k\KL(\nu^{k+1}\mid\nu^k)},
\]
which satisfies \(C_k\ge1\) by \Cref{thm:descent}.

\begin{table}[htbp]
	\centering
	\small
	\setlength{\tabcolsep}{3pt}
	\caption{MM descent and covariance noncollapse from the near-boundary initialization.  Here, \(K\) denotes the terminal accepted iteration, and \(\operatorname{res}_{K-1}\) denotes the final complete residual.}
	\label{tab:descent}
	\begin{tabular}{@{}crrrrr@{}}
		\toprule
		\((\tau_0,\tau_1)\) & \(K\) & \(\min_{k<K}C_k\) & \(\KL(\nu^K\mid\nu^{K-1})\) & \(\operatorname{res}_{K-1}\) & \([\lambda_{\min},\lambda_{\max}](\Sigma^K)\) \\
		\midrule
		\((1,1)\) & 596 & 1.5563 & \(7.475\times10^{-12}\) & \(2.73\times10^{-6}\) & \([1.007,2.284]\) \\
		\((1,4)\) & 2357 & 1.7603 & \(8.435\times10^{-12}\) & \(2.90\times10^{-6}\) & \([0.869,1.556]\) \\
		\((4,1)\) & 422 & 1.5183 & \(6.735\times10^{-12}\) & \(2.60\times10^{-6}\) & \([1.106,2.857]\) \\
		\bottomrule
	\end{tabular}
\end{table}

The terminal diagnostics in \Cref{tab:descent} show both that \(\min_{k<K}C_k>1\) for all three penalty pairs and that the smallest covariance eigenvalue rises from \(10^{-5}\) initially to an order-one value at termination.  These observations are consistent with \Cref{thm:descent,cor:trajectorynoncollapse}.  The trajectories in \Cref{fig:descent} display the corresponding objective gaps and descent certificates.  The unequal iteration counts for \((1,4)\) and \((4,1)\) also reflect the directional local linearization: on the same instance, the more asymmetric pairs \((1,8)\) and \((8,1)\) have finite-difference contraction estimates \(\AsymSlowRate\) and \(\AsymFastRate\), respectively.  These factors describe the eventual contraction rather than the full near-boundary transient.

\begin{figure}[tbp]
	\centering
	\begin{minipage}[t]{0.48\textwidth}
		\centering
		\includegraphics[width=\linewidth]{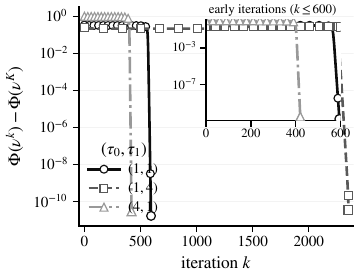}\\[-0.3em]
		\small (a) Profiled objective gap.
	\end{minipage}\hfill
	\begin{minipage}[t]{0.48\textwidth}
		\centering
		\includegraphics[width=\linewidth]{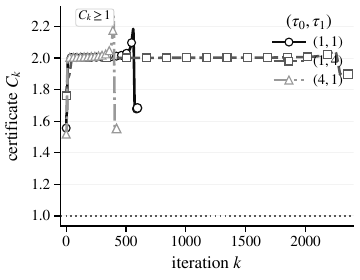}\\[-0.3em]
		\small (b) Descent certificate \(C_k\).
	\end{minipage}
	\caption{MM descent from the near-boundary initialization: profiled objective gaps (a) and certificates \(C_k\) (b), with the theoretical threshold \(C_k=1\) shown as a dotted line.}
	\label{fig:descent}
\end{figure}

\subsection{Remote-input attenuation and local contraction}

To probe \Cref{prop:redescending}, we hold the fixed point of seven inliers constant and move an eighth mean along \(Re_1\), leaving its mass and covariance unchanged.  For \(R\ge3\), the fit of \(\log\omega_o\) against \(R^2\) has slope \(-0.094\) and coefficient of determination \(R_{\rm fit}^2=0.999989\), consistent with the Gaussian-tail attenuation displayed in \Cref{fig:localmechanisms}(a).

For \Cref{thm:localrate}, we form a centered finite-difference Jacobian \(J_h\) at the fixed point and the similarity transform \(K_h=B_*^{1/2}J_hB_*^{-1/2}\).  With symmetric penalties \((1,1)\) and \(h=2\times10^{-6}\), the estimate \(\widehat\rho_*:=\|K_h\|_2=0.688\) agrees with the median observed ratio \(0.681\) over iterations 10--25, as shown in \Cref{fig:localmechanisms}(b).  It is stable to six decimals over the three scales in \Cref{tab:numericalvalidation}, with relative \(B_*\)-symmetry defect at most \(1.20\times10^{-8}\); the asymmetric values \(\AsymSlowRate\) and \(\AsymFastRate\) for \((1,8)\) and \((8,1)\) confirm the directional dependence.

\begin{figure}[tbp]
	\centering
	\begin{minipage}[t]{0.48\textwidth}
		\centering
		\includegraphics[width=\linewidth]{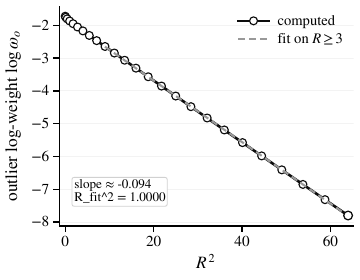}\\[-0.3em]
		\small (a) Remote-input attenuation.
	\end{minipage}\hfill
	\begin{minipage}[t]{0.48\textwidth}
		\centering
		\includegraphics[width=\linewidth]{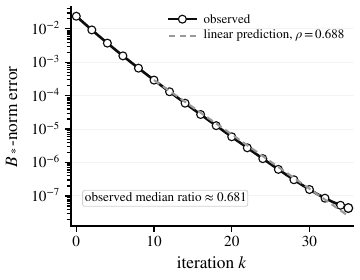}\\[-0.3em]
		\small (b) Local error contraction.
	\end{minipage}
	\caption{Remote-input attenuation (a) and local error contraction (b); dashed lines show the fitted tail law and linearized rate, respectively.}
	\label{fig:localmechanisms}
\end{figure}

\subsection{Large- and small-penalty asymptotics}

For large penalties we continue the stationary branch with \(\tau_0=\tau_1=\kappa\) and compare it with \(\nu_{\BW}\).  Both errors in \Cref{fig:penaltydiagnostics}(a) decrease with \(\kappa\), and the shape error is \(7.005\times10^{-2}\) at \(\kappa=128\); a separate grid-refinement calculation below probes global minimizers.

For small penalties we use a fixed panel \(\mathcal K_{\rm panel}\) of six interior Gaussian candidates and the ratios \((\bar\tau_0,\bar\tau_1)=(1,1),(1,3),(3,1)\), reporting
\[
\max_{\nu\in\mathcal K_{\rm panel}}\max_i
\left|
\frac{A_i^{(\kappa)}(\nu)}{\kappa(\bar\tau_0+\bar\tau_1)}
+\log\rho_{\theta}(\mu_i,\nu)
\right|,
\qquad
\theta=\frac{\bar\tau_0}{\bar\tau_0+\bar\tau_1}.
\]
The maximum error decreases for all three ratios as \(\kappa\downarrow0\), qualitatively illustrating \Cref{lem:smallendpoint}; see \Cref{fig:penaltydiagnostics}(b).

\begin{figure}[tbp]
	\centering
	\begin{minipage}[t]{0.48\textwidth}
		\centering
		\includegraphics[width=\linewidth]{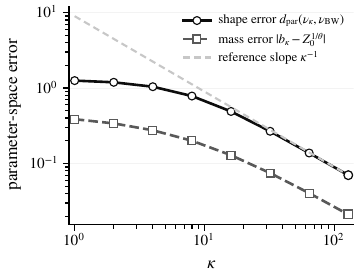}\\[-0.3em]
		\small (a) Large-penalty continuation.
	\end{minipage}\hfill
	\begin{minipage}[t]{0.48\textwidth}
		\centering
		\includegraphics[width=\linewidth]{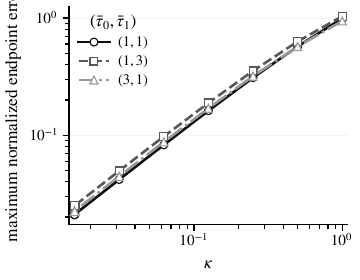}\\[-0.3em]
		\small (b) Small-penalty endpoint limit.
	\end{minipage}
	\caption{Penalty-limit diagnostics: large-penalty continuation errors (a) and small-penalty normalized endpoint errors (b).}
	\label{fig:penaltydiagnostics}
\end{figure}

A separate \(d=1\), three-input example with \((\bar\tau_0,\bar\tau_1)=(1,2)\) tests the refined limits in \(x=(m,\log\sigma^2)\) coordinates.  At large penalty, \(\kappa(x_\kappa-x_{\BW})\) approaches \((0.1061,2.7958)\), with Euclidean remainder \(3.13\times10^{-2}\) at \(\kappa=256\).  At small penalty, direct minimization gives distance \(1.25\times10^{-3}\) from a limiting Chernoff maximizer at \(\kappa=1/64\); see \Cref{fig:refinedlimits}.

\begin{figure}[tbp]
	\centering
	\begin{minipage}[t]{0.48\textwidth}
		\centering
		\includegraphics[width=\linewidth]{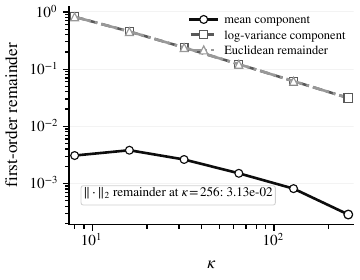}\\[-0.3em]
		\small (a) Large-penalty first-order remainder.
	\end{minipage}\hfill
	\begin{minipage}[t]{0.48\textwidth}
		\centering
		\includegraphics[width=\linewidth]{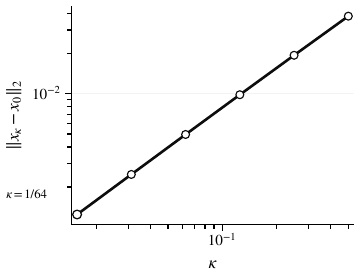}\\[-0.3em]
		\small (b) Small-penalty barycenter distance.
	\end{minipage}
	\caption{Refined one-dimensional limits: first-order large-penalty remainders (a) and distance to a limiting Chernoff maximizer (b).}
	\label{fig:refinedlimits}
\end{figure}

\subsection{Robustness and implementation checks}

For the baseline data with \((\tau_0,\tau_1)=(1,2)\), the complete residual in \eqref{eq:stoppingresidual} is at most \(10^{-5}\) for all 50 nondegenerate starts.  Relative to the best point, the largest parameter distance is \(1.42\times10^{-4}\) and the objective spread is \(1.63\times10^{-10}\).  No distinct attracting stationary point was observed, although global uniqueness at intermediate penalties remains open.

Table~\ref{tab:numericalvalidation} checks finite-difference stability of the local rate and, in a four-input \(d=1\) problem, a \(181\times181\) grid search over \((m,\log\sigma^2)\in[-4,5]\times[-3,2.5]\), followed by a \(161\times161\) local refinement around the best grid point.  Denote by \(\widehat\nu_\kappa\) the refined numerical minimizer obtained by this procedure.  Its distance \(d_{\rm par}(\widehat\nu_\kappa,\nu_{\BW})\) decreases from \(\LargeGridDistA\) at \(\kappa=4\) to \(\LargeGridDistC\) at \(\kappa=64\), directly probing \Cref{thm:largeglobal}.

\FloatBarrier
\begin{table}[htbp]
	\centering
	\small
	\caption{Finite-difference sensitivity and large-penalty minimizer checks.}
	\label{tab:numericalvalidation}
	\begin{minipage}[t]{0.58\textwidth}
		\centering
		\begin{tabular}{@{}rrr@{}}
			\toprule
			\multicolumn{3}{c}{Local-rate finite-difference sensitivity}\\
			\cmidrule(lr){1-3}
			\(h\) & \(\widehat\rho_*\) & \shortstack{relative \(B_*\)-symmetry\\ defect}\\
			\midrule
			\(10^{-5}\) & 0.687766 & \(1.79\times10^{-9}\)\\
			\(2\times10^{-6}\) & 0.687766 & \(3.83\times10^{-9}\)\\
			\(5\times10^{-7}\) & 0.687766 & \(1.20\times10^{-8}\)\\
			\bottomrule
		\end{tabular}
	\end{minipage}\hfill
	\begin{minipage}[t]{0.40\textwidth}
		\centering
		\begin{tabular}{@{}rr@{}}
			\toprule
			\multicolumn{2}{c}{Large-penalty grid comparison}\\
			\cmidrule(lr){1-2}
			\(\kappa\) & \(d_{\rm par}(\widehat\nu_\kappa,\nu_{\BW})\)\\
			\midrule
			4 & \LargeGridDistA\\
			16 & 0.1886\\
			64 & \LargeGridDistC\\
			\bottomrule
		\end{tabular}
	\end{minipage}
\end{table}

Finally, the 34 stress configurations span multiple random seeds, dimensions, covariance condition numbers, and asymmetric penalty pairs.  The complete residual in \eqref{eq:stoppingresidual} is at most \(10^{-5}\) for every trajectory, and the smallest terminal covariance eigenvalue is \(0.520\); detailed timings and solver diagnostics are archived with the code.

\section{Conclusion and perspectives}

Profiling the barycenter mass exactly turns the finite Gaussian KL-UOT barycenter problem into a smooth shape problem with endogenous Gibbs weights.  This formulation gives global attainment and explicit stationary moment equations, and it leads to a reverse-KL MM iteration.  Along every orbit, the covariance remains uniformly positive definite; combined with sufficient decrease, relative error, and the Kurdyka--\L{}ojasiewicz property, this yields finite-length convergence of the full sequence from any nondegenerate initialization to a stationary fixed point.  At the infinitesimal level, the diagonal second variation induces a BW--FR parallel-sum tensor whose finite-mass extension admits a radial cone representation.

Under common penalty scaling, global minimizers converge in the large-penalty regime to a Gaussian Wasserstein barycenter with effective weights, and the minimizer is unique and analytic in \(\kappa^{-1}\) once \(\kappa\) is sufficiently large.  At the opposite scale, global minimizers approach the compact maximizer set of a weighted Chernoff affinity functional.  The intermediate regime is less understood: neither limiting mechanism governs the problem directly, and even uniqueness of the minimizer remains open.  Extending the analysis from nondegenerate Gaussians to unrestricted finite measures would require new compactness and noncollapse arguments, while accommodating singular covariance limits would call for a suitable completion of the local geometry.  The expansions and cone structure obtained here also suggest higher-order corrections and minimizing-movement constructions adapted to this geometry.  Resolving these issues would clarify how much of the finite-dimensional Gaussian theory persists in broader measure spaces.

\appendix
\section{Verification of the pairwise Gaussian endpoint}
\label{app:pairwise}

We give a self-contained verification of the pairwise facts used in \Cref{thm:pairwiseinterface}.  The argument has three steps: reduction of the original probability problem to Gaussian moments, construction and uniqueness of the finite-dimensional minimizer, and recovery of the equality and analyticity statements.

\emph{Step 1: reduction without loss of minimizers.}
Let \((p,q)\in\mathcal P(\R^d)^2\) have finite objective in \eqref{eq:Adef}.  The argument following that definition, based on the entropy inequality, yields \(p,q\in\mathcal P_2(\R^d)\), so the moment-matched Gaussians \(G[p]\) and \(G[q]\) are well defined.  The information-projection identities established there give
\[
\KL(p\mid\mu_0)
=\KL(G[p]\mid\mu_0)+\KL(p\mid G[p]),
\qquad
\KL(q\mid\mu_1)
=\KL(G[q]\mid\mu_1)+\KL(q\mid G[q]),
\]
and Gelbrich's inequality gives \(W_2^2(p,q)\ge W_2^2(G[p],G[q])\).  Consequently, moment matching cannot increase any of the three terms in \eqref{eq:Adef}.  It follows that the unrestricted infimum is bounded below by the infimum over Gaussian pairs; the reverse inequality is immediate because Gaussian pairs are admissible.  Thus the two infima agree, and the Gaussian minimizer constructed below also minimizes the original probability problem.

\emph{Step 2: finite-dimensional existence and uniqueness.}
Write \(p=\N(u,P)\) and \(q=\N(v,Q)\) with \(P,Q\in\Spp^d\).  Up to constants independent of \((u,v,P,Q)\), the objective separates into mean and covariance blocks.  The mean block is strictly convex, and its first-order equations give exactly \eqref{eq:hstar}--\eqref{eq:uvstar}.

For the covariance variables, write
\[
\mathcal J_{\rm cov}(P,Q)
=d_{\BW}^2(P,Q)
+\frac{\tau_0}{2}\bigl[\tr(K_0P)-\log\det P\bigr]
+\frac{\tau_1}{2}\bigl[\tr(K_1Q)-\log\det Q\bigr],
\]
again up to an additive constant.  The squared Bures--Wasserstein distance is jointly convex in \((P,Q)\)~\cite{BhatiaJainLim2019}, while \(-\log\det\) is strictly convex on \(\Spp^d\), so \(\mathcal J_{\rm cov}\) is strictly convex.  To construct its stationary point, let \(L\succ0\) be the Gaussian Wasserstein map from \(P\) to \(Q\), so \(Q=LPL\).  The standard Bures differentials \(\nabla_Pd_{\BW}^2=I-L\) and \(\nabla_Qd_{\BW}^2=I-L^{-1}\) give the stationarity equations
\begin{equation}
	P^{-1}=K_0+r_0(I-L),
	\qquad
	Q^{-1}=K_1+r_1(I-L^{-1}).
	\label{eq:app-covstationarity}
\end{equation}
Since \(Q^{-1}=L^{-1}P^{-1}L^{-1}\), elimination yields the Riccati equation \(LC_1L-\delta L=C_0\).  With \(S=C_1^{1/2}LC_1^{1/2}\) and \(B=C_1^{1/2}C_0C_1^{1/2}\succ0\), it becomes \(S^2-\delta S=B\).  Any positive-definite solution commutes with \(B\), and spectral calculus gives \(S=S_*\) from \eqref{eq:Sstar} together with \(L=L_*\).

The corresponding covariance candidate is \((P_*,Q_*)\) from \eqref{eq:LPQ}.  Its positivity follows from
\[
(I-L_*)^2+\frac1{r_0}K_0+\frac1{r_1}L_*K_1L_*
=\frac{r_0+r_1}{r_0r_1}\bigl[K_0+r_0(I-L_*)\bigr].
\]
Since \(L_*\) is symmetric, the first term on the left is positive semidefinite and the remaining two are positive definite.  It follows that \(P_*^{-1}=K_0+r_0(I-L_*)\succ0\) and \(Q_*=L_*P_*L_*\succ0\).  The definition of \(P_*\) gives the first identity in \eqref{eq:app-covstationarity}, while the Riccati equation gives
\[
P_*^{-1}=L_*\bigl[K_1+r_1(I-L_*^{-1})\bigr]L_*,
\]
which is equivalent to the second identity because \(Q_*^{-1}=L_*^{-1}P_*^{-1}L_*^{-1}\).  The pair \((P_*,Q_*)\) is therefore a stationary point of \(\mathcal J_{\rm cov}\), and strict convexity makes it the unique global covariance minimizer.  Together with the unique mean minimizer, this constructs the unique Gaussian pair \((p_*,q_*)\) in \eqref{eq:Adef}.  By Step~1, this pair also attains the unrestricted probability infimum.

\emph{Step 3: equality conditions, optimal coupling, and analyticity.}
Let \((p,q)\) be any minimizer of the original probability problem.  Replacing it by \((G[p],G[q])\) cannot increase any objective term by Step~1.  Since \((p,q)\) already attains the common minimum, the total decrease is zero.  In particular, both nonnegative information-projection remainders vanish:
\[
\KL(p\mid G[p])=0,
\qquad
\KL(q\mid G[q])=0.
\]
Hence \(p=G[p]\) and \(q=G[q]\).  The finite-dimensional uniqueness from Step~2 then forces \((p,q)=(p_*,q_*)\), so every minimizer of the original problem is the same Gaussian pair.

Finally, \(L_*\succ0\) and \(Q_*=L_*P_*L_*\), so the standard Gaussian transport formula \eqref{eq:GaussianW2map} shows that \((\operatorname{Id},T_*)_\#p_*\) is the unique quadratic Wasserstein coupling of \((p_*,q_*)\).  Inversion and the principal square root are real analytic on \(\Spp^d\)~\cite{Higham2008}; the explicit formulas \eqref{eq:Sstar}--\eqref{eq:uvstar} then give the asserted joint analyticity of \(L_*,P_*,Q_*,u_*,v_*\) and \(\Ashape\).  This completes the proof of \Cref{thm:pairwiseinterface}.

\section*{Acknowledgements}
This research is supported by National Key R \& D Program of China (2024YFA1012401), the Science and Technology Commission of Shanghai Municipality (23JC1400501), and Natural Science Foundation of China (12241103).

\section*{Data Availability Statement}
Data and code are available on Zenodo: \href{https://doi.org/10.5281/zenodo.22256289}{10.5281/zenodo.22256289}.

\end{document}